\documentclass[10pt]{article}
\usepackage{amsthm}
\usepackage{amsmath}
\usepackage{amssymb}
\usepackage{makeidx}

\theoremstyle{definition}
\newtheorem{definition}{Definition}[section]
\newtheorem{notation}[definition]{Notation}
\newtheorem{example}[definition]{Example}
\newtheorem{openproblem}[definition]{Open problem}

\newtheorem{remark}[definition]{Remark}
\newtheorem{note}[definition]{Note}

\theoremstyle{plain}
\newtheorem{theorem}[definition]{Theorem}
\newtheorem{lemma}[definition]{Lemma}
\newtheorem{proposition}[definition]{Proposition}
\newtheorem{corollary}[definition]{Corollary}

\newcommand{\beq}{\begin{equation}}

\newcommand{\eeq}{\end{equation}}

\newcommand{\bdfn}{\begin{definition}}

\newcommand{\edfn}{\end{definition}}

\newcommand{\bthm}{\begin{theorem}}

\newcommand{\ethm}{\end{theorem}}

\newcommand{\bprop}{\begin{proposition}}

\newcommand{\eprop}{\end{proposition}}

\newcommand{\bcor}{\begin{corollary}}

\newcommand{\ecor}{\end{corollary}}

\newcommand{\blem}{\begin{lemma}}

\newcommand{\elem}{\end{lemma}}

\newcommand{\bex}{\begin{example}}

\newcommand{\eex}{\end{example}}

\newcommand{\bxc}{\begin{exercise}}

\newcommand{\exc}{\end{exercise}}

\newcommand{\bntn}{\begin{notation}}

\newcommand{\entn}{\end{notation}}

\newcommand{\be}{\begin{enumerate}}

\newcommand{\ee}{\end{enumerate}}

\newcommand{\bce}{\begin{center}}

\newcommand{\ece}{\end{center}}

\newcommand{\bi}{\begin{itemize}}

\newcommand{\ei}{\end{itemize}}

\newcommand{\bt}{\begin{tabular}}

\newcommand{\et}{\end{tabular}}

\newcommand{\ba}{\begin{array}} 

\newcommand{\ea}{\end{array}}

\numberwithin{equation}{section}

\def\N{{\mathbb N}}

\newcommand {\bua} {\begin{eqnarray*}}

\newcommand {\eua} {\end {eqnarray*}}

\begin{document}
\title{Congruence Boolean Lifting Property}
\author{George GEORGESCU and Claudia MURE\c SAN\thanks{Corresponding author.}\\ \footnotesize University of Bucharest\\ \footnotesize Faculty of Mathematics and Computer Science\\ \footnotesize Academiei 14, RO 010014, Bucharest, Romania\\ \footnotesize Emails: georgescu.capreni@yahoo.com; c.muresan@yahoo.com, cmuresan@fmi.unibuc.ro}
\date{\today }
\maketitle

\begin{abstract} We introduce and study the Congruence Boolean Lifting Property (CBLP) for congruence--distributive universal algebras, as well as a property related to CBLP, which we have called $(\star )$. CBLP extends the so--called Boolean Lifting Properties (BLP) from MV--algebras, BL--algebras and residuated lattices, but differs from the BLP when particularized to bounded distributive lattices. Important classes of universal algebras, such as discriminator varieties, fulfill the CBLP. The main results of the present paper include a characterization theorem for congruence--distributive algebras with CBLP and a structure theorem for semilocal arithmetical algebras with CBLP. When we particularize the CBLP to the class of residuated lattices and to that of bounded distributive lattices and we study its relation to other Boolean Lifting Properties for these algebras, interesting results concerning the image of the reticulation functor between these classes are revealed.\\ {\em 2010 Mathematics Subject Classification:} Primary: 08B10; secondary: 03C05, 06F35, 03G25, 08B05.\\ {\em Keywords:} Boolean Lifting Property; Boolean center; lattice; residuated lattice; reticulation; (congruence--distributive, congruence--permutable, arithmetical) algebra; discriminator variety.\end{abstract}

\section{Introduction}

A unital ring has the {\em Idempotent Lifting Property} ({\em ILP} or {\em LIP}) iff its idempotents can be lifted modulo every left ideal. The ILP is closely related to important classes of rings such as clean rings, exchange rings, Gelfand rings, maximal rings etc.. Several algebraic and topological characterizations of commutative unital rings with ILP are collected in \cite[Theorem $1.7$]{mcgov}.

In studying the ILP for commutative unital rings, it is essential that the set of idempotents of a commutative unital ring $R$ is a Boolean algebra (called the {\em Boolean center of $R$}). There are many algebraic structure to which one can associate a ``Boolean center``: bounded distributive lattices, $lu$--groups, MV--algebras, BL--algebras, residuated lattices etc.. For all of these algebras, a lifting condition for the elements of the Boolean center, similar to the ILP for rings, can be defined. In \cite{blpiasi}, \cite{blpdacs}, \cite{ggcm}, \cite{dcggcm}, we have defined and studied the Boolean Lifting Properties for residuated lattices and bounded distributive lattices, we have provided algebraic and topological characterizations for them, and established links between them, by means of the reticulation functor (\cite{eu3}, \cite{eu1}, \cite{eu}, \cite{eu2}, \cite{eu4}, \cite{eu5}, \cite{eu7}). We have also studied a type of generalization of the lifting properties for universal algebras in \cite{dcggcm} and \cite{eudacs}.

The issue of defining a condition type Boolean Lifting Property in the context of universal algebras naturally arises. Such a condition needs to extend the Boolean Lifting Properties in the particular cases of the structures mentioned above, thus, first of all, it needs to be defined for a class of universal algebras which includes these particular kinds of structures. This idea has started the research in the present paper. The Congruence Boolean Lifting Property (CBLP), which we study in this article, is a lattice--theoretic condition: a congruence--distributive algebra ${\cal A}$ has {\em CBLP} iff its lattice of congruences, ${\rm Con}({\cal A})$, is such that the Boolean elements of each of its principal filters are joins between the generator of that filter and Boolean elements of ${\rm Con}({\cal A})$. CBLP is the transcription in the language of universal algebra of property $(3)$ from \cite[Lemma $4$]{banasch}. In the particular cases of MV--algebras, BL--algebras and residuated lattices, CBLP is equivalent to the Boolean Lifting Property; however, this does not hold in the case of bounded distributive lattices, which all have CBLP, but not all have the Boolean Lifting Property. Consequently, the CBLP is only a partial solution to the issue mentioned above, but its study is motivated by the properties of the universal algebras with CBLP which we prove in this paper. Significant results on the Boolean Lifting Property for MV--algebras, BL--algebras and residuated lattices can be generalized to CBLP in important classes of universal algebras.

Section \ref{preliminaries} of this paper contains some well--known notions and previously known results from universal algebra that we need in the rest of the paper. The results in the following sections are new, excepting only the results cited from other papers. 

Section \ref{directproducts} is a collection of results concerning the form of congruences, and that of Boolean, finitely generated, prime and maximal congruences, in finite, as well as arbitrary direct products of congruence--distributive algebras from an equational class. These results serve as preparatives for the properties we obtain in the following sections.

In Section \ref{thecblp}, we introduce the property whose study is the aim of this paper: Congruence Boolean Lifting Property (abbreviated CBLP). We define and study the CBLP in equational classes of algebras whose lattice of congruences is distributive and has the property that its last element is compact. We prove a characterization theorem for the CBLP through algebraic and topological conditions, and identify important classes of algebras with CBLP, which include the class of bounded distributive lattices and discriminator varieties. We also study the behaviour of CBLP with respect to quotients and direct products, and its relation to a property we have called $(\star )$. Both CBLP and $(\star )$ extend properties which we have studied for residuated lattices in \cite{ggcm}, \cite{dcggcm}; many of the results in this section are inspired by results in \cite{ggcm}, \cite{dcggcm}. The preservation of the CBLP by quotients gives it an interesting behaviour in non--distributive lattices; in order to illustrate the related properties, we provide some examples.

Section \ref{cblpversusblp} is concerned with the particularizations of CBLP to the class of residuated lattices and that of bounded distributive lattices. For these algebras we have studied a property called the Boolean Lifting Property (BLP): for residuated lattices, in a restricted form in \cite{eu3}, \cite{eu4}, and, in the form which also appears in the present article, in \cite{ggcm} and \cite{dcggcm}; for bounded distributive lattices, in \cite{blpiasi}, \cite{blpdacs} and \cite{dcggcm}. It turns out that the CBLP and BLP coincide in the case of residuated lattices, but not in the case of bounded distributive lattices, where, as shown in Section \ref{thecblp}, the CBLP always holds, unlike the BLP. The reticulation functor from the category of residuated lattices to that of bounded distributive lattices preserves the BLP for filters (\cite{dcggcm}); clearly, the situation is different for the CBLP. These considerations make it easy to notice some properties concerning the image through the reticulation functor of certain classes of residuated lattices, including MV--algebras and BL--algebras, more precisely to exclude certain classes of bounded distributive lattices from these images.

In Section \ref{semilocal}, we return to the setting of universal algebra, and study the CBLP in semilocal arithmetical algebras; we find that, out of these algebras, the ones which fulfill the CBLP are exactly the finite direct products of local arithmetical algebras.

\section{Preliminaries}
\label{preliminaries}

We refer the reader to \cite{bal}, \cite{blyth}, \cite{bur}, \cite{gralgu} for a further study of the notions we recall in this section.

We shall denote by $\N $ the set of the natural numbers and by $\N ^*=\N \setminus \{0\}$. Throughout this paper, any direct product of algebras of the same type shall be considerred with the operations of that type of algebras defined componentwise. Also, throughout this paper, whenever there is no danger of confusion, for any non--empty family $(M_i)_{i\in I}$ of sets, by $\displaystyle (x_i)_{i\in I}\in \prod _{i\in I}M_i$ we shall mean: $x_i\in M_i$ for all $i\in I$. Finally, throughout this paper, a non--empty algebra will be 
called a {\em trivial algebra} iff it has only one element, and a {\em non--trivial algebra} iff it has at least two distinct elements.

Throughout this section, $\tau $ will be a universal algebra signature and ${\cal A}$ will be a $\tau $--algebra, with support set $A$. The {\em congruences} of ${\cal A}$ are the equivalences on $A$ which are compatible to the operations of ${\cal A}$. Let ${\rm Con}({\cal A})$ be the set of the congruences of ${\cal A}$, which is a complete lattice with respect to set inclusion. Notice that, for every $\phi ,\theta \in {\rm Con}({\cal A})$, $\phi \subseteq \theta $ iff $\phi $ is a refinement of $\theta $, that is all the congruence classes of $\theta $ are unions of congruence classes of $\phi $. In the complete lattice $({\rm Con}({\cal A}),\subseteq )$, the meet of any family of congruences of ${\cal A}$ is the intersection of those congruences. The join, however, does not coincide to the union; we shall use the common notation, $\vee $, for the join in ${\rm Con}({\cal A})$. Also, clearly, $\Delta _{\cal A}$ is the first element of the lattice ${\rm Con}({\cal A})$, and $\nabla _{\cal A}$ is the last element of the lattice ${\rm Con}({\cal A})$, where we have denoted by $\Delta _{\cal A}=id_A=\{(a,a)\ |\ a\in A\}$ and $\nabla _{\cal A}=A^2$. Any $\phi \in {\rm Con}({\cal A})$ such that $\phi \neq \nabla _{\cal A}$ is called a {\em proper congruence of ${\cal A}$.}

Since the intersection of any family of congruences of ${\cal A}$ is a congruence of ${\cal A}$, it follows that, for every subset $X$ of $A^2$, there exists the smallest congruence of ${\cal A}$ which includes $X$; this congruence is denoted by $Cg(X)$ and called {\em the congruence of ${\cal A}$ generated by $X$.} Whenever the algebra ${\cal A}$ needs to be specified, we shall denote $Cg(X)$ by $Cg_{\cal A}(X)$. It is obvious that the join in the lattice $({\rm Con}({\cal A}),\subseteq )$ is given by: for all $\phi ,\psi \in {\rm Con}({\cal A})$, $\phi \vee \psi =Cg(\phi \cup \psi)$. The congruences of ${\cal A}$ generated by finite subsets of $A^2$ are called {\em finitely generated congruences.} For every $a,b\in A$, $Cg(\{(a,b)\})$ is denoted, simply, by $Cg(a,b)$, and called the {\em principal congruence generated by $(a,b)$} in ${\cal A}$. Clearly, finitely generated congruences are exactly the finite joins of principal congruences, because, for any $X\subseteq A^2$, $\displaystyle Cg(X)=\bigvee _{(a,b)\in X}Cg(a,b)$ and, since every congruence is reflexive, $Cg(\emptyset )=\Delta _{\cal A}=Cg(a,a)$ for all $a\in A$. The set of the finitely generated congruences of ${\cal A}$ shall be denoted by ${\cal K}({\cal A})$. It is well known (\cite{bur}) that ${\rm Con}({\cal A})$ is an {\em algebraic lattice,} that is a complete lattice such that each of its elements is a supremum of compact elements, and that the compact elements of ${\rm Con}({\cal A})$ are exactly the finitely generated congruences of ${\cal A}$: ${\cal K}({\cal A})$ coincides to the set of the compact elements of ${\rm Con}({\cal A})$.

The $\tau $--algebra ${\cal A}$ is said to be:

\begin{itemize}
\item {\em congruence--distributive} iff the lattice ${\rm Con}({\cal A})$ is distributive;
\item {\em congruence--permutable} iff each two congruences of ${\cal A}$ permute (with respect to the composition of binary relations);
\item {\em arithmetical} iff it is both congruence--distributive and congruence--permutable.\end{itemize}

Throughout the rest of this section, the $\tau $--algebra ${\cal A}$ shall be considerred non--empty and congruence--distributive.
 
For any $n\in \N ^*$ and any non--empty $\tau $--algebras ${\cal A}_1$, ${\cal A}_2$, $\ldots $, ${\cal A}_n$, with support sets $A_1$, $\ldots $, $A_n$, respectively, if $\displaystyle {\cal A}=\prod _{i=1}^n{\cal A}_i$, then:

\begin{itemize}
\item if $\theta _i\in {\rm Con}({\cal A}_i)$ for all $i\in \overline{1,n}$, then we denote: $\theta _1\times \ldots \times \theta _n=\{((x_1,\ldots ,x_n),(y_1,\ldots ,y_n))\in A^2\ |\ (\forall \, i\in \overline{1,n})\, ((x_i,y_i)\in \theta _i)\}$;
\item if $\theta \in {\rm Con}({\cal A})$, then, for all $i\in \overline{1,n}$, we denote by: $\displaystyle pr_i(\theta )=\{(x,y)\in A_i^2\ |\ (\exists \, (x_1,\ldots ,x_n),(y_1,\ldots ,y_n)\in A)\, (((x_1,\ldots ,x_n),(y_1,\ldots ,y_n))\in \theta ,x_i=x,y_i=y)\}$.\end{itemize}

More generally, for any non--empty family $({\cal A}_i)_{i\in I}$ of non--empty $\tau $--algebras, if $A_i$ is the support set of ${\cal A}_i$ for each $i\in I$ and $\displaystyle {\cal A}=\prod _{i\in I}{\cal A}_i$, then:

\begin{itemize}
\item if $\theta _i\in {\rm Con}({\cal A}_i)$ for all $i\in I$, then we denote: $\displaystyle \prod _{i\in I}\theta _i=\{((x_i)_{i\in I},(y_i)_{i\in I})\in A^2\ |\ (\forall \, i\in I)\, ((x_i,y_i)\in \theta _i)\}$;
\item if $\theta \in {\rm Con}({\cal A})$, then, for all $i\in I$, we denote by: $\displaystyle pr_i(\theta )=\{(x,y)\in A_i^2\ |\ (\exists \, (x_t)_{t\in I},(y_t)_{t\in I}\in A)\, (((x_t)_{t\in I},$\linebreak $(y_t)_{t\in I})\in \theta ,x_i=x,y_i=y)\}$.\end{itemize}

\begin{lemma}{\rm \cite{bj}} Let $n\in \N ^*$, ${\cal A}_1$, ${\cal A}_2$, $\ldots $, ${\cal A}_n$ be non--empty congruence--distributive $\tau $--algebras, and assume that $\displaystyle {\cal A}=\prod _{i=1}^n{\cal A}_i$. Then:

\begin{enumerate}
\item\label{totdinbj1} given any $\theta _i\in {\rm Con}({\cal A}_i)$ for every $i\in \overline{1,n}$, it follows that $\theta _1\times \ldots \times \theta _n\in {\rm Con}({\cal A})$ and, for each $i\in \overline{1,n}$, $pr_i(\theta _1\times \ldots \times \theta _n)=\theta _i$;
\item\label{totdinbj0} given any $\theta \in {\rm Con}({\cal A})$, it follows that: for each $i\in \overline{1,n}$, $pr_i(\theta )\in {\rm Con}({\cal A}_i)$, and $\theta =pr_1(\theta )\times \ldots \times pr_n(\theta )$;
\item\label{totdinbj2} the function $\displaystyle f:\prod _{i=1}^n{\rm Con}({\cal A}_i)\rightarrow {\rm Con}({\cal A})$, defined by $f(\theta _1,\ldots ,\theta _n)=\theta _1\times \ldots \times \theta _n$ for all $\theta _1\in {\rm Con}({\cal A}_1)$, $\ldots $, $\theta _n\in {\rm Con}({\cal A}_n)$, is a bounded lattice isomorphism, whose inverse is defined by: $f^{-1}(\theta )=(pr_1(\theta ),\ldots ,pr_n(\theta ))$ for all $\theta \in {\rm Con}({\cal A})$.\end{enumerate}\label{totdinbj}\end{lemma}

A proper congruence $\phi $ of ${\cal A}$ is called a {\em prime congruence} iff, given any $\theta _1,\theta _2\in {\rm Con}({\cal A})$, if $\theta _1\cap \theta _2\subseteq \phi $, then $\theta _1\subseteq \phi $ or $\theta _2\subseteq \phi $. The maximal elements of the set of proper congruences of ${\cal A}$ are called {\em maximal congruences.} We shall denote by ${\rm Spec}({\cal A})$ the set of the prime congruences of ${\cal A}$, by ${\rm Max}({\cal A})$ the set of the maximal congruences of ${\cal A}$ and by $\displaystyle {\rm Rad}({\cal A})=\bigcap _{\theta \in {\rm Max}({\cal A})}\theta $.

Let us consider the following hypothesis:

\begin{flushleft}
\hspace*{20pt} (H)$\quad $ $\nabla _{\cal A}\in {\cal K}({\cal A})$ (that is: $\nabla _{\cal A}$ is a compact element of ${\rm Con}({\cal A})$; equivalently: $\nabla _{\cal A}$ is a finitely generated congruence of ${\cal A}$).\end{flushleft}

Clearly, ${\cal K}({\cal A})$ is a sublattice of ${\rm Con}({\cal A})$, because, for all $X\subseteq A^2$ and $Y\subseteq A^2$, $Cg(X)\vee Cg(Y)=Cg(X\cap Y)$ and $Cg(X)\cap Cg(Y)=Cg(X\cup Y)$. Furthermore, $\Delta _{\cal A}=Cg(\emptyset )\in {\cal K}({\cal A})$, hence, if (H) is satisfied, then ${\cal K}({\cal A})$ is a bounded sublattice of ${\rm Con}({\cal A})$.

The following lemma is well known and straightforward.

\begin{lemma}
If ${\cal A}$ fulfills (H), then:

\begin{enumerate}
\item\label{folclor1} any proper congruence of ${\cal A}$ is included in a maximal congruence;
\item\label{folclor2} ${\rm Max}({\cal A})\subseteq {\rm Spec}({\cal A})$;
\item\label{folclor3} any proper congruence of ${\cal A}$ equals the intersection of the prime congruences that include it.\end{enumerate}\label{folclor}\end{lemma}

Clearly, if ${\cal A}$ fulfills (H), then: ${\cal A}$ is non--trivial iff $\Delta _{\cal A}\neq \nabla _{\cal A}$ iff $\Delta _{\cal A}$ is a proper congruence of ${\cal A}$ iff ${\cal A}$ has proper congruences iff ${\cal A}$ has maximal congruences, where the last equivalence follows from Lemma \ref{folclor}, (\ref{folclor1}).

We say that the {\em Chinese Remainder Theorem} ({\em CRT}, for short) holds in ${\cal A}$ iff: for all $n\in \N ^*$, any $\theta _1,\ldots ,\theta _n\in {\rm Con}({\cal A})$ and any $a_1,\ldots ,a_n\in A$, if $(a_i,a_j)\in \theta _i\vee \theta _j$ for all $i,j\in \overline{1,n}$, then there exists an $a\in A$ such that $(a,a_i)\in \theta _i$ for all $i\in \overline{1,n}$.

\begin{proposition}{\rm \cite{bj}} CRT holds in ${\cal A}$ iff ${\cal A}$ is arithmetical.\label{propozitie5.6}\end{proposition}

For every congruence $\theta $ of ${\cal A}$, we shall denote:

\begin{itemize}

\item by $[\theta )$ the principal filter of the lattice ${\rm Con}({\cal A})$ generated by $\theta $, that is $[\theta )=\{\phi \in {\rm Con}({\cal A})\ |\ \theta \subseteq \phi \}$; as is the case for any lattice filter, $[\theta )$ is a sublattice (not bounded sublattice) of ${\rm Con}({\cal A})$, and $h_{\textstyle \theta }:{\rm Con}({\cal A})\rightarrow [\theta )$, $h_{\textstyle \theta }(\phi )=\phi \vee \theta $ for all $\phi \in {\rm Con}({\cal A})$, is a bounded lattice morphism;
\item for any $a\in A$, by $a/\theta $ the equivalence class of $a$ with respect to $\theta $, and by $A/\theta $ the quotient set of $A$ with respect to $\theta $; in what follows, we shall assume that $A/\theta $ becomes an algebra of the same kind as ${\cal A}$, with the operations defined canonically, and we shall denote by ${\cal A}/\theta $ the algebraic structure of $A/\theta $;
\item by $p_{\textstyle \theta }:A\rightarrow A/\theta $ the canonical surjection with respect to $\theta $;
\item for any $X\subseteq A^2$ and any $Y\subseteq A$, by $X/\theta =\{(p_{\textstyle \theta }(a),p_{\textstyle \theta }(b))\ |\ (a,b)\in X\}=\{(a/\theta ,b/\theta )\ |\ (a,b)\in X\}$ and by $Y/\theta =p_{\textstyle \theta }(Y)=\{a/\theta \ |\ a\in Y\}$;
\item by $s_{\textstyle \theta }:{\rm Con}({\cal A}/\theta )\rightarrow [\theta )$ the function defined by: for any $\alpha \in {\rm Con}({\cal A}/\theta )$, $s_{\textstyle \theta }(\alpha )=\{(a,b)\in A^2\ |\ (p_{\textstyle \theta }(a),p_{\textstyle \theta }(b))$\linebreak $\in \alpha \}$; as shown in \cite{bur}, $s_{\textstyle \theta }$ is a bounded lattice isomorphism, whose inverse is defined by $s_{\textstyle \theta }^{-1}(\phi )=\phi /\theta =\{(p_{\textstyle \theta }(a),p_{\textstyle \theta }(b))\ |\ (a,b)\in \phi \}=\{(a/\theta ,b/\theta )\ |\ (a,b)\in \phi \}$ for every $\phi \in [\theta )$; consequently, if ${\cal A}$ and ${\cal A}/\theta $ fulfill (H), then ${\rm Max}({\cal A}/\theta)=s_{\textstyle \theta }^{-1}({\rm Max}({\cal A})\cap [\theta ))=\{\phi /\theta \ | \phi \in {\rm Max}({\cal A}),\theta \subseteq \phi \}$; thus, if ${\cal A}$ and ${\cal A}/\theta $ fulfill the hypothesis (H), then ${\rm Max}({\cal A}/{\rm Rad}({\cal A}))=s_{\textstyle {\rm Rad}({\cal A})}^{-1}({\rm Max}({\cal A}))$, which is isomorphic to ${\rm Max}({\cal A})$, and ${\rm Rad}({\cal A}/{\rm Rad}({\cal A}))={\rm Rad}({\cal A})/{\rm Rad}({\cal A})=\{\nabla _{\cal A}/{\rm Rad}({\cal A})\}$.\end{itemize}

\begin{lemma}{\rm \cite[Theorem $2.3$, (iii)]{bj}}
For every $\theta \in {\rm Con}({\cal A})$ and any $X\subseteq A^2$, $Cg_{{\cal A}/\theta }(X/\theta )=(Cg_{\cal A}(X)\vee \theta )/\theta $.\label{superlema}\end{lemma}

For every bounded distributive lattice $L$, we shall denote by ${\cal B}(L)$ the {\em Boolean center} of $L$, that is the set of the complemented elements of $L$. Then ${\cal B}(L)$ is a Boolean algebra, and, given any bounded distributive lattice $M$ and any bounded lattice morphism $f:L\rightarrow M$, the image of the restriction ${\cal B}(f)$ of $f$ to ${\cal B}(L)$ is included in ${\cal B}(M)$. Thus ${\cal B}$ becomes a covariant functor from the category of bounded distributive lattices to the category of Boolean algebras.

Clearly, if $(L_i)_{i\in I}$ is an arbitrary family of bounded distributive lattices, then $\displaystyle {\cal B}(\prod _{i\in I}L_i)=\prod _{i\in I}{\cal B}(L_i)$.

A congruence $\theta $ of ${\cal A}$ is called a {\em factor congruence} iff there exists $\theta ^*\in {\rm Con}({\cal A})$ such that $\theta \vee \theta ^*=\nabla _{\cal A}$, $\theta \cap \theta ^*=\Delta _{\cal A}$ and $\theta \circ  \theta ^*=\theta ^*\circ \theta $. In other words, the factor congruences of ${\cal A}$ are the elements of ${\cal B}({\rm Con}({\cal A}))$ that permute with their complement. If the algebra ${\cal A}$ is arithmetical, then it is congruence--permutable, thus the set of its factor congruences coincides to ${\cal B}({\rm Con}({\cal A}))$.

All finite direct products of algebras in this paper are considerred non--empty.

\begin{lemma}{\rm \cite{bj}} Let $n\in \N ^*$ and consider $n$ arithmetical algebras, ${\cal A}_1$, ${\cal A}_2$, $\ldots $, ${\cal A}_n$. Then the following are equivalent:

\begin{enumerate}
\item\label{dinbj1} ${\cal A}$ is isomorphic to the direct product $\displaystyle \prod _{i=1}^n{\cal A}_i$;
\item\label{dinbj2} there exist $\alpha _1,\ldots ,\alpha _n\in {\cal B}({\rm Con}({\cal A}))$ such that $\alpha _i\vee \alpha _j=\nabla _{\cal A}$ for all $i,j\in \overline{1,n}$ such that $i\neq j$, $\displaystyle \bigcap _{i=1}^n\alpha _i=\Delta _{\cal A}$ and ${\cal A}_i$ is isomorphic to ${\cal A}/_{\textstyle \alpha _i}$ for each $i\in \overline{1,n}$.\end{enumerate}\label{dinbj}\end{lemma}

For every bounded distributive lattice $L$, we denote by ${\rm Id}(L)$ the lattice of ideals of $L$, by ${\rm Spec}_{\rm Id}(L)$ the set of the prime ideals of $L$, by ${\rm Max}_{\rm Id}(L)$ the set of the maximal ideals of $L$ and by ${\rm Rad}_{\rm Id}(L)$ the intersection of all maximal ideals of $L$. $L$ is said to be {\em ${\rm Id}$--local} iff it has exactly one maximal ideal.

We recall that a bounded distributive lattice $L$ is called:

\begin{itemize}
\item a {\em normal} lattice iff, for every $x,y\in L$ such that $x\vee y=1$, there exist $e,f\in L$ such that $e\wedge f=0$ and $x\vee e=y\vee f=1$;
\item a {\em B--normal} lattice iff, for every $x,y\in L$ such that $x\vee y=1$, there exist $e,f\in {\cal B}(L)$ such that $e\wedge f=0$ and $x\vee e=y\vee f=1$;
\item a {\em conormal} lattice iff its dual is normal;
\item a {\em B--conormal} lattice iff its dual is B--normal.\end{itemize}

Throughout the rest of this section, $L$ will be a bounded distributive lattice. Throughout the rest of this paper, by lattice we shall mean bounded distributive lattice.

Clearly, any B--normal lattice is normal, and any B--conormal lattice is conormal. Trivially, any Boolean algebra is B--normal and B--conormal.

\begin{lemma}{\rm \cite{brdi}} The following are equivalent:

\begin{itemize}
\item $L$ is ${\rm Id}$--local;
\item for all $x,y\in L$, $x\vee y=1$ implies $x=1$ or $y=1$.\end{itemize}\label{lema2.5}\end{lemma}

\begin{lemma}{\rm \cite{bur}} ${\rm Rad}_{\rm Id}(L)=\{a\in L\ |\ (\forall \, x\in L)\, (a\vee x=1\Rightarrow x=1)\}$.\label{lema2.7}\end{lemma}

\section{Direct Products of Algebras}
\label{directproducts}

In this section, we obtain a series of results concerning the congruences of finite and those of arbitrary direct products of congruence--distributive algebras from an equational class, results which we need in the sequel.

Throughout the rest of this paper, $\tau $ will be a universal algebra signature, ${\cal C}$ shall be an equational class of congruence--distributive $\tau $--algebras and, unless mentioned otherwise, ${\cal A}$ will be a non--empty algebra from ${\cal C}$, with support set $A$. See the notations in Section \ref{preliminaries} for what follows.

\begin{lemma} Let $n\in \N ^*$, ${\cal A}_1,\ldots ,{\cal A}_n$ be algebras from ${\cal C}$, with support sets $A_1,\ldots ,A_n$, respectively, and assume that $\displaystyle {\cal A}=\prod _{i=1}^n{\cal A}_i$.

\begin{enumerate}
\item\label{lemaI1} For all $i\in \overline{1,n}$, let $X_i\subseteq A_i^2$, and let $X=\{((a_1,\ldots ,a_n),(b_1,\ldots ,b_n))\ |\ (\forall \, i\in \overline{1,n})\, ((a_i,b_i)\in X_i)\}$. Then $Cg_{\cal A}(X)=Cg_{{\cal A}_1}(X_1)\times \ldots \times Cg_{{\cal A}_n}(X_n)$.
\item\label{lemaI2} Let $X\subseteq A^2$ and, for all $i\in \overline{1,n}$, $X_i=\{(x,y)\in A_i^2\ |\ (\exists \, ((a_1,\ldots ,a_n),(b_1,\ldots ,b_n))\in X)\, (a_i=x,b_i=y)\}$. Then $Cg_{\cal A}(X)=Cg_{{\cal A}_1}(X_1)\times \ldots \times Cg_{{\cal A}_n}(X_n)$.\end{enumerate}\label{lemaI}\end{lemma}

\begin{proof} (\ref{lemaI1}) Clearly, $Cg_{{\cal A}_1}(X_1)\times \ldots \times Cg_{{\cal A}_n}(X_n)\supseteq X$, thus $Cg_{{\cal A}_1}(X_1)\times \ldots \times Cg_{{\cal A}_n}(X_n)\supseteq Cg_{\cal A}(X)$; also, for all $i\in \overline{1,n}$, $pr_i(Cg_{\cal A}(X))\supseteq X_i$, thus $pr_i(Cg_{\cal A}(X))\supseteq Cg_{{\cal A}_i}(X_i)$, hence $Cg_{\cal A}(X)\supseteq Cg_{{\cal A}_1}(X_1)\times \ldots \times Cg_{{\cal A}_n}(X_n)$ by Lemma \ref{totdinbj}. Therefore $Cg_{\cal A}(X)=Cg_{{\cal A}_1}(X_1)\times \ldots \times Cg_{{\cal A}_n}(X_n)$.

\noindent (\ref{lemaI2}) Denote $Y=\{((a_1,\ldots ,a_n),(b_1,\ldots ,b_n))\ |\ (\forall \, i\in \overline{1,n})\, ((a_i,b_i)\in X_i)\}\subseteq A^2$. Clearly, $X\subseteq Y$, thus $Cg_{\cal A}(X)\subseteq Cg_{\cal A}(Y)=Cg_{{\cal A}_1}(X_1)\times \ldots \times Cg_{{\cal A}_n}(X_n)$, by (\ref{lemaI1}). Also, clearly, for all $i\in \overline{1,n}$, $pr_i(Cg_{\cal A}(X))\supseteq X_i$, thus $pr_i(Cg_{\cal A}(X))\supseteq Cg_{{\cal A}_i}(X_i)$, hence $Cg_{\cal A}(X)\supseteq Cg_{{\cal A}_1}(X_1)\times \ldots \times Cg_{{\cal A}_n}(X_n)$ by Lemma \ref{totdinbj}. Therefore $Cg_{\cal A}(X)=Cg_{{\cal A}_1}(X_1)\times \ldots \times Cg_{{\cal A}_n}(X_n)$.\end{proof}

\begin{proposition} Let $n\in \N ^*$, ${\cal A}_1,\ldots ,{\cal A}_n$ be algebras from ${\cal C}$, and assume that $\displaystyle {\cal A}=\prod _{i=1}^n{\cal A}_i$. Then:

\begin{enumerate}
\item\label{propO1} the function $\displaystyle g:\prod _{i=1}^n{\cal B}({\rm Con}({\cal A}_i))\rightarrow {\cal B}({\rm Con}({\cal A}))$, defined by $g(\theta _1,\ldots ,\theta _n)=\theta _1\times \ldots \times \theta _n$ for all $\theta _1\in {\cal B}({\rm Con}({\cal A}_1)),\ldots ,\theta _n\in {\cal B}({\rm Con}({\cal A}_n))$, is a Boolean isomorphism;
\item\label{propO2} the function $\displaystyle h:\prod _{i=1}^n{\cal K}({\cal A}_i)\rightarrow {\cal K}({\cal A})$, defined by $h(\theta _1,\ldots ,\theta _n)=\theta _1\times \ldots \times \theta _n$ for all $\theta _1\in {\cal K}({\cal A}_1),\ldots ,\theta _n\in {\cal K}({\cal A}_n)$, is a bijection.\end{enumerate}\label{propO}\end{proposition}

\begin{proof} (\ref{propO1}) Since $\displaystyle {\cal B}(\prod _{i=1}^n{\rm Con}({\cal A}_i))=\prod _{i=1}^n{\cal B}({\rm Con}({\cal A}_i))$, it follows that $g={\cal B}(f)$, the image through the functor ${\cal B}$ of the bounded lattice isomorphism $f$ from Lemma \ref{totdinbj}, (\ref{totdinbj2}), hence $g$ is a Boolean isomorphism.

\noindent  (\ref{propO2}) With the notations in Lemma \ref{lemaI}, (\ref{lemaI1}), if $X_1,\ldots ,X_n$ are finite, then $X$ is finite, hence, if $\theta _i\in {\cal K}({\cal A}_i)$ for all $i\in \overline{1,n}$, then $\theta _1\times \ldots \times \theta _n\in {\cal K}({\cal A})$, thus $h$ is well defined. With the notations in Lemma \ref{lemaI}, (\ref{lemaI2}), if $X$ is finite, then $X_1,\ldots ,X_n$ are finite, thus $h$ is surjective. Finally, since $h$ is the restriction to $\displaystyle \prod _{i=1}^n{\cal K}({\cal A}_i)$ of the function $f$ from Lemma \ref{totdinbj}, (\ref{totdinbj2}), and $f$ is injective, it follows that $h$ is injective. Therefore $h$ is bijective.\end{proof}

\begin{lemma}
Let $({\cal A}_i)_{i\in I}$ be a non--empty family of non--empty algebras from ${\cal C}$, assume that $\displaystyle {\cal A}=\prod _{i\in I}{\cal A}_i$, let $J$ be an arbitrary non--empty set and, for every $i\in I$, let $(\theta _{i,j})_{j\in J}\subseteq {\rm Con}({\cal A}_i)$. Then $\displaystyle \bigcap _{j\in J}(\prod _{i\in I}\theta _{i,j})=\prod _{i\in I}(\bigcap _{j\in J}\theta _{i,j})$.\label{interscong}\end{lemma}

\begin{proof} For all $i\in I$, let $A_i$ be the support set of ${\cal A}_i$ and $a_i,b_i\in A_i$. Then: $\displaystyle ((a_i)_{i\in I},(b_i)_{i\in I})\in \bigcap _{j\in J}(\prod _{i\in I}\theta _{i,j})$ iff, for all $i\in I$ and all $j\in J$, $(a_i,b_i)\in \theta _{i,j}$, iff $\displaystyle ((a_i)_{i\in I},(b_i)_{i\in I})\in \prod _{i\in I}(\bigcap _{j\in J}\theta _{i,j})$, hence the equality in the enunciation.\end{proof}

\begin{lemma} Let $({\cal A}_i)_{i\in I}$ be a non--empty family of non--empty algebras from ${\cal C}$, and assume that $\displaystyle {\cal A}=\prod _{i\in I}{\cal A}_i$. Then:

\begin{enumerate}
\item\label{lemal1} for any $\displaystyle (\theta _i)_{i\in I}\in \prod _{i\in I}{\rm Con}({\cal A}_i)$, it follows that $\displaystyle \prod _{i\in I}\theta _i\in {\rm Con}({\cal A})$ and, for each $j\in \overline{1,n}$, $\displaystyle pr_j(\prod _{i\in I}\theta _i)=\theta _j$;

\item\label{lemal0} for any $\theta \in {\rm Con}({\cal A})$, it follows that: for each $i\in I$, $pr_i(\theta )\in {\rm Con}({\cal A}_i)$, and $\displaystyle \theta \subseteq \prod _{i\in I}pr_i(\theta )$;
\item\label{lemal2} the function $\displaystyle f:\prod _{i\in I}{\rm Con}({\cal A}_i)\rightarrow {\rm Con}({\cal A})$, defined by $\displaystyle f((\theta _i)_{i\in I})=\prod _{i\in I}\theta _i$ for all $\displaystyle (\theta _i)_{i\in I}\in \prod _{i\in I}{\rm Con}({\cal A}_i)$, is an injective bounded lattice morphism.\end{enumerate}\label{lemal}\end{lemma}

\begin{proof} For every $i\in I$, let $A_i$ be the support set of ${\cal A}_i$.

\noindent (\ref{lemal1}) Let $\displaystyle (\theta _i)_{i\in I}\in \prod _{i\in I}{\rm Con}({\cal A}_i)$, and let $\omega $ be an operation symbol from $\tau $, of arity $k\in \N ^*$. For each $j\in \overline{1,k}$, let $\displaystyle (a_{j,i})_{i\in I},(b_{j,i})_{i\in I}\in \prod _{i\in I}A_i$ and, for every $i\in I$, let $\omega ^{\textstyle {\cal A}_i}(a_{1,i},\ldots ,a_{k,i})=a_i\in A_i$ and $\omega ^{\textstyle {\cal A}_i}(b_{1,i},\ldots ,b_{k,i})=b_i\in A_i$. Assume that, for all $j\in \overline{1,k}$, $\displaystyle ((a_{j,i})_{i\in I},(b_{j,i})_{i\in I})\in \prod _{i\in I}\theta _i$, that is, for all $i\in I$ and all $j\in \overline{1,k}$, $(a_{j,i},b_{j,i})\in \theta _i$. For all $i\in I$, since $\theta _i\in {\rm Con}({\cal A}_i)$, it follows that $(a_i,b_i)\in \theta _i$; thus $\displaystyle ((a_i)_{i\in I},(b_i)_{i\in I})\in \prod _{i\in I}{\rm Con}({\cal A}_i)$. Since $\displaystyle {\cal A}=\prod _{i\in I}{\cal A}_i$, the following hold: $\omega ^{\textstyle {\cal A}}((a_{1,i})_{i\in I},\ldots ,(a_{k,i})_{i\in I})=(\omega ^{\textstyle {\cal A}_i}(a_{1,i},\ldots ,a_{k,i}))_{i\in I}=(a_i)_{i\in I}$ and $\omega ^{\textstyle {\cal A}}((b_{1,i})_{i\in I},\ldots ,(b_{k,i})_{i\in I})=(\omega ^{\textstyle {\cal A}_i}(b_{1,i},\ldots ,b_{k,i}))_{i\in I}=(b_i)_{i\in I}$. Therefore $\displaystyle (\omega ^{\textstyle {\cal A}}((a_{1,i})_{i\in I},\ldots ,(a_{k,i})_{i\in I}),$\linebreak $\displaystyle \omega ^{\textstyle {\cal A}}((b_{1,i})_{i\in I},\ldots ,(b_{k,i})_{i\in I}))\in \prod _{i\in I}\theta _i$. Hence $\displaystyle \prod _{i\in I}\theta _i\in {\rm Con}({\cal A})$. It is immediate that, for all $j\in I$, $\displaystyle pr_j(\prod _{i\in I}\theta _i)=\theta _j$. 

\noindent (\ref{lemal0}) Straightforward.

\noindent (\ref{lemal2}) (\ref{lemal1}) ensures us that the image of $f$ is, indeed, included in ${\rm Con}({\cal A})$. Bounded distributive lattices form an equational class, thus $\displaystyle \prod _{i\in I}{\rm Con}({\cal A}_i)$ is a bounded distributive lattice, with the operations defined componentwise; from this it is straightforward that $f$ is a bounded lattice morphism. The injectivity of $f$ follows from the second statement in (\ref{lemal1}).\end{proof}

Throughout the rest of this paper, we shall assume that all non--empty algebras from ${\cal C}$ fulfill the hypothesis (H) (see Section \ref{preliminaries}).

\begin{proposition} Let $n\in \N ^*$, ${\cal A}_1$, ${\cal A}_2$, $\ldots $, ${\cal A}_n$ be non--empty algebras from ${\cal C}$, and assume that $\displaystyle {\cal A}=\prod _{i=1}^n{\cal A}_i$. Then:

\begin{enumerate}
\item\label{specprod1} $\displaystyle {\rm Max}({\cal A})=\bigcup _{i=1}^n\{\nabla _{{\cal A}_{\scriptstyle 1}}\times \ldots \times \nabla _{{\cal A}_{\scriptstyle i-1}}\times \theta \times \nabla _{{\cal A}_{\scriptstyle i+1}}\times \ldots \times \nabla _{{\cal A}_{\scriptstyle n}}\ |\ \theta \in {\rm Max}({\cal A}_i)\}$; consequently, $\displaystyle |{\rm Max}({\cal A})|=\sum _{i=1}^n|{\rm Max}({\cal A}_i)|$ and ${\rm Rad}({\cal A})={\rm Rad}({\cal A}_1)\times \ldots \times {\rm Rad}({\cal A}_n)$;
\item\label{specprod2} $\displaystyle {\rm Spec}({\cal A})=\bigcup _{i=1}^n\{\nabla _{{\cal A}_{\scriptstyle 1}}\times \ldots \times \nabla _{{\cal A}_{\scriptstyle i-1}}\times \theta \times \nabla _{{\cal A}_{\scriptstyle i+1}}\times \ldots \times \nabla _{{\cal A}_{\scriptstyle n}}\ |\ \theta \in {\rm Spec}({\cal A}_i)\}$; consequently, $\displaystyle |{\rm Spec}({\cal A})|=\sum _{i=1}^n|{\rm Spec}({\cal A}_i)|$.\end{enumerate}\label{specprod}\end{proposition}

\begin{proof} (\ref{specprod1}) It is immediate that, for every $i\in \overline{1,n}$ and every $\theta \in {\rm Max}({\cal A}_i)$, $\nabla _{{\cal A}_{\scriptstyle 1}}\times \ldots \times \nabla _{{\cal A}_{\scriptstyle i-1}}\times \theta \times \nabla _{{\cal A}_{\scriptstyle i+1}}\times \ldots \times \nabla _{{\cal A}_{\scriptstyle n}}\in {\rm Max}({\cal A})$.

Now, given any $\phi \in {\rm Max}({\cal A})$, it follows that $\phi $ is a proper congruence, thus there exist $a,b\in A$ such that $(a,b)\notin \phi $, that is, for some $i\in \overline{1,n}$, $(a_i,b_i)\notin \phi _i$, where we have denoted $a=(a_1,\ldots ,a_n)$, $b=(b_1,\ldots ,b_n)$ and $\phi =\phi _1\times \ldots \times \phi _n$, with $a_k,b_k\in A_k$ (the support set of ${\cal A}_k$) and $\phi _k\in {\rm Con}({\cal A}_k)$ for all $k\in \overline{1,n}$. So $\phi _i\neq \nabla _{{\cal A}_{\scriptstyle i}}$. Assume by absurdum that there exists a $j\in \overline{1,n}$ such that $j\neq i$ and $\phi _j\neq \nabla _{{\cal A}_{\scriptstyle j}}$. Then $\phi \subsetneq \nabla _{{\cal A}_{\scriptstyle 1}}\times \ldots \times \nabla _{{\cal A}_{\scriptstyle i-1}}\times \theta \times \nabla _{{\cal A}_{\scriptstyle i+1}}\times \ldots \times \nabla _{{\cal A}_{\scriptstyle n}}\subsetneq \nabla _{\cal A}$, which is a contradiction to the maximality of $\phi $. Hence $\phi =\nabla _{{\cal A}_{\scriptstyle 1}}\times \ldots \times \nabla _{{\cal A}_{\scriptstyle i-1}}\times \phi _i\times \nabla _{{\cal A}_{\scriptstyle i+1}}\times \ldots \times \nabla _{{\cal A}_{\scriptstyle n}}$. It is clear that the proper congruence $\phi _i\in {\rm Max}({\cal A}_i)$, because otherwise we would get another contradiction to the maximality of $\phi $.

Therefore ${\rm Max}({\cal A})$ has the form in the enunciation, otherwise written $\displaystyle {\rm Max}({\cal A})=\bigcup _{i=1}^n(\{\nabla _{{\cal A}_{\scriptstyle 1}}\}\times \ldots \times \{\nabla _{{\cal A}_{\scriptstyle i-1}}\}\times {\rm Max}({\cal A}_i)\times \{\nabla _{{\cal A}_{\scriptstyle i+1}}\}\times \ldots \times \{\nabla _{{\cal A}_{\scriptstyle n}}\})$, from which the expression of its cardinality follows by noticing that, since $\nabla _{{\cal A}_{\scriptstyle i}}\notin {\rm Max}({\cal A}_i)$ for any $i\in \overline{1,n}$, the sets $\{\nabla _{{\cal A}_{\scriptstyle 1}}\}\times \ldots \times \{\nabla _{{\cal A}_{\scriptstyle i-1}}\}\times {\rm Max}({\cal A}_i)\times \{\nabla _{{\cal A}_{\scriptstyle i+1}}\}\times \ldots \times \{\nabla _{{\cal A}_{\scriptstyle n}}\}$, with $i\in \overline{1,n}$, are mutually disjoint, and that they are in bijection to the sets ${\rm Max}({\cal A}_i)$, with $i\in \overline{1,n}$, respectively. The formula of ${\rm Rad}({\cal A})$ now follows by Lemma \ref{interscong}.

\noindent (\ref{specprod2}) Let $i\in \overline{1,n}$, $\theta \in {\rm Spec}({\cal A}_i)$ and denote $\phi =\nabla _{{\cal A}_{\scriptstyle 1}}\times \ldots \times \nabla _{{\cal A}_{\scriptstyle i-1}}\times \theta \times \nabla _{{\cal A}_{\scriptstyle i+1}}\times \ldots \times \nabla _{{\cal A}_{\scriptstyle n}}$. Let $\alpha ,\beta \in {\rm Con}({\cal A})$ such that $\alpha \cap \beta \subseteq \theta $. Then, if $\alpha =\alpha _1\times \ldots \times \alpha _n$ and $\beta =\beta _1\times \ldots \times \beta _n$, with $\alpha _k,\beta _k\in {\rm Con}({\cal A}_k)$ for all $k\in \overline{1,n}$, it follows that $\alpha _i\cap \beta _i\subseteq \theta \in {\rm Spec}({\cal A}_i)$, thus $\alpha _i\subseteq \theta $ or $\beta _i\subseteq \theta $, thus $\alpha \subseteq \phi $ or $\beta \subseteq \phi $. Therefore $\phi \in {\rm Spec}({\cal A})$.

Now let $\phi \in {\rm Spec}({\cal A})$, so $\phi $ is a proper congruence of ${\cal A}$, from which, just as above for maximal congruences, we get that, if $\phi =\phi _1\times \ldots \times \phi _n$, with $\phi _k\in {\rm Con}({\cal A}_k)$ for all $k\in \overline{1,n}$, then there exists an $i\in \overline{1,n}$ such that $\phi _i$ is a proper congruence of ${\cal A}_i$. If there also exists a $j\in \overline{1,n}$ with $j\neq i$ and $\phi _j$ a proper congruence of ${\cal A}_j$, then, by denoting $\alpha =\phi _1\times \ldots \times \phi _{i-1}\times \nabla _{{\cal A}_{\scriptstyle i}}\times \phi _{i+1}\times \ldots \times \phi _n$ and $\beta =\phi _1\times \ldots \times \phi _{j-1}\times \nabla _{{\cal A}_{\scriptstyle j}}\times \phi _{j+1}\times \ldots \times \phi _n$, we get that $\alpha \cap \beta =\phi $, but $\alpha \nsubseteq \phi $ and $\beta \nsubseteq \phi $, which is a contradiction to the primality of $\phi $. Hence $\phi =\nabla _{{\cal A}_{\scriptstyle 1}}\times \ldots \times \nabla _{{\cal A}_{\scriptstyle i-1}}\times \phi _i\times \nabla _{{\cal A}_{\scriptstyle i+1}}\times \ldots \times \nabla _{{\cal A}_{\scriptstyle n}}$, with $\phi _i\neq \nabla _{{\cal A}_{\scriptstyle i}}$. Assume by absurdum that $\phi _i\notin {\rm Spec}({\cal A}_i)$, that is there exist $\alpha _i,\beta _i\in {\rm Con}({\cal A}_i)$ such that $\alpha _i\cap \beta _i\subseteq \phi _i$, but $\alpha _i\nsubseteq \phi _i$ and $\beta _i\nsubseteq \phi _i$. Denote $\alpha =\nabla _{{\cal A}_{\scriptstyle 1}}\times \ldots \times \nabla _{{\cal A}_{\scriptstyle i-1}}\times \alpha _i\times \nabla _{{\cal A}_{\scriptstyle i+1}}\times \ldots \times \nabla _{{\cal A}_{\scriptstyle n}}$ and $\beta =\nabla _{{\cal A}_{\scriptstyle 1}}\times \ldots \times \nabla _{{\cal A}_{\scriptstyle i-1}}\times \beta _i\times \nabla _{{\cal A}_{\scriptstyle i+1}}\times \ldots \times \nabla _{{\cal A}_{\scriptstyle n}}$. Then $\alpha \cap \beta =\nabla _{{\cal A}_{\scriptstyle 1}}\times \ldots \times \nabla _{{\cal A}_{\scriptstyle i-1}}\times (\alpha _i\cap \beta _i)\times \nabla _{{\cal A}_{\scriptstyle i+1}}\times \ldots \times \nabla _{{\cal A}_{\scriptstyle n}}\subseteq \nabla _{{\cal A}_{\scriptstyle 1}}\times \ldots \times \nabla _{{\cal A}_{\scriptstyle i-1}}\times \phi _i\times \nabla _{{\cal A}_{\scriptstyle i+1}}\times \ldots \times \nabla _{{\cal A}_{\scriptstyle n}}=\phi $, but $\alpha \nsubseteq \phi $ and $\beta \nsubseteq \phi $, which is a contradiction to the primality of $\phi $. Hence $\phi _i\in {\rm Spec}({\cal A}_i)$.

Therefore ${\rm Spec}({\cal A})$ has the form in the enunciation, otherwise written $\displaystyle {\rm Spec}({\cal A})=\bigcup _{i=1}^n(\{\nabla _{{\cal A}_{\scriptstyle 1}}\}\times \ldots \times \{\nabla _{{\cal A}_{\scriptstyle i-1}}\}\times {\rm Spec}({\cal A}_i)\times \{\nabla _{{\cal A}_{\scriptstyle i+1}}\}\times \ldots \times \{\nabla _{{\cal A}_{\scriptstyle n}}\})$, from which the expression of its cardinality follows by noticing that, since $\nabla _{{\cal A}_{\scriptstyle i}}\notin {\rm Spec}({\cal A}_i)$ for any $i\in \overline{1,n}$, the sets $\{\nabla _{{\cal A}_{\scriptstyle 1}}\}\times \ldots \times \{\nabla _{{\cal A}_{\scriptstyle i-1}}\}\times {\rm Spec}({\cal A}_i)\times \{\nabla _{{\cal A}_{\scriptstyle i+1}}\}\times \ldots \times \{\nabla _{{\cal A}_{\scriptstyle n}}\}$, with $i\in \overline{1,n}$, are mutually disjoint, and that they are in bijection to the sets ${\rm Spec}({\cal A}_i)$, with $i\in \overline{1,n}$, respectively.\end{proof}

\begin{proposition}
Let $({\cal A}_i)_{i\in I}$ be a non--empty family of non--empty algebras from ${\cal C}$, and assume that $\displaystyle {\cal A}=\prod _{i\in I}{\cal A}_i$. Then:

\begin{enumerate}
\item\label{specprodarb1} $\displaystyle {\rm Max}({\cal A})\supseteq \bigcup _{j\in I}\{\prod _{i\in I}\phi _i\ |\ \phi _j\in {\rm Max}({\cal A}_j),(\forall \, i\in I\setminus \{j\})\, (\phi _i=\nabla _{{\cal A}_{\scriptstyle i}})\}$; consequently, $\displaystyle |{\rm Max}({\cal A})|\geq \sum _{j\in I}|{\rm Max}({\cal A}_j)|$ and $\displaystyle {\rm Rad}({\cal A})\subseteq \prod _{j\in I}{\rm Rad}({\cal A}_j)$;

\item\label{specprodarb2} $\displaystyle {\rm Spec}({\cal A})\supseteq \bigcup _{j\in I}\{\prod _{i\in I}\phi _i\ |\ \phi _j\in {\rm Spec}({\cal A}_j),(\forall \, i\in I\setminus \{j\})\, (\phi _i=\nabla _{{\cal A}_{\scriptstyle i}})\}$; consequently, $\displaystyle |{\rm Spec}({\cal A})|\geq \sum _{j\in I}|{\rm Spec}({\cal A}_j)|$.\end{enumerate}\label{specprodarb}\end{proposition}

\begin{proof} (\ref{specprodarb1}) For every $j\in I$, the following hold: $\displaystyle \prod _{i\in I}{\cal A}_i={\cal A}_j\times \prod _{i\in I\setminus \{j\}}{\cal A}_i$, hence, according to Proposition \ref{specprod}, (\ref{specprod1}), applied for $n=2$, $\displaystyle {\rm Max}({\cal A})={\rm Max}(\prod _{i\in I}{\cal A}_i)={\rm Max}({\cal A}_j\times \prod _{i\in I\setminus \{j\}}{\cal A}_i)=\{\theta _j\times \nabla _{\prod _{i\in I\setminus \{j\}}{\cal A}_i}\ |\ \theta _j\in {\rm Max}({\cal A}_j)\}\cup \{\nabla _{{\cal A}_j}\times \theta \ |\ \theta \in {\rm Max}(\prod _{i\in I\setminus \{j\}}{\cal A}_i)\}\supseteq \{\theta _j\times \nabla _{\prod _{i\in I\setminus \{j\}}{\cal A}_i}\ |\ \theta _j\in {\rm Max}({\cal A}_j)\}=\{\theta _j\times \prod _{i\in I\setminus \{j\}}\nabla _{{\cal A}_i}\ |\ \theta _j\in {\rm Max}({\cal A}_j)\}$, and the latter set is isomorphic to ${\rm Max}({\cal A}_j)$, so its cardinality coincides to $|{\rm Max}({\cal A}_j)|$. Therefore, $\displaystyle {\rm Max}({\cal A})\supseteq \bigcup _{j\in I}\{\theta _j\times \prod _{i\in I\setminus \{j\}}\nabla _{{\cal A}_i}\ |\ \theta _j\in {\rm Max}({\cal A}_j)\}$, and the sets in this union are, obviously, mutually disjoint, hence $\displaystyle |{\rm Max}({\cal A})|\geq \sum _{j\in I}|{\rm Max}({\cal A}_j)|$. Also, the previous inclusion shows that $\displaystyle {\rm Rad}({\cal A})\subseteq \bigcap _{j\in I}\bigcap _{\theta _j\in {\rm Max}({\cal A}_j)}(\theta _j\times \prod _{i\in I\setminus \{j\}}\nabla _{{\cal A}_i})=\prod _{j\in I}\bigcap _{\theta _j\in {\rm Max}({\cal A}_j)}\theta _j=\prod _{j\in I}{\rm Rad}({\cal A}_j)$, by Lemma \ref{interscong}.

\noindent (\ref{specprodarb2}) Analogously to (\ref{specprodarb1}), but applying Proposition \ref{specprod}, (\ref{specprod2}), instead of Proposition \ref{specprod}, (\ref{specprod1}).\end{proof}

\section{Introducing the Congruence Boolean Lifting Property}
\label{thecblp}

In this section we introduce the property we call CBLP, which constitutes the subject of this paper, identify important classes of congruences and classes of algebras which fulfill CBLP, prove a structure theorem for algebras with CBLP, and study CBLP in quotient algebras, in direct products of algebras and in relation to other significant properties concerning congruence--distributive algebras. 

Until mentioned otherwise, $\theta $ shall be an arbitrary but fixed congruence of ${\cal A}$. Let us consider the functions $u_{\textstyle \theta }:{\rm Con}({\cal A})\rightarrow {\rm Con}({\cal A}/\theta )$ and $v_{\textstyle \theta }:{\rm Con}({\cal A})\rightarrow [\theta )$, defined by: for all $\alpha \in {\rm Con}({\cal A})$, $u_{\textstyle \theta }(\alpha )=(\alpha \vee \theta )/\theta $ and $v_{\textstyle \theta }(\alpha )=\alpha \vee \theta $.

\begin{lemma}
\begin{enumerate}
\item\label{uvstheta1} $u_{\textstyle \theta }$ and $v_{\textstyle \theta }$ are bounded lattice morphisms;
\item\label{uvstheta2} the first diagram below (in the category of bounded distributive lattices) is commutative, and hence the second diagram below (in the category of Boolean algebras) is commutative; since $s_{\textstyle \theta }$ is a bounded lattice isomorphism (see Section \ref{preliminaries}), it follows that ${\cal B}(s_{\textstyle \theta })$ is a Boolean isomorphism:\end{enumerate}

\vspace*{-10pt}

\begin{center}
\begin{tabular}{cc}
\begin{picture}(120,53)(0,0)
\put(0,30){${\rm Con}({\cal A})$}
\put(35,33){\vector (1,0){40}}
\put(15,27){\vector (3,-2){33}}
\put(94,26){\vector (-3,-2){33}}
\put(50,38){$u_{\textstyle \theta }$}
\put(23,11){$v_{\textstyle \theta }$}
\put(80,10){$s_{\textstyle \theta }$}
\put(76,30){${\rm Con}({\cal A}/\theta )$}
\put(50,0){$[\theta )$}
\end{picture}

&\hspace*{20pt}
\begin{picture}(160,53)(0,0)
\put(0,30){${\cal B}({\rm Con}({\cal A}))$}
\put(49,33){\vector (1,0){40}}

\put(15,27){\vector (3,-2){33}}
\put(110,27){\vector (-3,-2){34}}
\put(56,38){${\cal B}(u_{\textstyle \theta })$}

\put(7,10){${\cal B}(v_{\textstyle \theta })$}
\put(93,8){${\cal B}(s_{\textstyle \theta })$}
\put(91,30){${\cal B}({\rm Con}({\cal A}/\theta ))$}

\put(50,0){${\cal B}([\theta ))$}
\end{picture}
\end{tabular}\end{center}\label{uvstheta}\end{lemma}

\vspace*{-7pt}

\begin{proof} Straightforward.\end{proof}

\begin{lemma}
The following are equivalent:

\begin{enumerate}
\item\label{camcalalr1} $\theta \subseteq {\rm Rad}({\cal A})$ and $\theta \in {\cal B}({\rm Con}({\cal A}))$;
\item\label{camcalalr2} $\theta=\Delta _{\cal A}$.\end{enumerate}\label{camcalalr}\end{lemma}

\begin{proof} (\ref{camcalalr2})$\Rightarrow $(\ref{camcalalr1}): Obvious.

\noindent (\ref{camcalalr1})$\Rightarrow $(\ref{camcalalr2}): Assume that $\theta \subseteq {\rm Rad}({\cal A})$ and $\theta \in {\cal B}({\rm Con}({\cal A}))$. The fact that $\theta \in {\cal B}({\rm Con}({\cal A}))$ means that there exists $\phi \in {\rm Con}({\cal A})$ such that $\theta \cap \phi =\Delta _{\cal A}$ and $\theta \vee \phi =\nabla _{\cal A}$. Now the fact that $\theta \subseteq {\rm Rad}({\cal A})$, that is $\theta \subseteq \mu $ for all $\mu\in {\rm Max}({\rm Con}({\cal A}))$, shows that $\phi \vee \mu =\nabla _{\cal A}$, for all $\mu \in {\rm Max}({\rm Con}({\cal A}))$. Assume by absurdum that $\phi \neq \nabla _{\cal A}$. Then $\phi \subseteq \mu $ for some $\mu\in {\rm Max}({\rm Con}({\cal A}))$, according to Lemma \ref{folclor}. We get that $\nabla _{\cal A}=\theta \vee \phi \subseteq \mu \vee \mu =\mu $, a contradiction to $\mu\in {\rm Max}({\rm Con}({\cal A}))\subseteq {\rm Con}({\cal A})\setminus \{\nabla _{\cal A}\}$. Thus $\phi =\nabla _{\cal A}$, hence $\theta =\theta \cap \nabla _{\cal A}=\theta \cap \phi =\Delta _{\cal A}$.\end{proof}

\begin{corollary}
If $\theta \subseteq {\rm Rad}({\cal A})$, then ${\cal B}(u_{\textstyle \theta })$ and ${\cal B}(v_{\textstyle \theta })$ are injective.\label{candsrad}\end{corollary}

\begin{proof} Since $s_{\textstyle \theta }$ is a bounded lattice isomorphism, it follows that ${\cal B}(s_{\textstyle \theta })$ is a Boolean isomorphism. Now Lemma \ref{uvstheta}, (\ref{uvstheta2}), shows that ${\cal B}(v_{\textstyle \theta })$ is injective iff ${\cal B}(u_{\textstyle \theta })$ is injective. Let $\alpha \in {\cal B}({\rm Con}({\cal A}))$ such that ${\cal B}(v_{\textstyle \theta })(\alpha )=\theta $, that is $\alpha \vee \theta =\theta $, so that $\alpha \subseteq \theta \subseteq {\rm Rad}({\cal A})$. Thus $\alpha \in {\cal B}({\rm Con}({\cal A}))$ and $\alpha \subseteq {\rm Rad}({\cal A})$, hence $\alpha =\Delta _{\cal A}$ by Lemma \ref{camcalalr}. The fact that ${\cal B}(v_{\textstyle \theta })$ is a Boolean morphism now shows that ${\cal B}(v_{\textstyle \theta })$ is injective.\end{proof}

\begin{definition}
We say that $\theta $ has the {\em Congruence Boolean Lifting Property} (abbreviated {\em CBLP}) iff ${\cal B}(u_{\textstyle \theta })$ is surjective.\end{definition}

\begin{remark}
As shown by Lemma \ref{uvstheta}, (\ref{uvstheta2}), since ${\cal B}(s_{\textstyle \theta })$ is bijective, we have: $\theta $ has CBLP iff ${\cal B}(v_{\textstyle \theta })$ is surjective.

Furthermore, according to Corollary \ref{candsrad}, if $\theta \subseteq {\rm Rad}({\cal A})$, then we have: $\theta $ has CBLP iff ${\cal B}(u_{\textstyle \theta })$ is bijective iff ${\cal B}(v_{\textstyle \theta })$ is bijective.\label{usiv}\end{remark}

\begin{remark}
Obviously, $v_{\theta }$ is surjective, because, for any $\psi \in [\theta )\subseteq {\rm Con}({\cal A})$, $v_{\theta }(\psi )=\psi \vee \theta =\psi $.\label{vesurj}\end{remark}

\begin{definition}
Let $\Omega \subseteq {\rm Con}({\cal A})$. We say that ${\cal A}$ has the {\em $\Omega $--Congruence Boolean Lifting Property} (abbreviated {\em $\Omega $--CBLP}) iff every $\omega \in \Omega $ has CBLP. We say that ${\cal A}$ has the {\em Congruence Boolean Lifting Property (CBLP)} iff ${\cal A}$ has ${\rm Con}({\cal A})$--CBLP.\end{definition}

The definition of CBLP is inspired by a property in \cite[Lemma $4$]{banasch}.

\begin{proposition}\begin{enumerate}
\item\label{toatebool1} If $[\theta )\subseteq {\cal B}({\rm Con}({\cal A}))$, then each $\phi \in [\theta )$ has CBLP and fulfills ${\cal B}({\rm Con}({\cal A}/\phi ))={\rm Con}({\cal A}/\phi )$.
\item\label{toatebool0} If ${\cal B}({\rm Con}({\cal A}))={\rm Con}({\cal A})$, then ${\cal A}$ has CBLP and, for each $\phi \in {\rm Con}({\cal A})$, ${\cal B}({\rm Con}({\cal A}/\phi ))={\rm Con}({\cal A}/\phi )$.\end{enumerate}\label{toatebool}\end{proposition}

\begin{proof} (\ref{toatebool1}) Assume that $[\theta )\subseteq {\cal B}({\rm Con}({\cal A}))$, and let $\phi \in [\theta )$, so that $[\phi )\subseteq [\theta )\subseteq {\cal B}({\rm Con}({\cal A}))$. Let $\gamma \in {\cal B}({\rm Con}({\cal A}/\phi ))$. Then $\gamma \in {\rm Con}({\cal A}/\phi )$, so $\gamma =\alpha /\phi $, for some $\alpha \in [\phi )\subseteq {\cal B}({\rm Con}({\cal A}))$, thus $\gamma =\alpha /\phi =(\alpha \vee \phi )/\phi =u_{\phi }(\alpha )={\cal B}(u_{\phi })(\alpha )$. Therefore ${\cal B}({\rm Con}({\cal A}/\phi ))\supseteq {\cal B}(u_{\phi })({\cal B}({\rm Con}({\cal A})))\supseteq {\rm Con}({\cal A}/\phi )\supseteq {\cal B}({\rm Con}({\cal A}/\phi ))$, hence ${\cal B}(u_{\phi })({\cal B}({\rm Con}({\cal A})))={\cal B}({\rm Con}({\cal A}/\phi ))={\rm Con}({\cal A}/\phi )$, thus $\phi $ has CBLP.

\noindent (\ref{toatebool0}) This statement can be derived from Remarks \ref{vesurj} and \ref{usiv}, but it also follows from (\ref{toatebool1}), by the fact that ${\rm Con}({\cal A})=[\Delta _{\cal A})$.\end{proof}

\begin{remark} $\Delta _{\cal A}$ and $\nabla _{\cal A}$ have CBLP. For $\Delta _{\cal A}$, we can apply Remark \ref{usiv} and the fact that $\Delta _{\cal A}\subseteq {\rm Rad}({\cal A})$, or we can notice that $[\Delta _{\cal A})={\rm Con}({\cal A})$ and $v_{\Delta _{\cal A}}$ is the identity of ${\rm Con}({\cal A})$, thus it is a bounded lattice isomorphism, hence ${\cal B}(v_{\Delta _{\cal A}})$ is a Boolean isomorphism, so it is surjective, thus $\Delta _{\cal A}$ has CBLP. For $\nabla _{\cal A}$ we can apply Proposition \ref{toatebool}, (\ref{toatebool1}), or simply notice that $[\nabla _{\cal A})=\{\nabla _{\cal A}\}$, thus ${\cal B}([\nabla _{\cal A}))=\{\nabla _{\cal A}\}$, hence ${\cal B}(v_{\nabla _{\cal A}})$ is surjective, thus $\nabla _{\cal A}$ has CBLP.\label{deltanabla}\end{remark}

In what follows, the complementation in the Boolean algebra ${\cal B}({\rm Con}({\cal A}))$ shall be denoted by $\neg \, $, and, for every $\phi \in {\rm Con}({\cal A})$, the complementation in the Boolean algebra ${\cal B}([\phi ))$ shall be denoted by $\neg _{\phi }$. Notice that $\theta /\theta =\Delta _{{\cal A}/\theta }$ and $\nabla _{\cal A}/\theta =\nabla _{{\cal A}/\theta }$.

\begin{remark} If $\theta \in {\rm Max}({\cal A})$, then the following hold:

\begin{itemize}
\item $[\theta )=\{\alpha \in {\rm Con}({\cal A})\ |\ \theta \subseteq \alpha \}=\{\theta ,\nabla _{\cal A}\}$, with $\theta \neq \nabla _{\cal A}$, thus $[\theta )$ is the two--element chain, which is a Boolean algebra, so ${\cal B}([\theta ))=[\theta )=\{\theta ,\nabla _{\cal A}\}$;
\item ${\rm Con}({\cal A}/\theta )=\{\alpha /\theta \ |\ \alpha \in [\theta )\}=\{\theta /\theta ,\nabla _{\cal A}/\theta \}$, with $\theta /\theta \neq \nabla _{\cal A}/\theta $, because $s_{\textstyle \theta }$ is injective (see Section \ref{preliminaries}); thus ${\rm Con}({\cal A}/\theta )$ is the two--element chain, which is a Boolean algebra, thus ${\cal B}({\rm Con}({\cal A}/\theta ))={\rm Con}({\cal A}/\theta )=\{\theta /\theta ,\nabla _{\cal A}/\theta \}$.\end{itemize}\label{R1}\end{remark}

\begin{remark} ${\cal B}([\theta ))=\{\theta ,\nabla _{\cal A}\}$ iff ${\cal B}({\rm Con}({\cal A}/\theta ))=\{\theta /\theta ,\nabla _{\cal A}/\theta \}$. Indeed, since $s_{\textstyle \theta }$ is a bounded lattice isomorphism (see Section \ref{preliminaries}), it follows that ${\cal B}(s_{\textstyle \theta }):{\cal B}({\rm Con}({\cal A}/\theta ))\rightarrow {\cal B}( [\theta ))$ is a Boolean isomorphism, whose inverse is ${\cal B}(s_{\textstyle \theta }^{-1})$. Therefore ${\cal B}([\theta ))={\cal B}(s_{\textstyle \theta })({\cal B}({\rm Con}({\cal A}/\theta )))=\{s_{\textstyle \theta }(\alpha )\ |\ \alpha \in {\cal B}({\rm Con}({\cal A}/\theta ))\}$ and  ${\cal B}({\rm Con}({\cal A}/\theta ))={\cal B}(s_{\textstyle \theta }^{-1})({\cal B}([\theta )))=\{s_{\textstyle \theta }^{-1}(\alpha )\ |\ \alpha \in {\cal B}([\theta ))\}$, hence the equivalence above.\label{R2}\end{remark}

\begin{lemma}\begin{enumerate}

\item\label{L2(1)} If $\theta \in {\rm Spec}({\cal A})$, then ${\cal B}([\theta ))=\{\theta ,\nabla _{\cal A}\}$ and ${\cal B}({\rm Con}({\cal A}/\theta ))=\{\theta /\theta ,\nabla _{\cal A}/\theta \}$.
\item\label{L2(2)} If $\theta \in {\rm Max}({\cal A})$, then ${\cal B}([\theta ))=\{\theta ,\nabla _{\cal A}\}$ and ${\cal B}({\rm Con}({\cal A}/\theta ))=\{\theta /\theta ,\nabla _{\cal A}/\theta \}$.\end{enumerate}\label{L2}\end{lemma}

\begin{proof} (\ref{L2(1)}) Assume that $\theta \in {\rm Spec}({\cal A})$. Let $\alpha \in {\cal B}([\theta ))$ and denote $\beta=\neg _{\theta }\alpha $. Then $\alpha \cap \beta =\theta $, thus $\alpha \cap \beta \subseteq \theta $, hence $\alpha \subseteq \theta $ or $\beta \subseteq \theta $ since $\theta $ is a prime congruence. But $\alpha ,\beta \in [\theta )$, that is $\theta \subseteq \alpha $ and $\theta \subseteq \beta $. Therefore $\alpha =\theta $ or $\beta =\theta $. If $\beta =\theta $, then $\alpha =\neg _{\theta }\beta =\neg _{\theta }\theta =\nabla _{\cal A}$. Hence ${\cal B}([\theta ))\subseteq \{\theta ,\nabla _{\cal A}\}$. But, clearly, $\{\theta ,\nabla _{\cal A}\subseteq {\cal B}([\theta ))\}$. Therefore ${\cal B}([\theta ))=\{\theta ,\nabla _{\cal A}\}$. Hence ${\cal B}({\rm Con}({\cal A}/\theta ))=\{\theta /\theta ,\nabla _{\cal A}/\theta \}$ by Remark \ref{R2}.

\noindent (\ref{L2(2)}) This is part of Remark \ref{R1}, but also follows from (\ref{L2(1)}) and Lemma \ref{folclor}, (\ref{folclor1}).\end{proof}

\begin{lemma}\begin{enumerate}
\item\label{L1(1)} If ${\cal B}([\theta ))=\{\theta ,\nabla _{\cal A}\}$, then $\theta $ has CBLP.
\item\label{L1(2)} If ${\cal B}({\rm Con}({\cal A}/\theta ))=\{\theta /\theta ,\nabla _{\cal A}/\theta \}$, then $\theta $ has CBLP.\end{enumerate}\label{L1}\end{lemma}

\begin{proof} (\ref{L1(1)}) Assume that ${\cal B}([\theta ))=\{\theta ,\nabla _{\cal A}\}$. Since, clearly, $\Delta _{\cal A},\nabla _{\cal A}\in {\cal B}({\rm Con}({\cal A}))$, it follows that ${\cal B}(v_{\textstyle \theta })({\cal B}({\rm Con}({\cal A})))\supseteq {\cal B}(v_{\textstyle \theta })(\{\Delta _{\cal A},\nabla _{\cal A}\})=\{v_{\textstyle \theta }(\Delta _{\cal A}),v_{\textstyle \theta }(\nabla _{\cal A})\}=\{\Delta _{\cal A}\vee \theta ,\nabla _{\cal A}\vee \theta \}=\{\theta ,\nabla _{\cal A}\}={\cal B}([\theta ))$, thus ${\cal B}(v_{\textstyle \theta })({\cal B}({\rm Con}({\cal A})))={\cal B}([\theta ))$, that is ${\cal B}(v_{\textstyle \theta })$ is surjective, so $\theta $ has CBLP.

\noindent (\ref{L1(2)}) By (\ref{L1(1)}) and Remark \ref{R2}.\end{proof}

\begin{proposition}\begin{enumerate}
\item\label{primecblp1} Any prime congruence of ${\cal A}$ has CBLP.

\item\label{primecblp2} Any maximal congruence of ${\cal A}$ has CBLP.\end{enumerate}\label{primecblp}\end{proposition}

\begin{proof} (\ref{primecblp1}) By Lemma \ref{L2}, (\ref{L2(1)}), and Lemma \ref{L1}.

\noindent (\ref{primecblp2}) By (\ref{primecblp1}) and Lemma \ref{folclor}, (\ref{folclor2}).\end{proof}

\begin{lemma} ${\cal B}({\rm Con}({\cal A}))\subseteq {\cal K}({\cal A})$.\label{boolfg}\end{lemma}

\begin{proof} Let $\phi \in {\cal B}({\rm Con}({\cal A}))$. Then $\displaystyle \phi =Cg(\phi )=\bigvee _{(a,b)\in \phi }Cg(a,b)$, $\phi \cap \neg \, \phi =\Delta _{\cal A}$ and $\phi \vee \neg \, \phi =\nabla _{\cal A}$, so, by the hypothesis (H), there exists a finite set $X\subseteq \phi $ such that $Cg(X)\vee \neg \, \phi =\nabla _{\cal A}$. But then $Cg(X)\subseteq \phi $, thus $Cg(X)\cap \neg \, \phi \subseteq \phi \cap \neg \, \phi =\Delta _{\cal A}$, hence we also have $Cg(X)\cap \neg \, \phi =\Delta _{\cal A}$. Therefore $Cg(X)=\neg \, \neg \, \phi =\phi $, so $\phi =Cg(X)\in {\cal K}({\cal A})$, hence ${\cal B}({\rm Con}({\cal A}))\subseteq {\cal K}({\cal A})$.\end{proof}

\begin{corollary}
$u_{\textstyle \theta }({\cal K}({\cal A}))={\cal K}({\cal A}/\theta )$.\label{corsuper}\end{corollary}

\begin{proof} By Lemma \ref{superlema}, both inclusions hold.\end{proof}

\begin{proposition}\begin{enumerate}
\item\label{prop1(1)} If ${\cal B}({\rm Con}({\cal A}))={\cal K}({\cal A})$ and ${\cal B}({\rm Con}({\cal A}/\theta ))={\cal K}({\cal A}/\theta )$, then $\theta $ has CBLP.

\item\label{prop1(2)}
 If every non--empty algebra ${\cal N}$ from ${\cal C}$ has ${\cal B}({\rm Con}({\cal N}))={\cal K}({\cal N})$, then every non--empty algebra from ${\cal C}$ has CBLP.\end{enumerate}\label{prop1}\end{proposition}

\begin{proof} (\ref{prop1(1)}) If ${\cal A}$ and ${\cal A}/\theta $ are such that their Boolean congruences coincide to their compact congruences, then, by Corollary \ref{corsuper}, ${\cal B}({\rm Con}({\cal A}/\theta ))={\cal K}({\cal A}/\theta )=u_{\textstyle \theta }({\cal K}({\cal A}))=u_{\textstyle \theta }({\cal B}({\rm Con}({\cal A}/\theta )))={\cal B}(u_{\textstyle \theta })({\cal B}({\rm Con}({\cal A}/\theta )))$, thus $\theta $ has CBLP.

\noindent (\ref{prop1(2)}) If every non--empty algebra ${\cal N}$ from ${\cal C}$ has ${\cal B}({\rm Con}({\cal N}))={\cal K}({\cal N})$, then ${\cal B}({\rm Con}({\cal A}))={\cal K}({\cal A})$ and, for each $\phi \in {\rm Con}({\cal A})$, ${\cal B}({\rm Con}({\cal A}/\phi ))={\cal K}({\cal A}/\phi )$, hence, according to (\ref{prop1(1)}), each $\phi \in {\rm Con}({\cal A})$ has CBLP, that is ${\cal A}$ has CBLP. Since ${\cal A}$ is an arbitrary non--empty algebra from ${\cal C}$, it follows that every non--empty algebra from ${\cal C}$ has CBLP.\end{proof}

\begin{corollary}
Any bounded distributive lattice has CBLP.\label{d01cblp}\end{corollary}

\begin{proof} By Proposition \ref{prop1} and the fact that, if ${\cal A}$ is a bounded distributive lattice, then, according to \cite[p. $127$]{blyth}, $\displaystyle {\cal B}({\rm Con}({\cal A}))=\{\bigvee _{i=1}^nCg(a_i,b_i)\ |\ n\in \N ^*,(\forall \, i\in \overline{1,n})\, (a_i,b_i\in A)\}={\cal K}({\cal A})$.\end{proof}

\begin{remark} In bounded non--distributive lattices, the CBLP is neither always present, nor always absent. Indeed, let ${\cal D}$ be the diamond and ${\cal P}$ be the pentagon, with the elements denoted as in the following Hasse diagrams:\vspace*{-15pt}

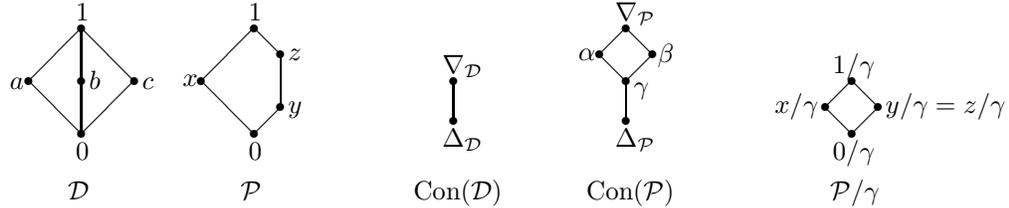
\begin{center}
\begin{tabular}{ccccc}
\begin{picture}(40,100)(0,0)
\put(30,25){\line(0,1){40}}
\put(30,25){\line(1,1){20}}
\put(30,25){\line(-1,1){20}}
\put(30,65){\line(1,-1){20}}
\put(30,65){\line(-1,-1){20}}
\put(30,25){\circle*{3}}
\put(10,45){\circle*{3}}
\put(3,42){$a$}
\put(30,45){\circle*{3}}
\put(33,42){$b$}
\put(50,45){\circle*{3}}
\put(53,42){$c$}
\put(30,65){\circle*{3}}
\put(28,15){$0$}
\put(28,68){$1$}
\put(25,0){${\cal D}$}
\end{picture}
&\hspace*{10pt}
\begin{picture}(40,100)(0,0)
\put(30,25){\line(1,1){10}}
\put(30,25){\line(-1,1){20}}
\put(30,65){\line(1,-1){10}}
\put(30,65){\line(-1,-1){20}}
\put(40,35){\line(0,1){20}}
\put(30,25){\circle*{3}}
\put(10,45){\circle*{3}}
\put(3,43){$x$}
\put(40,35){\circle*{3}}
\put(43,33){$y$}
\put(30,65){\circle*{3}}

\put(40,55){\circle*{3}}
\put(43,53){$z$}
\put(25,0){${\cal P}$}
\put(28,15){$0$}
\put(28,68){$1$}
\end{picture}
&\hspace*{30pt}
\begin{picture}(40,100)(0,0)
\put(20,30){\line(0,1){15}}
\put(20,30){\circle*{3}}
\put(20,45){\circle*{3}}
\put(16,20){$\Delta _{\cal D}$}
\put(16,48){$\nabla _{\cal D}$}
\put(5,0){${\rm Con}({\cal D})$}
\end{picture}
&\hspace*{10pt}
\begin{picture}(40,100)(0,0)
\put(20,45){\line(1,1){10}}
\put(20,45){\line(-1,1){10}}
\put(20,65){\line(1,-1){10}}
\put(20,65){\line(-1,-1){10}}
\put(20,45){\line(0,-1){15}}
\put(16,20){$\Delta _{\cal P}$}
\put(16,68){$\nabla _{\cal P}$}
\put(20,30){\circle*{3}}
\put(23,40){$\gamma $}
\put(2,53){$\alpha $}
\put(32,52){$\beta $}
\put(20,45){\circle*{3}}
\put(10,55){\circle*{3}}
\put(30,55){\circle*{3}}
\put(20,65){\circle*{3}}
\put(5,0){${\rm Con}({\cal P})$}
\end{picture}
&\hspace*{30pt}
\begin{picture}(40,100)(0,0)
\put(20,25){\line(1,1){10}}
\put(20,25){\line(-1,1){10}}
\put(20,45){\line(1,-1){10}}
\put(20,45){\line(-1,-1){10}}
\put(20,25){\circle*{3}}
\put(-9,33){$x/\gamma $}
\put(33,33){$y/\gamma =z/\gamma $}
\put(10,35){\circle*{3}}
\put(30,35){\circle*{3}}
\put(20,45){\circle*{3}}
\put(13,15){$0/\gamma $}
\put(13,48){$1/\gamma $}
\put(12,0){${\cal P}/\gamma $}
\end{picture}
\end{tabular}
\end{center}\vspace*{-5pt}

Let us denote, for any set $M$ and any partition $\pi $ of $M$, by $eq(\pi )$ the equivalence on $M$ which corresponds to $\pi $; also, if $\pi =\{M_1,\ldots ,M_n\}$ for some $n\in \N ^*$, then we shall denote by $eq(M_1,\ldots ,M_n)=eq(\pi )$.

The well--known fact that the classes of a congruence of a lattice $L$ are convex sublattices of $L$ make it easy to prove that ${\rm Con}({\cal D})=\{\Delta _{\cal D},\nabla _{\cal D}\}$, which is isomorphic to the two--element Boolean algebra, ${\cal L}_2$, and ${\rm Con}({\cal P})=\{\Delta _{\cal P},\alpha ,\beta ,\gamma ,\nabla _{\cal P}\}$, where $\alpha =eq(\{0,y,z\},\{x,1\})$, $\beta =eq(\{0,x\},\{y,z,1\})$ and $\gamma =eq(\{0\},\{x\},\{y,z\},\{1\})$, with the lattice structure represented above.

Thus ${\cal B}({\rm Con}({\cal D}))={\rm Con}({\cal D})=\{\Delta _{\cal D},\nabla _{\cal D}\}$, hence ${\cal D}$ has CBLP by Remark \ref{deltanabla}. The lattice structure of ${\rm Con}({\cal P})$ is the one represented above. By Remark \ref{deltanabla}, $\Delta _{\cal P}$ and $\nabla _{\cal P}$ have CBLP. $[\alpha )$ and $[\beta )$ are isomorphic to the standard Boolean algebra, ${\cal L}_2$: $[\alpha )=\{\alpha ,\nabla _{\cal P}\}$ and $[\beta )=\{\beta ,\nabla _{\cal P}\}$, thus ${\cal B}([\alpha ))=[\alpha )=\{\alpha ,\nabla _{\cal P}\}$ and ${\cal B}([\beta ))=[\beta )=\{\beta ,\nabla _{\cal P}\}$, hence $\alpha $ and $\beta $ have CBLP by Lemma \ref{L1}, (\ref{L1(1)}). But ${\cal P}/\gamma $ is isomorphic to the four--element Boolean algebra, ${\cal L}_2^2$, which, being a finite Boolean algebra, is isomorphic to its congruence lattice, so ${\rm Con}({\cal P}/\gamma )$ is isomorphic to ${\cal L}_2^2$, thus ${\cal B}({\rm Con}({\cal P}/\gamma ))={\rm Con}({\cal P}/\gamma )$ is isomorphic to ${\cal L}_2^2$, while the lattice structure of ${\rm Con}({\cal P})$ shows that ${\cal B}({\rm Con}({\cal P}))=\{\Delta _{\cal P},\nabla _{\cal P}\}$, which is isomorphic to ${\cal L}_2$, thus ${\cal B}(u_{\gamma }):{\cal B}({\rm Con}({\cal P}))\rightarrow {\cal B}({\rm Con}({\cal P}/\gamma ))$ can not be surjective, which means that $\gamma $ does not have CBLP. Therefore ${\cal P}$ does not have CBLP.\label{exdsip}\end{remark}

Now let us recall some definitions and results from \cite[Chapter $4$]{bj} and \cite[Chapter IV, Section $9$]{bur} concerning discriminator varieties. The {\em discriminator function} on a set $A$ is the mapping $t:A^3\rightarrow A$ defined by: for all $a,b,c\in A$,$$t(a,b,c)=\begin{cases}a, & \mbox{if }a\neq b,\\ c, & \mbox{if }a=b.\end{cases}$$A {\em discriminator term} on the algebra ${\cal A}$ is a term from the first order language associated to $\tau $ with the property that $t^{\cal A}$ is the discriminator function on $A$. The algebra ${\cal A}$ is called a {\em discriminator algebra} iff there exists a discriminator term on ${\cal A}$. An equational class ${\cal D}$ is called a {\em discriminator equational class} iff it is generated by a class of algebras which have a common discriminator term (equivalently, iff the subdirectly irreducible algebras from ${\cal D}$ have a common discriminator term).

\begin{proposition}{\rm \cite{bj}} Let ${\cal D}$ be a discriminator equational class and ${\cal A}$ be an algebra from ${\cal D}$. Then:\begin{itemize}
\item\label{discriminator1} ${\cal A}$ is an arithmetical algebra;
\item\label{discriminator2} any compact congruence of ${\cal A}$ is principal;
\item\label{discriminator3} any principal congruence of ${\cal A}$ is a factor congruence.\end{itemize}\label{discriminator}\end{proposition}

\begin{corollary}
All non--empty algebras from a discriminator equational class which satisfy (H) have CBLP.\label{cordiscrim}\end{corollary}

\begin{proof} Let ${\cal D}$ be a discriminator equational class and ${\cal A}$ be an algebra from ${\cal D}$ which satisfies (H). Then, by Proposition \ref{discriminator}, ${\cal A}$ is an arithmetical algebra, thus its set of factor congruences coincides to ${\cal B}({\rm Con}({\cal A}))$, hence ${\cal K}({\cal A})\subseteq {\cal B}({\rm Con}({\cal A}))$. By Lemma \ref{boolfg}, the converse inclusion holds, as well. Therefore ${\cal B}({\rm Con}({\cal A}))={\cal K}({\cal A})$, so ${\cal A}$ has CBLP by Proposition \ref{prop1}.\end{proof}

\begin{remark}
Among the discriminator equational classes with all members satisfying (H), there are important classes of algebras of logic such as: Boolean algebras, Post algebras, $n$--valued MV--algebras, monadic algebras, cylindric algebras etc.. Recently, in \cite{lpt}, it has been proven that ${\rm G\ddot{o}del}$ residuated lattices form a discriminator equational class. By Corollary \ref{cordiscrim}, it follows that all the algebras in these classes have CBLP.\label{remdiscrim}\end{remark}

From here until the end of this section, $\theta $ shall no longer be fixed.

For any $\theta \in {\rm Con}({\cal A})$, we shall denote by $V(\theta )=\{\pi \in {\rm Spec}({\cal A})\ |\ \theta \subseteq \pi \}$ and by $D(\theta )={\rm Spec}({\cal A})\setminus V(\theta )$.

\begin{lemma}
Let $\sigma ,\tau \in {\rm Con}({\cal A})$, $I$ be a non--empty set and $(\theta _i)_{i\in I}\subseteq {\rm Con}({\cal A})$. Then:\begin{enumerate}
\item\label{ltop1} $V(\sigma )=\emptyset $ iff $D(\sigma )={\rm Spec}({\cal A})$ iff $\sigma =\nabla _{\cal A}$;
\item\label{ltop2} $V(\sigma )={\rm Spec}({\cal A})$ iff $D(\sigma )=\emptyset $ iff $\sigma =\Delta _{\cal A}$;
\item\label{ltop3} $V(\sigma \cap \tau )=V(\sigma )\cup V(\tau )$ and $D(\sigma \cap \tau )=D(\sigma )\cap D(\tau )$;
\item\label{ltop4} $\displaystyle V(\bigvee _{i\in I}\theta _i)=\bigcap _{i\in I}V(\theta _i)$ and $\displaystyle D(\bigvee _{i\in I}\theta _i)=\bigcup _{i\in I}D(\theta _i)$;
\item\label{ltop5} $V(\sigma )\subseteq V(\tau )$ iff $D(\sigma )\supseteq D(\tau )$ iff $\sigma \supseteq \tau $;
\item\label{ltop6} $V(\sigma )=V(\tau )$ iff $D(\sigma )=D(\tau )$ iff $\sigma =\tau $;
\item\label{ltop7} $\{D(\theta )\ |\ \theta \in {\rm Con}({\cal A})\}$ is a topology on ${\rm Spec}({\cal A})$.\end{enumerate}\label{ltop}\end{lemma}

\begin{proof} (\ref{ltop1}) $D(\sigma )={\rm Spec}({\cal A})$ iff $V(\sigma )=\emptyset $ iff $\sigma =\nabla _{\cal A}$, according to Lemma \ref{folclor}, (\ref{folclor1}) and (\ref{folclor2}).

\noindent (\ref{ltop2}) By Lemma \ref{folclor}, (\ref{folclor3}), $\displaystyle \bigcap _{\pi \in {\rm Spec}({\cal A})}\pi =\Delta _{\cal A}$. $D(\sigma )=\emptyset $ iff $V(\sigma )={\rm Spec}({\cal A})$ iff $\sigma \subseteq \pi $ for all $\pi \in {\rm Spec}({\cal A})$ iff $\displaystyle \sigma \subseteq \bigcap _{\pi \in {\rm Spec}({\cal A})}\pi =\Delta _{\cal A}$ iff $\sigma =\Delta _{\cal A}$.

\noindent (\ref{ltop3}) Every $\pi \in {\rm Spec}({\cal A})$ satisfies: $\pi \in V(\sigma )\cup V(\tau )$ iff $\sigma \subseteq \pi $ or $\tau \subseteq \pi $ iff $\sigma \cap \tau \subseteq \pi $ iff $\pi \in V(\sigma \cap \tau )$. Thus $V(\sigma \cap \tau )=V(\sigma )\cup V(\tau )$, hence $D(\sigma \cap \tau )=D(\sigma )\cap D(\tau )$.

\noindent (\ref{ltop4}) Every $\phi \in {\rm Con}({\cal A})$ satisfies: $\displaystyle \phi \in \bigcap _{i\in I}V(\theta _i)$ iff, for all $i\in I$, $\theta _i\subseteq \phi $ iff $\displaystyle \bigvee _{i\in I}\theta _i\subseteq \phi $ iff $\displaystyle \phi \in V(\bigvee _{i\in I}\theta _i)$. Thus $\displaystyle V(\bigvee _{i\in I}\theta _i)=\bigcap _{i\in I}V(\theta _i)$, hence $\displaystyle D(\bigvee _{i\in I}\theta _i)=\bigcup _{i\in I}D(\theta _i)$.

\noindent (\ref{ltop5}) $D(\sigma )\supseteq D(\tau )$ iff $V(\sigma )\subseteq V(\tau )$. $\sigma \subseteq \tau $ clearly implies $V(\sigma )\subseteq V(\tau )$. $\tau \in V(\tau )$, therefore $V(\sigma )\subseteq V(\tau )$ implies $\tau \in V(\sigma )$, that is $\sigma \subseteq \tau $.

\noindent (\ref{ltop6}) By (\ref{ltop5}).

\noindent (\ref{ltop7}) By (\ref{ltop1}), (\ref{ltop2}), (\ref{ltop3}) and (\ref{ltop4}).\end{proof}

\begin{lemma}
Let $\sigma ,\tau \in {\rm Con}({\cal A})$. Then: $D(\sigma )=V(\tau )$ iff $\sigma ,\tau \in {\cal B}({\rm Con}({\cal A}))$ and $\tau =\neg \, \sigma $.\label{lprep}\end{lemma}

\begin{proof} By Lemma \ref{ltop}, (\ref{ltop1}), (\ref{ltop2}), (\ref{ltop3}), (\ref{ltop4}) and (\ref{ltop6}), the following hold: $D(\sigma )=V(\tau )$ iff $D(\sigma )={\rm Spec}({\cal A})\setminus D(\tau )$ iff $D(\sigma )\cup  D(\tau )={\rm Spec}({\cal A})$ and $D(\sigma )\cap  D(\tau )=\emptyset $ iff $D(\sigma \vee \tau )=D(\nabla _{\cal A})$ and $D(\sigma \cap \tau )=D(\Delta _{\cal A})$ iff $\sigma \vee \tau =\nabla _{\cal A}$ and $\sigma \cap \tau =\Delta _{\cal A}$ iff $\sigma ,\tau \in {\cal B}({\rm Con}({\cal A}))$ and $\tau =\neg \, \sigma $.\end{proof}

\begin{lemma}
The set of the clopen sets of the topological space $({\rm Spec}({\cal A}),\{D(\theta )\ |\ \theta \in {\rm Con}({\cal A})\})$ is $\{V(\alpha )\ |\ \alpha \in {\cal B}({\rm Con}({\cal A})\}))$.\label{lclp}\end{lemma}

\begin{proof} The set of the closed sets of the topological space $({\rm Spec}({\cal A}),\{D(\theta )\ |\ \theta \in {\rm Con}({\cal A})\})$ is $\{V(\theta )\ |\ \theta \in {\rm Con}({\cal A})\})$. Hence a subset $S\subseteq {\rm Spec}({\cal A})$ is clopen in this topological space iff $S=D(\sigma )=V(\tau )$ for some $\sigma ,\tau \in {\rm Con}({\cal A})$, which is equivalent to $\sigma ,\tau \in {\cal B}({\rm Con}({\cal A})$ and $\tau =\neg \, \sigma $ according to Lemma \ref{lprep}. Hence, by Lemma \ref{lprep}, $S$ is clopen iff $S=V(\tau )$ for some $\tau \in {\cal B}({\rm Con}({\cal A})$.\end{proof}

We recall that a topological space $(X,{\cal T})$ is said to be {\em strongly zero--dimensional} iff, for every $U,V\in {\cal T}$ such that $X=U\cup V$, there exist two clopen sets $C$ and $D$ of $(X,{\cal T})$ such that $C\subseteq U$, $D\subseteq V$, $C\cap D=\emptyset $ and $C\cup D=X$.

\begin{note}
The equivalence between statements (\ref{blpbnorm1}) and (\ref{blpbnorm2}) in the next proposition is implied by \cite[Lemma $4$]{banasch} in the particular case when the intersection in ${\rm Con}({\cal A})$ is completely distributive with respect to the join, but, for the sake of completeness, we shall provide a proof for it in our setting.\end{note}

\begin{proposition}
The following are equivalent:

\begin{enumerate}
\item\label{blpbnorm1} ${\cal A}$ has CBLP;
\item\label{blpbnorm2} the lattice ${\rm Con}({\cal A})$ is B--normal;
\item\label{blpbnorm0} for any $n\in \N ^*$ and every $\phi _1,\ldots ,\phi _n\in {\rm Con}({\cal A})$ such that $\phi _1\vee \ldots \vee \phi _n=\nabla _{\cal A}$, there exist $\alpha _1,\ldots ,\alpha _n\in {\cal B}({\rm Con}({\cal A}))$ such that $\alpha _1\cap \ldots \cap \alpha _n=\Delta _{\cal A}$ and $\phi _1\vee \alpha _1=\ldots =\phi _n\vee \alpha _n=\nabla _{\cal A}$; 
\item\label{blpbnorm3} for every $\phi, \psi \in {\cal K}({\cal A})$ such that $\phi \vee \psi =\nabla _{\cal A}$, there exist $\alpha ,\beta \in {\cal B}({\rm Con}({\cal A}))$ such that $\alpha \cap \beta =\Delta _{\cal A}$ and $\phi \vee \alpha =\psi \vee \beta =\nabla _{\cal A}$;
\item\label{blpbnorm5} for any $n\in \N ^*$ and every $\phi _1,\ldots ,\phi _n\in {\cal K}({\cal A})$ such that $\phi _1\vee \ldots \vee \phi _n=\nabla _{\cal A}$, there exist $\alpha _1,\ldots ,\alpha _n\in {\cal B}({\rm Con}({\cal A}))$ such that $\alpha _1\cap \ldots \cap \alpha _n=\Delta _{\cal A}$ and $\phi _1\vee \alpha _1=\ldots =\phi _n\vee \alpha _n=\nabla _{\cal A}$;
\item\label{blpbnorm6} the topological space $({\rm Spec}({\cal A}),\{D(\theta )\ |\ \theta \in {\rm Con}({\cal A})\})$ is strongly zero--dimensional.\end{enumerate}\label{blpbnorm}\end{proposition}

\begin{proof} (\ref{blpbnorm1})$\Rightarrow $(\ref{blpbnorm2}): Let $\phi ,\psi \in {\rm Con}({\cal A})$ such that $\phi \cap \psi =\nabla _{\cal A}$. Denote $v={\cal B}(v_{(\phi \cap \psi )}):{\cal B}({\rm Con}({\cal A}))\rightarrow {\cal B}([\phi \vee \psi ))$. By Remark \ref{usiv}, the Boolean morphism $v$ is surjective. $\varphi ,\psi \in {\cal B}([\phi \vee \psi ))$, because $\psi $ is the complement of $\phi $ in the lattice $[\phi \vee \psi )$, that is $\psi =\neg _{(\phi \vee \psi )}\, \phi $. Thus there exists $\alpha \in {\cal B}({\cal A})$ such that $\alpha \vee (\psi \cap \phi )=v(\alpha )=\phi =\neg _{(\phi \vee \psi )}\, \psi $, hence $\psi =\neg _{(\phi \vee \psi )}\, \phi =\neg _{(\phi \vee \psi )}\, v(\alpha )=v(\neg \, \alpha )=\neg \, \alpha\vee (\phi \cap \psi )$. Therefore $\phi \vee \neg \, \alpha =\alpha \vee (\psi \cap \phi )\vee \neg \, \alpha =(\psi \cap \phi )\vee \alpha \vee \neg \, \alpha =(\psi \cap \phi )\vee \nabla _{\cal A}=\nabla _{\cal A}$, $\psi \vee \alpha =\neg \, \alpha\vee (\phi \cap \psi )\vee \alpha =(\phi \cap \psi )\vee \alpha \vee \neg \, \alpha =(\phi \cap \psi )\vee \nabla _{\cal A}=\nabla _{\cal A}$ and, of course, $\alpha \cap \neg \, \alpha =\Delta _{\cal A}$. So ${\rm Con}({\cal A})$ is B--normal.

\noindent (\ref{blpbnorm2})$\Rightarrow $(\ref{blpbnorm1}): Let $\theta \in {\rm Con}({\cal A})$ and let us denote by $v={\cal B}(v_{\theta }):{\cal B}({\rm Con}({\cal A}))\rightarrow {\cal B}([\theta ))$. Let $\phi \in {\cal B}([\theta ))$, so that there exists $\psi \in {\rm Con}({\cal A})$ such that $\phi \vee \psi =\nabla _{\cal A}$ and $\phi \wedge \psi =\theta $, that is $\psi =\neg _{\theta }\, \phi $. Since $\phi \vee \psi =\nabla _{\cal A}$ and ${\rm Con}({\cal A})$ is B--normal, it follows that there exist $\alpha ,\beta \in {\cal B}({\rm Con}({\cal A}))$ such that $\alpha \cap \beta =\Delta _{\cal A}$ and $\alpha \vee \phi =\beta \vee \psi =\nabla _{\cal A}$. Then $\theta \vee \alpha =(\phi \cap \psi )\vee \alpha =(\phi \vee \alpha )\cap (\psi \vee \alpha )=\nabla _{\cal A}\cap (\psi \vee \alpha )=\psi \vee \alpha $. $\alpha \cap \beta =\Delta _{\cal A}$ in the Boolean algebra ${\cal B}({\rm Con}({\cal A}))$, thus $\beta \leq \neg \, \alpha $, so, since $\beta \vee \psi =\nabla _{\cal A}$, it follows that $\nabla _{\cal A}=\beta \vee \psi \leq \neg \, \alpha \vee \psi $, hence $\neg \, \alpha \vee \psi =\nabla _{\cal A}$. Therefore $\alpha =\alpha \cap (\neg \, \alpha \vee \psi)=(\alpha \cap \neg \, \alpha )\vee (\alpha \cap \psi )=\Delta _{\cal A}\vee (\alpha \cap \psi )=\alpha \cap \psi $, thus $\alpha \leq \psi $, so $\psi \vee \alpha =\psi $. Hence $\theta \vee \alpha =\psi \vee \alpha =\psi $, so that $v(\alpha )=\theta \vee \alpha =\psi $, thus $v(\neg \, \alpha )=\neg _{\theta }\, v(\alpha )=\neg _{\theta }\, \psi =\phi $. Hence $v$ is surjective, that is $\theta $ has CBLP. Therefore ${\cal A}$ has CBLP.

\noindent (\ref{blpbnorm2})$\Rightarrow $(\ref{blpbnorm3}): Trivial.

\noindent (\ref{blpbnorm3})$\Rightarrow $(\ref{blpbnorm2}): Let $\phi, \psi \in {\rm Con}({\cal A})$ such that $\phi \vee \psi =\nabla _{\cal A}$. Since the lattice ${\rm Con}({\cal A})$ is algebraic, it follows that there exist $(\phi _i)_{i\in I}\subseteq {\cal K}({\cal A})$ and $(\psi _j)_{j\in J}\subseteq {\cal K}({\cal A})$ such that $\displaystyle \phi =\bigvee _{i\in I}\phi _i$ and $\displaystyle \psi =\bigvee _{j\in J}\psi _j$, thus $\displaystyle \bigvee _{i\in I}\phi _i\vee \bigvee _{j\in J}\psi _j=\nabla _{\cal A}$. Since $\nabla _{\cal A}$ is compact, it follows that there exist $I_0\subseteq I$ and $J_0\subseteq J$ such that $I_0$ and $J_0$ are finite and $\displaystyle \bigvee _{i\in I_0}\phi _i\vee \bigvee _{j\in J_0}\psi _j=\nabla _{\cal A}$. Let $\displaystyle \phi _0=\bigvee _{i\in I_0}\phi _i$ and $\displaystyle \psi _0=\bigvee _{j\in J_0}\psi _j$. Then $\phi _0\vee \psi _0=\nabla _{\cal A}$ and, clearly, $\phi _0,\psi _0\in {\cal K}({\cal A})$. It follows that there exist $\alpha ,\beta \in {\cal B}({\rm Con}({\cal A}))$ such that $\alpha \cap \beta =\Delta _{\cal A}$ and $\phi _0\vee \alpha =\psi _0\vee \beta =\nabla _{\cal A}$. Since, obviously, $\phi _0\subseteq \phi$ and $\psi _0\subseteq \psi$, we obtain $\phi \vee \alpha =\psi \vee \beta =\nabla _{\cal A}$.

\noindent (\ref{blpbnorm0})$\Rightarrow $(\ref{blpbnorm2}): Trivial: just take $n=2$.

\noindent (\ref{blpbnorm2})$\Rightarrow $(\ref{blpbnorm0}): Assume that the lattice ${\rm Con}({\cal A})$ is B--normal, and let $n\in \N ^*$ and $\phi _1,\ldots ,\phi _n\in {\rm Con}({\cal A})$ such that $\phi _1\vee \ldots \vee \phi _n=\nabla _{\cal A}$. Then, by \cite[Proposition $12$]{blpiasi}, it follows that there exist $\beta _1,\ldots ,\beta _n\in {\cal B}({\rm Con}({\cal A}))$ such that $\beta _1\vee \ldots \vee \beta _n=\nabla _{\cal A}$, $\beta _i\cap \beta _j=\Delta _{\cal A}$ for all $i,j\in \overline{1,n}$ with $i\neq j$ and $\phi _i\supseteq \beta _i$ for all $i\in \overline{1,n}$. For all $i\in \overline{1,n}$, let $\alpha _i=\neg \, \beta _i\in {\cal B}({\rm Con}({\cal A}))$. Then $\alpha _1\cap \ldots \cap \alpha _n=\neg \, (\beta _1\vee \ldots \vee \beta _n)=\neg \, \nabla _{\cal A}=\Delta _{\cal A}$ and, for all $i\in \overline{1,n}$, $\phi _i\vee \alpha _i\geq \beta _i\vee \alpha _i=\beta _i\vee \neg \, \beta _i=\nabla _{\cal A}$, thus $\phi _i\vee \alpha _i=\nabla _{\cal A}$.

\noindent (\ref{blpbnorm5})$\Rightarrow $(\ref{blpbnorm3}): Trivial: just take $n=2$.

\noindent (\ref{blpbnorm3})$\Rightarrow $(\ref{blpbnorm5}): We apply induction on $n\in \N ^*$. For $n=1$, we have $\phi _1=\nabla _{\cal A}$. We may take $\alpha _1=\Delta _{\cal A}\in {\cal B}({\rm Con}({\cal A}))$, and we get $\phi _1\vee \alpha _1=\nabla _{\cal A}\vee \Delta _{\cal A}=\nabla _{\cal A}$.

Now assume that the statement in (\ref{blpbnorm5}) is valid for some $n\in \N ^*$, and let $\phi _1,\ldots ,\phi _{n+1}\in {\cal K}({\cal A})$ such that $\phi _1\vee \ldots \vee \phi _{n+1}=\nabla _{\cal A}$. Then $\phi _1\vee \ldots \vee \phi _n\in {\cal K}({\cal A})$ and $(\phi _1\vee \ldots \vee \phi _n)\vee \phi _{n+1}=\nabla _{\cal A}$, so, by the hypothesis (\ref{blpbnorm3}), it follows that there exist $\alpha ,\beta \in {\cal B}({\rm Con}({\cal A}))$ such that $\alpha \cap \beta =\Delta _{\cal A}$, $\phi _{n+1}\vee \beta =\nabla _{\cal A}$ and $(\phi _1\vee \alpha )\vee \ldots \vee (\phi _n\vee \alpha )=\phi _1\vee \ldots \vee \phi _n\vee \alpha =\nabla _{\cal A}$. Then $\alpha \in {\cal K}({\cal A})$ by Lemma \ref{boolfg}, hence $\phi _1\vee \alpha ,\ldots ,\phi _n\vee \alpha \in {\cal K}({\cal A})$, thus, by the induction hypothesis, it follows that there exist $\gamma _1,\ldots ,\gamma _n\in {\cal B}({\rm Con}({\cal A}))$ such that $\gamma _1\cap \ldots \cap \gamma _n=\Delta _{\cal A}$ and $\phi _1\vee \alpha \vee \gamma _1=\ldots =\phi _n\vee \alpha \vee \gamma _n=\nabla _{\cal A}$. Let $\alpha _{n+1}=\beta \in {\cal B}({\rm Con}({\cal A}))$ and, for all $i\in \overline{1,n}$, $\alpha _i=\alpha \vee \gamma _i\in {\cal B}({\rm Con}({\cal A}))$. Then, for all $i\in \overline{1,n+1}$, $\phi _i\vee \alpha _i=\nabla _{\cal A}$, and $\alpha _1\cap \ldots \cap \alpha _{n+1}=(\alpha \vee \gamma _1)\cap (\alpha \vee \gamma _n)\cap \beta =[\alpha \vee (\gamma _1\cap \ldots \cap \gamma _n)]\cap \beta =(\alpha \vee \Delta _{\cal A})\cap \beta =\alpha \cap \beta =\Delta _{\cal A}$.

\noindent (\ref{blpbnorm2})$\Rightarrow $(\ref{blpbnorm6}): Let $U,V\in \{D(\theta )\ |\ \theta \in {\rm Con}({\cal A})\}$ such that $U\cup V={\rm Spec}({\cal A})$, that is $U=D(\sigma )$ and $V=D(\tau )$ for some $\sigma ,\tau \in {\rm Con}({\cal A})$ such that $D(\nabla _{\cal A})={\rm Spec}({\cal A})=D(\sigma )\cup D(\tau )=D(\sigma \vee \tau )$, which means that $\sigma \vee \tau =\nabla _{\cal A}$, according to Lemma \ref{ltop}, (\ref{ltop1}), (\ref{ltop4}) and (\ref{ltop6}). Since the lattice ${\rm Con}({\cal A})$ is B--normal, it follows that there exist $\alpha ,\beta \in {\cal B}({\rm Con}({\cal A}))$ such that $\alpha \cap \beta =\Delta _{\cal A}$ and $\alpha \vee \tau =\beta \vee \sigma =\nabla _{\cal A}$, hence $\beta \subseteq \neg \, \alpha $ and thus $\neg \, \alpha \vee \sigma =\nabla _{\cal A}$. We have obtained $\alpha \vee \tau =\neg \, \alpha \vee \sigma =\nabla _{\cal A}$, thus $\neg \, \alpha \subseteq \tau $ and $\alpha =\neg \, \neg \, \alpha \subseteq \sigma $, therefore $D(\neg \, \alpha )\subseteq D(\tau )=V$ and $D(\alpha )\subseteq D(\sigma )=U$, by Lemma \ref{ltop}, (\ref{ltop5}). By Lemma \ref{ltop}, (\ref{ltop1}), (\ref{ltop2}), (\ref{ltop3}) and (\ref{ltop4}), $D(\alpha )\cap D(\neg \, \alpha )=D(\alpha \cap \neg \, \alpha )=D(\Delta _{\cal A})=\emptyset $ and $D(\alpha )\cup D(\neg \, \alpha )=D(\alpha \vee \neg \, \alpha )=D(\nabla _{\cal A})={\rm Spec}({\cal A})$. By Lemma \ref{lclp}, $D(\alpha )$ and $D(\neg \, \alpha )$ are clopen sets of the topological space $({\rm Spec}({\cal A}),\{D(\theta )\ |\ \theta \in {\rm Con}({\cal A})\})$. Therefore this topological space is strongly zero--dimensional.

\noindent (\ref{blpbnorm6})$\Rightarrow $(\ref{blpbnorm2}): Let $\sigma ,\tau \in {\rm Con}({\cal A})$ such that $\sigma \vee \tau =\nabla _{\cal A}$. Then $D(\sigma )\cup D(\tau )=D(\sigma \vee \tau )=D(\nabla _{\cal A})={\rm Spec}({\cal A})$, according to Lemma \ref{ltop}, (\ref{ltop4}) and (\ref{ltop1}). By the hypothesis of this implication and Lemma \ref{lclp}, it follows that there exist $\alpha ,\beta \in {\cal B}({\rm Con}({\cal A}))$ such that $D(\alpha )\subseteq D(\sigma )$, $D(\beta )\subseteq D(\tau )$, $D(\alpha )\cap D(\beta )=\emptyset $ and $D(\alpha )\cup D(\beta )={\rm Spec}({\cal A})$. Then $D(\alpha )\cup D(\tau )=D(\beta )\cup D(\sigma )={\rm Spec}({\cal A})$. By Lemma \ref{ltop}, (\ref{ltop1}), (\ref{ltop2}), (\ref{ltop3}), (\ref{ltop4}) and (\ref{ltop6}), it follows that $D(\alpha \cap \beta )=D(\Delta _{\cal A})$ and $D(\alpha \vee \tau )=D(\beta \vee \sigma )=D(\nabla _{\cal A})$, thus $\alpha \cap \beta =\Delta _{\cal A}$ and $\alpha \vee \tau =\beta \vee \sigma =\nabla _{\cal A}$, hence the lattice ${\rm Con}({\cal A})$ is B--normal.\end{proof}

\begin{corollary} ${\cal A}$ has CBLP iff, for all $\theta \in {\rm Con}({\cal A})$, ${\cal A}/\theta $ has CBLP.\label{corolar4.7}\end{corollary}

\begin{proof} Assume that ${\cal A}$ has CBLP, which means that ${\cal A}$ is B--normal, according to Proposition \ref{blpbnorm}. Let $\theta \in {\rm Con}({\cal A})$. Let us prove that the lattice $[\theta )$ is B--normal. So let $\phi ,\psi \in [\theta )$ such that $\phi \vee \psi =\nabla _{\cal A}$. Then, since ${\cal A}$ is B--normal, it follows that there exist $\alpha ,\beta \in {\cal B}({\rm Con}({\cal A}))$ such that $\alpha \cap \beta =\Delta _{\cal A}$ and $\phi \vee \alpha =\psi \vee \beta =\nabla _{\cal A}$. Let us consider the Boolean morphism ${\cal B}(v_{\textstyle \theta }):{\cal B}({\rm Con}({\cal A}))\rightarrow {\cal B}([\theta ))$. We have: ${\cal B}(v_{\textstyle \theta })(\alpha ),{\cal B}(v_{\textstyle \theta })(\beta )\in {\cal B}([\theta ))$, ${\cal B}(v_{\textstyle \theta })(\alpha )\cap {\cal B}(v_{\textstyle \theta })(\beta )={\cal B}(v_{\textstyle \theta })(\alpha \cap \beta )={\cal B}(v_{\textstyle \theta })(\Delta _{\cal A})=\theta $, ${\cal B}(v_{\textstyle \theta })(\phi )\vee {\cal B}(v_{\textstyle \theta })(\alpha )={\cal B}(v_{\textstyle \theta })(\phi \vee \alpha )={\cal B}(v_{\textstyle \theta })(\nabla _{\cal A})=\nabla _{\cal A}$ and ${\cal B}(v_{\textstyle \theta })(\psi )\vee {\cal B}(v_{\textstyle \theta })(\beta )={\cal B}(v_{\textstyle \theta })(\psi \vee \beta )={\cal B}(v_{\textstyle \theta })(\nabla _{\cal A})=\nabla _{\cal A}$, hence the lattice $[\theta )$ is B--normal. Since the lattices ${\rm Con}({\cal A}/\theta )$ and $[\theta )$ are isomorphic, it follows that ${\rm Con}({\cal A}/\theta )$ is B--normal, hence ${\cal A}/\theta $ has CBLP, according to Proposition \ref{blpbnorm}. For the converse implication, just take $\theta =\Delta _{\cal A}$, so that ${\cal A}/\theta ={\cal A}/\Delta _{\cal A}$ is isomorphic to ${\cal A}$.\end{proof}

\begin{corollary}
Let $n\in \N ^*$, ${\cal A}_1,\ldots ,{\cal A}_n$ be algebras and $\displaystyle {\cal A}=\prod _{i=1}^n{\cal A}_i$. Then: ${\cal A}$ has CBLP iff, for all $i\in \overline{1,n}$, ${\cal A}_i$ has CBLP.\label{corolar4.8}\end{corollary}

\begin{proof} By Lemma \ref{totdinbj}, (\ref{totdinbj2}), ${\rm Con}({\cal A})$ is isomorphic to $\displaystyle \prod _{i=1}^n{\rm Con}({\cal A}_i)$. By \cite[Propositions $12$ and $13$]{blpiasi}, $\displaystyle \prod _{i=1}^n{\rm Con}({\cal A}_i)$ is B--normal iff, for all $i\in \overline{1,n}$, ${\rm Con}({\cal A}_i)$ is B--normal. By Proposition \ref{blpbnorm}, we obtain that: ${\cal A}$ has CBLP iff ${\rm Con}({\cal A})$ is B--normal iff, for all $i\in \overline{1,n}$, ${\rm Con}({\cal A}_i)$ is B--normal, iff, for all $i\in \overline{1,n}$, ${\cal A}_i$ has CBLP.\end{proof}

In the following results, we shall designate most lattices by their underlying sets.

\begin{remark} It is well known and straightforward that, if $L$ is a lattice, $S$ is a sublattice of $L$ and $\theta \in {\rm Con}(L)$, then $\theta \cap S^2\in {\rm Con}(S)$.\label{congrsl}\end{remark}

\begin{remark} Let $L$ be a lattice with $1$ and $M$ be a lattice with $0$. We shall denote by $L\dotplus M$ the ordinal sum between $L$ and $M$. Let $c$ be the common element of $L$ and $M$ in the lattice $L\dotplus M$. Using a notation from Remark \ref{exdsip}, for any $\phi \in {\rm Con}(L)$ and any $\psi \in {\rm Con}(M)$, we shall denote by $\phi \dotplus \psi =eq((L/\phi \setminus c/\phi )\cup \{c/\phi \cup c/\psi \}\cup (M/\psi \setminus c/\psi ))$.

\begin{enumerate}
\item\label{ordsum1} Then ${\rm Con}(L\dotplus M)=\{\phi \dotplus \psi \ |\ \phi \in {\rm Con}(L),\psi \in {\rm Con}(M)\}$.

Indeed, it is straightforward that, for any $\phi \in {\rm Con}(L)$ and any $\psi \in {\rm Con}(M)$, we have $\phi \dotplus \psi \in {\rm Con}(L\dotplus M)$, and the fact that $L\dotplus M$ has no other congruences follows from Remark \ref{congrsl}.

\item\label{ordsum2} It is easy to notice that $(L\dotplus M)/(\Delta _L\dotplus \nabla _M)\cong L$ and $(L\dotplus M)/(\nabla _L\dotplus \Delta _M)\cong M$.

\item\label{ordsum3} It is immediate that, for all $\alpha ,\phi \in {\rm Con}(L)$ and all $\beta ,\psi \in {\rm Con}(M)$, $\alpha \dotplus \beta \subseteq \phi \dotplus \phi $ iff $\alpha \subseteq \phi $ and $\beta \subseteq \psi $ iff $\alpha \times \beta \subseteq \phi \times \psi $, which, together with the form of ${\rm Con}(L\dotplus M)$ established above, shows that the mapping $\phi \dotplus \psi \mapsto \phi \times \psi $ is an order isomorphism, and thus a bounded lattice isomorphism between ${\rm Con}(L\dotplus M)$ and ${\rm Con}(L\times M)$, which, in turn, is isomorphic to ${\rm Con}(L)\times {\rm Con}(M)$ by Lemma \ref{totdinbj}, (\ref{totdinbj2}). From this and Proposition \ref{propO}, (\ref{propO1}), we deduce that the Boolean algebras ${\cal B}({\rm Con}(L\dotplus M))$, ${\cal B}({\rm Con}(L\times M))$ and ${\cal B}({\rm Con}(L))\times {\cal B}({\rm Con}(M))$ are isomorphic.\end{enumerate}

In what follows, we shall keep the notations from this remark.\label{ordsum}\end{remark}

\begin{corollary}\begin{enumerate}
\item\label{ordscudsip1} The lattice of congruences of any ordinal sum of finite lattices which are either distributive or isomorphic to the diamond is a Boolean algebra, hence any such ordinal sum has CBLP (regardless of whether it is finite). The lattice of congruences of any direct product of finite lattices which are either distributive or isomorphic to the diamond is a Boolean algebra, hence any such direct product has CBLP (regardless of whether it is finite).
\item\label{ordscudsip2} If $L$ is a lattice with $1$, $M$ is a lattice with $0$, and $L\dotplus M$ has the CBLP, then both $L$ and $M$ have the CBLP. 
\item\label{ordscudsip3} Any ordinal sum of bounded lattices which contains the pentagon does not have the CBLP.\end{enumerate}\label{ordscudsip}\end{corollary}

\begin{proof} (\ref{ordscudsip1}) As we have seen in Remark \ref{exdsip}, ${\rm Con}({\cal D})$ is isomorphic to the two--element Boolean algebra. According to a result in \cite{blyth}, the lattice of congruences of any finite distributive lattice is a Boolean algebra. By Remark \ref{ordsum}, (\ref{ordsum3}), it follows that, if $L$ is an ordinal sum of finite lattices which are either distributive or isomorphic to ${\cal D}$, then ${\rm Con}(L)$ is a Boolean algebra, that is ${\cal B}({\rm Con}(L))={\rm Con}(L)$, hence $L$ has CBLP by Proposition \ref{toatebool}, (\ref{toatebool0}), and the same holds if if $L$ is a direct product of finite lattices which are either distributive or isomorphic to ${\cal D}$.

\noindent (\ref{ordscudsip2}) By Corollary \ref{corolar4.7} and Remark \ref{ordsum}, \ref{ordsum2}.

\noindent (\ref{ordscudsip3}) By (\ref{ordscudsip2}) and Remark \ref{exdsip}, in which we have shown that ${\cal P}$ does not have CBLP.\end{proof}

\begin{remark} If a lattice has CBLP, then its sublattices do not necessarily have CBLP. To illustrate this property, we provide an example of a non--modular lattice with CBLP. Let $E$ be the following bounded non--distributive lattice, in which ${\cal P}$ is embedded; we know from Remark \ref{exdsip} that ${\cal P}$ does not have the CBLP.\vspace*{-15pt}

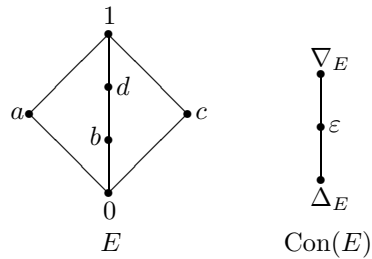
\begin{center}
\begin{tabular}{cc}
\begin{picture}(40,100)(0,0)
\put(40,20){\line(0,1){60}}
\put(40,20){\line(1,1){30}}
\put(40,20){\line(-1,1){30}}
\put(40,80){\line(1,-1){30}}
\put(40,80){\line(-1,-1){30}}
\put(40,20){\circle*{3}}
\put(40,40){\circle*{3}}
\put(40,60){\circle*{3}}
\put(40,80){\circle*{3}}
\put(10,50){\circle*{3}}
\put(70,50){\circle*{3}}
\put(3,48){$a$}
\put(33,38){$b$}
\put(43,57){$d$}
\put(73,48){$c$}
\put(38,10){$0$}
\put(38,84){$1$}
\put(37,-2){$E$}\end{picture}
&\hspace*{45pt}
\begin{picture}(40,100)(0,0)
\put(20,25){\line(0,1){40}}
\put(20,25){\circle*{3}}
\put(20,45){\circle*{3}}
\put(20,65){\circle*{3}}
\put(16,15){$\Delta _E$}
\put(23,43){$\varepsilon $}
\put(16,68){$\nabla _E$}
\put(6,-2){${\rm Con}(E)$}
\end{picture}\end{tabular}\end{center}\vspace*{-5pt}

By using Remark \ref{congrsl} and the calculations in Remark \ref{exdsip}, it is easy to obtain that, if we denote by $\varepsilon =eq(\{0\},\{a\},\{b,d\},\{c\},\{1\})$, then ${\rm Con}(E)=\{\Delta _E,\varepsilon ,\nabla _E\}$, which is isomorphic to the three--element chain (hence ${\cal B}({\rm Con}(E))=\{\Delta _E,\nabla _E\}$ is isomorphic to the two--element Boolean algebra, but we do not even need its form). By Remark \ref{deltanabla}, $\Delta _E$ and $\nabla _E$ have CBLP. $E/\varepsilon $ is isomorphic to ${\cal D}$, hence, by Remark \ref{exdsip}, ${\rm Con}(E/\varepsilon )=\{\Delta _{ E/\varepsilon },\nabla _{E/\varepsilon }\}$, which is isomorphic to the two--element Boolean algebra, thus ${\cal B}({\rm Con}(E/\varepsilon ))={\rm Con}(E/\varepsilon )=\{\Delta _{E/\varepsilon },\nabla _{E/\varepsilon }\}$, hence $\varepsilon $ has CBLP by Lemma  \ref{L2}, (\ref{L2(2)}).\end{remark}

\begin{corollary} If an algebra has CBLP, then its subalgebras do not necessarily have CBLP.\end{corollary}

\begin{remark} Let $\theta \in {\rm Con}({\cal A})$. Then the inequality of cardinalities $|{\cal B}({\rm Con}({\cal A}/\theta ))|\leq |{\cal B}({\rm Con}({\cal A}))|$ does not imply the surjectivity of ${\cal B}(u_{\theta })$, that is it does not imply that $\theta $ has CBLP.

Indeed, let $Z$ be the ordinal sum between ${\cal P}$ and the two--element chain, ${\cal L}_2$, which, as Corollary \ref{ordscudsip}, (\ref{ordscudsip3}), ensures us, does not have CBLP:\vspace*{-8pt}

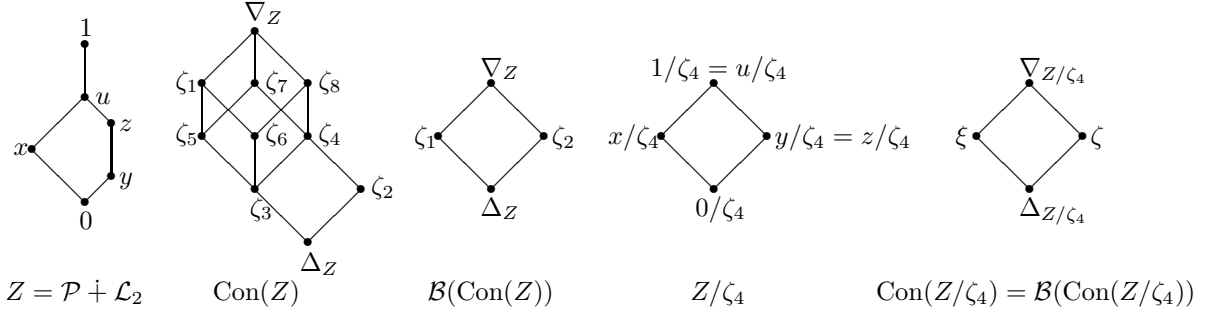
\begin{center}
\begin{tabular}{ccccc}
\begin{picture}(40,100)(0,0)
\put(30,20){\line(1,1){10}}
\put(30,20){\line(-1,1){20}}
\put(30,60){\line(1,-1){10}}
\put(30,60){\line(-1,-1){20}}
\put(40,30){\line(0,1){20}}
\put(30,60){\line(0,1){20}}
\put(30,20){\circle*{3}}
\put(10,40){\circle*{3}}
\put(40,30){\circle*{3}}
\put(30,60){\circle*{3}}
\put(30,80){\circle*{3}}
\put(40,50){\circle*{3}}
\put(3,38){$x$}
\put(43,27){$y$}
\put(43,47){$z$}
\put(28,10){$0$}
\put(34,57){$u$}
\put(28,83){$1$}
\put(0,-17){$Z={\cal P}\dotplus {\cal L}_2$}
\end{picture}
&\hspace*{9pt}
\begin{picture}(60,100)(0,0)
\put(30,25){\line(0,1){20}}
\put(10,45){\line(0,1){20}}
\put(50,45){\line(0,1){20}}
\put(30,65){\line(0,1){20}}
\put(30,25){\line(1,1){20}}
\put(50,5){\line(-1,1){40}}
\put(50,5){\line(1,1){20}}
\put(30,45){\line(1,1){20}}
\put(30,45){\line(-1,1){20}}
\put(70,25){\line(-1,1){40}}
\put(30,65){\line(-1,-1){20}}
\put(30,85){\line(1,-1){20}}
\put(30,85){\line(-1,-1){20}}
\put(30,25){\circle*{3}}
\put(30,45){\circle*{3}}
\put(10,45){\circle*{3}}
\put(50,45){\circle*{3}}
\put(50,5){\circle*{3}}
\put(47,-6){$\Delta _Z$}
\put(14,-17){${\rm Con}(Z)$}
\put(70,25){\circle*{3}}
\put(10,65){\circle*{3}}
\put(50,65){\circle*{3}}
\put(30,65){\circle*{3}}
\put(0,63){$\zeta _1$}
\put(34,63){$\zeta _7$}
\put(54,63){$\zeta _8$}
\put(0,43){$\zeta _5$}
\put(34,43){$\zeta _6$}
\put(54,43){$\zeta _4$}
\put(28,15){$\zeta _3$}
\put(74,23){$\zeta _2$}
\put(30,85){\circle*{3}}
\put(26,89){$\nabla _Z$}
\end{picture}
&\hspace*{14pt}
\begin{picture}(60,100)(0,0)
\put(30,25){\line(1,1){20}}
\put(30,25){\line(-1,1){20}}
\put(30,65){\line(1,-1){20}}
\put(30,65){\line(-1,-1){20}}
\put(30,25){\circle*{3}}
\put(10,45){\circle*{3}}
\put(50,45){\circle*{3}}
\put(30,65){\circle*{3}}
\put(26,15){$\Delta _Z$}
\put(26,68){$\nabla _Z$}
\put(1,42){$\zeta _1$}
\put(53,42){$\zeta _2$}
\put(6,-17){${\cal B}({\rm Con}(Z))$}\end{picture}
&\hspace*{9pt}
\begin{picture}(60,100)(0,0)
\put(30,25){\line(1,1){20}}
\put(30,25){\line(-1,1){20}}
\put(30,65){\line(1,-1){20}}
\put(30,65){\line(-1,-1){20}}
\put(30,25){\circle*{3}}
\put(10,45){\circle*{3}}
\put(50,45){\circle*{3}}
\put(30,65){\circle*{3}}
\put(23,15){$0/\zeta _4$}
\put(6,68){$1/\zeta _4=u/\zeta _4$}
\put(-10,42){$x/\zeta _4$}
\put(53,42){$y/\zeta _4=z/\zeta _4$}
\put(21,-17){$Z/\zeta _4$}\end{picture}
&\hspace*{44pt}
\begin{picture}(60,100)(0,0)
\put(30,25){\line(1,1){20}}
\put(30,25){\line(-1,1){20}}
\put(30,65){\line(1,-1){20}}
\put(30,65){\line(-1,-1){20}}
\put(30,25){\circle*{3}}
\put(10,45){\circle*{3}}
\put(50,45){\circle*{3}}
\put(30,65){\circle*{3}}
\put(26,15){$\Delta _{Z/\zeta _4}$}
\put(26,68){$\nabla _{Z/\zeta _4}$}
\put(2,42){$\xi $}
\put(53,42){$\zeta $}
\put(-28,-17){${\rm Con}(Z/\zeta _4)={\cal B}({\rm Con}(Z/\zeta _4))$}\end{picture}\end{tabular}\end{center}\vspace*{12pt}

The congruences of $Z$ are easy to calculate by using Remark \ref{ordsum}, (\ref{ordsum1}), the calculations in Remark \ref{exdsip} and the fact that the finite Boolean algebra ${\cal L}_2$ is isomorphic to its lattice of congruences: ${\rm Con}({\cal L}_2)=\{\Delta _{{\cal L}_2},\nabla _{{\cal L}_2}\}$: ${\rm Con}(Z)=\{\Delta _Z,\zeta _1,\zeta _2,\zeta _3,\zeta _4,\zeta _5,\zeta _6,\zeta _7,\zeta _8,\nabla _Z \}$, where $\zeta _1=eq(\{0,x,y,z,u\},\{1\})$, $\zeta _2=eq(\{0\},\{x\},\{y\},\{z\},$\linebreak $\{u,1\})$, $\zeta _3=eq(\{0\},\{x\},\{y,z\},\{u\},\{1\})$, $\zeta _4=eq(\{0\},\{x\},\{y,z\},\{u,1\})$, $\zeta _5=eq(\{0,y,z\},\{x,u\},\{1\})$, $\zeta _6=eq(\{0,x\},\{y,z,u\},\{1\})$, $\zeta _7=eq(\{0,y,z\},\{x,u,1\})$ and $\zeta _8=eq(\{0,x\},\{y,z,u,1\})$, with the lattice structure represented above. Therefore ${\cal B}({\rm Con}(Z))=\{\Delta _Z,\zeta _1,\zeta _2,\nabla _Z\}$, which is isomorphic to the four--element Boolean algebra, ${\cal L}_2^2$. Now let us look at the congruence $\zeta _4$. $Z/\zeta _4$ is isomorphic to ${\cal L}_2^2$, hence it is isomorphic to its lattice of congruences: ${\rm Con}(Z/\zeta _4)={\cal B}({\rm Con}(Z/\zeta _4))=\{\Delta _{Z/\zeta _4},\xi ,\zeta ,\nabla _{Z/\zeta _4}\}$, where $\xi =eq(\{0/\zeta _4,x/\zeta _4\},\{y/\zeta _4,1/\zeta _4\})$ and $\zeta =eq(\{0/\zeta _4,y/\zeta _4\},\{x/\zeta _4,1/\zeta _4\})$. Thus ${\cal B}({\rm Con}(Z/\zeta _4))$ is isomorphic to ${\cal B}({\rm Con}(Z))$. Let us calculate ${\cal B}(u_{\zeta _4})$ in each element of ${\cal B}({\rm Con}(Z))$: ${\cal B}(u_{\zeta _4})(\Delta _Z)=u_{\zeta _4}(\Delta _Z)=\Delta _{Z/\zeta _4}$, ${\cal B}(u_{\zeta _4})(\nabla _Z)=u_{\zeta _4}(\nabla _Z)=\nabla _{Z/\zeta _4}$, ${\cal B}(u_{\zeta _4})(\zeta _1)=u_{\zeta _4}(\zeta _1)=(\zeta _1\vee \zeta _4)/\zeta _4=\nabla _Z/\zeta _4=\nabla _{Z/\zeta _4}$ and ${\cal B}(u_{\zeta _4})(\zeta _2)=u_{\zeta _4}(\zeta _2)=(\zeta _2\vee \zeta _4)/\zeta _4=\zeta _4/\zeta _4=\Delta _{Z/\zeta _4}$, hence ${\cal B}(u_{\zeta _4})({\cal B}({\rm Con}(Z)))=\{\Delta _{Z/\zeta _4},\nabla _{Z/\zeta _4}\}\subsetneq {\cal B}({\rm Con}(Z/\zeta _4))$, so ${\cal B}(u_{\zeta _4})$ is not surjective, which means that $\zeta _4$ does not have CBLP.\end{remark}

We say that ${\cal A}$ is {\em local} iff it has exactly one maximal congruence.

\begin{corollary}
If the algebra ${\cal A}$ is local, then the lattice ${\rm Con}({\cal A})$ is ${\rm Id}$--local.\label{corolar2.6}\end{corollary}

\begin{proof} Assume that ${\cal A}$ is local and let $\theta $ be the unique maximal congruence of ${\cal A}$. Let $\alpha ,\beta \in {\rm Con}({\cal A})$ such that $\alpha \vee \beta =\nabla _{\cal A}$. Assume by absurdum that $\alpha \neq \nabla _{\cal A}$ and $\beta \neq \nabla _{\cal A}$. Then, according to Lemma \ref{folclor}, (\ref{folclor1}), it follows that $\alpha \subseteq \theta $ and $\beta \subseteq \theta $, thus $\alpha \vee \beta \subseteq \theta $, which is a contradiction to the choice of $\alpha $ and $\beta $. Hence $\alpha =\nabla _{\cal A}$ or $\beta =\nabla _{\cal A}$.\end{proof}

\begin{lemma}
Any {\rm Id}--local lattice is B--normal.\label{lema2.8}\end{lemma}

\begin{proof} By Lemma \ref{lema2.5}.\end{proof}

\begin{corollary}
Any local algebra has CBLP.\label{corolar4.9}\end{corollary}

\begin{proof} Assume that ${\cal A}$ is a local algebra. Then, by Corollary \ref{corolar2.6}, Lemma \ref{lema2.8} and Proposition \ref{blpbnorm}, it follows that ${\rm Con}({\cal A})$ is {\rm Id}--local, thus it is B--normal, hence ${\cal A}$ has CBLP.\end{proof}

\begin{corollary}
Any finite direct product of local algebras has CBLP.\label{finprodloc}\end{corollary}

We recall that a {\em normal algebra} is an algebra whose lattice of congruences is normal.

\begin{remark}\begin{itemize}
\item Any algebra with CBLP is a normal algebra.
\item Any local algebra is a normal algebra.
\item Any bounded distributive lattice is a normal algebra.\end{itemize}

The first statement follows from Proposition \ref{blpbnorm}, while the second follows from the first and Corollary \ref{corolar4.9}, and the third follows from the first and Corollary \ref{d01cblp}.\label{remarca4.10}\end{remark} 

\begin{lemma}
Let $\theta \in {\rm Con}({\cal A})$. If $\theta \vee {\rm Rad}({\cal A})=\nabla _{\cal A}$, then $\theta =\nabla _{\cal A}$.\label{lema4.11}\end{lemma}

\begin{proof} Let $\theta \in {\rm Con}({\cal A})$ such that $\theta \vee {\rm Rad}({\cal A})=\nabla _{\cal A}$. Since ${\rm Rad}({\cal A})\subseteq \phi $ for all $\phi \in {\rm Max}({\cal A})$, it follows that $\theta \vee \phi =\nabla _{\cal A}$ for all $\phi \in {\rm Max}({\cal A})$. Assume by absurdum that $\theta $ is a proper congruence of ${\cal A}$, so that $\theta \subseteq \phi _0$ for some $\phi _0\in {\rm Max}({\cal A})$, by Lemma \ref{folclor}, (\ref{folclor1}). But then it follows that $\phi _0=\theta \vee \phi _0=\nabla _{\cal A}$, which is a contradiction to the fact that $\phi _0$ is a maximal, and thus a proper congruence of ${\cal A}$. Therefore $\theta =\nabla _{\cal A}$.\end{proof}

\begin{proposition}
If ${\cal A}$ is a normal algebra, then ${\rm Rad}({\cal A})$ has CBLP.\label{propozitie4.12}\end{proposition}

\begin{proof} We shall consider the Boolean morphism ${\cal B}(v_{\textstyle {\rm Rad}({\cal A})}):{\cal B}({\rm Con}({\cal A}))\rightarrow {\cal B}([{\rm Rad}({\cal A})))$. Let $\phi \in {\cal B}([{\rm Rad}({\cal A})))$, so that $\phi \vee \psi =\nabla _{\cal A}$ and $\phi \cap \psi ={\rm Rad}({\cal A})$ for some $\psi \in {\rm Con}({\cal A})$. But the algebra ${\cal A}$ is normal, thus the lattice ${\rm Con}({\cal A})$ is normal, hence there exist $\alpha ,\beta \in {\rm Con}({\cal A})$ such that $\alpha \cap \beta =\Delta _{\cal A}$ and $\phi \vee \alpha =\psi \vee \beta =\nabla _{\cal A}$, thus $\alpha \vee \beta \vee \phi =\alpha \vee \beta \vee \psi =\nabla _{\cal A}$, so $\alpha \vee \beta \vee {\rm Rad}({\cal A})=\alpha \vee \beta \vee (\phi \wedge \psi )=(\alpha \vee \beta \vee \phi )\wedge (\alpha \vee \beta \vee \psi )=\nabla _{\cal A}\wedge \nabla _{\cal A}=\nabla _{\cal A}$, hence $\alpha \vee \beta =\nabla _{\cal A}$, by Lemma \ref{lema4.11}. We have obtained that $\alpha \vee \beta =\nabla _{\cal A}$ and $\alpha \wedge \beta =\Delta _{\cal A}$, thus $\alpha ,\beta \in {\cal B}({\rm Con}({\cal A}))$. $\phi =\phi \vee \Delta _{\cal A}=\phi \vee (\alpha \wedge \beta)=(\phi \vee \alpha )\wedge (\phi \vee \beta )=\nabla _{\cal A}\wedge (\phi \vee \beta)=\phi \vee \beta $. We obtain: ${\cal B}(v_{\textstyle {\rm Rad}({\cal A})})(\beta )=\beta \vee {\rm Rad}({\cal A})=\beta \vee (\phi \cap \psi )=(\beta \vee \phi )\cap (\beta \vee \psi )=\phi \cap \nabla _{\cal A}=\phi $. Therefore ${\cal B}(v_{\textstyle {\rm Rad}({\cal A})})$ is surjective, hence ${\rm Rad}({\cal A})$ has CBLP by Remark \ref{usiv}.\end{proof}

The following property generalizes property $(\star )$ for residuated lattices from \cite{ggcm}, \cite{dcggcm}.

\begin{definition}

We say that ${\cal A}$ satisfies {\em the property $(\star )$} iff: for all $\theta \in {\rm Con}({\cal A})$, there exist $\alpha \in {\cal K}({\cal A})$ and $\beta \in {\cal B}({\rm Con}({\cal A}))$ such that $\alpha \subseteq {\rm Rad}({\cal A})$ and $\theta =\alpha \vee \beta $.\end{definition}

\begin{proposition}
If ${\cal A}$ satisfies $(\star )$, then ${\cal A}$ has CBLP.\label{propozitie4.14}\end{proposition}

\begin{proof} Assume that ${\cal A}$ satisfies $(\star )$, and let $\phi ,\psi \in {\cal K}({\cal A})$ such that $\phi \vee \psi =\nabla _{\cal A}$. Condition $(\star )$ ensures us that $\phi =\alpha \vee \beta $ and $\psi =\gamma \vee \delta $ for some $\alpha ,\gamma \in {\cal K}({\cal A})$ such that $\alpha \subseteq {\rm Rad}({\cal A})$ and $\gamma \subseteq {\rm Rad}({\cal A})$, and some $\beta ,\delta \in {\cal B}({\rm Con}({\cal A}))$. Then $\beta \vee \delta \in {\cal B}({\rm Con}({\cal A}))$ and $\beta \vee \delta \vee {\rm Rad}({\cal A})\supseteq \beta \vee \delta \vee \alpha \vee \gamma =\alpha \vee \beta \vee \gamma \vee \delta =\phi \vee \psi =\nabla _{\cal A}$, so $\beta \vee \delta \vee {\rm Rad}({\cal A})=\nabla _{\cal A}$, hence $\beta \vee \delta =\nabla _{\cal A}$ by Lemma \ref{lema4.11}. Then $\neg \, \beta ,\neg \, \delta \in {\cal B}({\rm Con}({\cal A}))$, $\neg \, \beta \cap \neg \, \delta=\neg \, (\beta \vee \delta )=\neg \, \nabla _{\cal A}=\Delta _{\cal A}$, $\phi \vee \neg \, \beta =\alpha \vee \beta \vee \neg \, \beta =\alpha \vee \nabla _{\cal A}=\nabla _{\cal A}$ and $\psi \vee \neg \, \delta =\gamma \vee \delta \vee \neg \, \delta =\gamma \vee \nabla _{\cal A}=\nabla _{\cal A}$. By Proposition \ref{blpbnorm}, it follows that ${\cal A}$ has CBLP.\end{proof}

\begin{proposition}

${\cal A}$ satisfies $(\star )$ iff, for all $\theta \in {\rm Con}({\cal A})$, ${\cal A}/\theta $ satisfies $(\star )$.\label{starcaturi}\end{proposition}

\begin{proof} Assume that ${\cal A}$ satisfies $(\star )$, and let $\theta \in {\rm Con}({\cal A})$. Let $\phi \in {\rm Con}({\cal A})$ such that $\theta \subseteq \phi $. Then there exist $\alpha \in {\cal K}({\cal A})$ and $\beta \in {\cal B}({\rm Con}({\cal A}))$ such that $\alpha \subseteq {\rm Rad}({\cal A})$ and $\phi =\alpha \vee \beta $. We obtain: $\phi =\phi \vee \theta =\phi \vee \theta \vee \theta $, thus $\phi /\theta =(\phi \vee \theta \vee \theta)/\theta =(\alpha \vee \theta \vee \beta \vee \theta)/\theta =(\alpha \vee \theta )/\theta \vee (\beta \vee \theta )/\theta $. Since $\beta \in {\cal B}({\rm Con}({\cal A}))$, we have $(\beta \vee \theta )/\theta =u_{\textstyle \theta }(\beta )={\cal B}(u_{\textstyle \theta })(\beta )\in {\cal B}({\rm Con}({\cal A}/\theta ))$, while $(\alpha \vee \theta )/\theta \in {\cal K}({\cal A}/\theta )$ by Corollary \ref{corsuper}. Finally, $\displaystyle (\alpha \vee \theta )/\theta \subseteq ({\rm Rad}({\cal A})\vee \theta )/\theta =((\bigcap _{\mu \in {\rm Max}({\cal A})}\mu )\vee \theta )/\theta \subseteq ((\bigcap _{\stackrel{\scriptstyle \mu \in {\rm Max}({\cal A})}{\mu \supseteq \theta }}\mu )\vee \theta )/\theta =(\bigcap _{\stackrel{\scriptstyle \mu \in {\rm Max}({\cal A})}{\mu \supseteq \theta }}\mu )/\theta =\bigcap _{\stackrel{\scriptstyle \mu \in {\rm Max}({\cal A})}{\mu \supseteq \theta }}\mu /\theta ={\rm Rad}({\cal A}/\theta )$. Therefore ${\cal A}/\theta $ satisfies $(\star )$.

For the converse implication, just take $\theta =\Delta _{\cal A}$, so that ${\cal A}/\theta ={\cal A}/\Delta _{\cal A}$ is isomorphic to ${\cal A}$.\end{proof}

\begin{proposition}
Let $n\in \N ^*$ and ${\cal A}_1,\ldots ,{\cal A}_n$ be algebras such that $\displaystyle {\cal A}=\prod _{i=1}^n{\cal A}_i$. Then: ${\cal A}$ satisfies $(\star )$ iff, for all $i\in \overline{1,n}$, ${\cal A}_i$ satisfies $(\star )$.\end{proposition}

\begin{proof} Assume that, for all $i\in \overline{1,n}$, ${\cal A}_i$ satisfies $(\star )$, and let $\theta \in {\rm Con}({\cal A})$. For all $i\in \overline{1,n}$, let $\theta _i=pr_i(\theta )\in {\rm Con}({\cal A}_i)$ by Lemma \ref{totdinbj}. Then, for each $i\in \overline{1,n}$, there exist $\alpha _i\in {\cal K}({\cal A}_i)$ and $\beta _i\in {\cal B}({\rm Con}({\cal A}_i))$ such that $\alpha _i\subseteq {\rm Rad}({\cal A}_i)$ and $\theta _i=\alpha _i\vee \beta _i$. Let $\alpha =\alpha _1\times \ldots \times \alpha _n,\beta =\beta _1\times \ldots \times \beta _n$. By Proposition \ref{propO}, (\ref{propO2}) and (\ref{propO1}), we have $\alpha \in {\cal K}({\cal A})$ and $\beta \in {\cal B}({\rm Con}({\cal A}))$. Also, $\alpha =\alpha _1\times \ldots \times \alpha _n\subseteq {\rm Rad}({\cal A}_1)\times \ldots \times {\rm Rad}({\cal A}_n)={\rm Rad}({\cal A})$ by Lemma \ref{totdinbj} and Proposition \ref{specprod}, (\ref{specprod1}), and $\theta =\theta _1\times \ldots \times \theta _n=(\alpha _1\vee \beta _1)\times \ldots \times (\alpha _n\vee \beta _n)=(\alpha _1\times \ldots \times \alpha _n)\vee (\beta _1\times \ldots \times \beta _n)=\alpha \vee \beta $, by Lemma \ref{totdinbj}, (\ref{totdinbj2}). Therefore ${\cal A}$ satisfies $(\star )$.

Now assume that ${\cal A}$ satisfies $(\star )$, and, for all $i\in \overline{1,n}$, let $\theta _i\in {\rm Con}({\cal A}_i)$. Denote $\theta =\theta _1\times \ldots \times \theta _n\in {\rm Con}({\cal A})$, by Lemma \ref{totdinbj}. Then there exist $\alpha \in {\cal K}({\cal A})$ and $\beta \in {\cal B}({\rm Con}({\cal A}))$ such that $\alpha \subseteq {\rm Rad}({\cal A})$ and $\theta =\alpha \vee \beta $. For each $i\in \overline{1,n}$, let $\alpha _i=pr_i(\alpha )$ and $\beta _i=pr_i(\beta )$. By Lemma \ref{totdinbj}, Proposition \ref{propO}, (\ref{propO2}) and (\ref{propO1}), and Proposition \ref{specprod}, (\ref{specprod1}), it follows that, for all $i\in \overline{1,n}$, $\theta _i=pr_i(\theta )=pr_i(\alpha \vee \beta )=pr_i(\alpha )\vee pr_i(\beta )=\alpha _i\vee \beta _i$, $\alpha _i\in {\cal K}({\cal A}_i)$, $\beta _i\in {\cal B}({\rm Con}({\cal A}_i))$ and $\alpha _i=pr_i(\alpha )\subseteq pr_i({\rm Rad}({\cal A}))={\rm Rad}({\cal A}_i)$. Thus, for all $i\in \overline{1,n}$, ${\cal A}_i$ satisfies $(\star )$.\end{proof}

\section{CBLP Versus BLP in Residuated Lattices and Bounded Distributive Lattices}
\label{cblpversusblp}

In this section, we recall some results on the Boolean Lifting Property (BLP) for residuated lattices and bounded distributive lattices, as well as the reticulation functor between these categories of algebras, and obtain new results, concerning the relationships between CBLP and BLP in these categories, and the behaviour of the reticulation functor with respect to CBLP. From these results it is easy to derive notable properties concerning the image of the reticulation functor.

We refer the reader to \cite{bal}, \cite{gal}, \cite{haj}, \cite{ior}, \cite{kow}, \cite{pic}, \cite{tur} for a further study of the results on residuated lattices that we use in this section. For the results on bounded distributive lattices, we refer the reader to \cite{bal}, \cite{blyth}, \cite{bur}.

Throughout this section, all algebras will be designated by their underlying sets.

We recall that a {\em residuated lattice} is an algebra $(A,\vee ,\wedge ,\odot ,\rightarrow ,0,1)$ of type $(2,2,2,2,0,0)$ such that $(A,\vee ,\wedge ,0,1)$ is a bounded lattice, $(A,\odot ,1)$ is a commutative monoid and every $a,b,c\in A$ satisfy {\em the law of residuation}: $a\leq b\rightarrow c$ iff $a\odot b\leq c$, where $\leq $ is the order of $(A,\vee ,\wedge )$. The operation $\odot $ is called {\em product} or {\em multiplication}, and the operation $\rightarrow $ is called {\em implication} or {\em residuum}. It is well known that residuated lattices form an equational class. A {\em $G\ddot{o}del$ algebra} is a residuated lattice in which $\odot =\wedge $.

Throughout this section, unless mentioned otherwise, $(A,\vee ,\wedge ,\odot ,\rightarrow ,0,1)$ shall be an arbitrary residuated lattice. We recall the definitions of the derivative operations $\neg \, $ (the {\em negation}) and $\leftrightarrow $ (the {\em equivalence} or the {\em biresiduum}) on the elements of $A$: for all $a,b\in A$, $\neg \, a=a\rightarrow 0$ and $a\leftrightarrow b=(a\rightarrow b)\wedge (b\rightarrow a)$. We also recall that, for all $a\in A$ and any $n\in \N $, we denote: $a^0=1$ and $a^{n+1}=a^n\odot a$. Next we shall recall some things about the arithmetic of a residuated lattice, its Boolean center, its filters and congruences, as well as the Boolean Lifting Property in a residuated lattice, and we shall prove several new results regarding these notions.

\begin{lemma}{\rm \cite{bal}, \cite{gal}, \cite{haj}, \cite{ior}, \cite{kow}, \cite{pic}, \cite{tur}} For any $a,b\in A$, the following hold:

\begin{enumerate}
\item\label{aritmlr1} $a\rightarrow b=1$ iff $a\leq b$; $a\leftrightarrow b=1$ iff $a=b$;
\item\label{aritmlr2} $a\odot (a\rightarrow b)\leq b$.
\end{enumerate}\label{aritmlr}\end{lemma}

A {\em filter} of $A$ is a non--empty subset $F$ of $A$ such that, for all $x,y\in A$:

\begin{itemize}
\item if $x,y\in F$, then $x\odot y\in F$;
\item if $x\in F$ and $x\leq y$, then $y\in F$.\end{itemize}

The set of the filters of $A$ is denoted by ${\rm Filt}(A)$. $({\rm Filt}(A),\subseteq )$ is a bounded poset, with first element $\{1\}$ and last element $A$. Clearly, a filter equals $A$ iff it contains $0$.

The intersection of any family of filters of $A$ is a filter of $A$, hence, for any $X\subseteq A$, there exists a smallest filter of $A$ which includes $X$; this filter is denoted by $[X)$ and called the {\em filter generated by $X$}. For every $x\in A$, $[\{x\})$ is denoted, simply, by $[x)$, and called the {\em principal filter generated by $x$}. Clearly, $[\emptyset )=\{1\}=[1)$, while, for any $\emptyset \neq X\subseteq A$, $[X)=\{a\in A\ |\ (\exists \, n\in \N ^*)\, (\exists \, x_1,\ldots ,x_n\in X)\, (x_1\odot \ldots \odot x_n\leq a)\}=\{a\in A\ |\ (\exists \, n\in \N )\, (\exists \, x_1,\ldots ,x_n\in X)\, (x_1\odot \ldots \odot x_n\leq a)\}$, where we make the convention that the product of the empty family is $1$. Thus, for any $x\in A$, $[x)=\{a\in A\ |\ (\exists \, n\in \N ^*)\, (x^n\leq a)\}=\{a\in A\ |\ (\exists \, n\in \N )\, (x^n\leq a)\}$. We denote by ${\rm PFilt}(A)$ the set of the principal filters of $A$.

For every $F,G\in {\rm Filt}(A)$, we denote by $F\vee G=[F\cup G)$. Moreover, for any $(F_i)_{i\in I}\subseteq {\rm Filt}(A)$, we denote by $\displaystyle \bigvee _{i\in I}F_i=[\bigcup _{i\in I}F_i)$. $({\rm Filt}(A),\vee ,\cap ,\{1\},A)$ is a complete bounded distributive lattice, orderred by set inclusion. ${\rm PFilt}(A)$ is a bounded sublattice of ${\rm Filt}(A)$, because $\{1\}=[1)$, $A=[0)$ and, for all $x,y\in A$, $[x)\vee [y)=[x\odot y)$ and $[x)\cap [y)=[x\vee y)$.

To every filter $F$ of $A$, one can associate a congruence $\sim _F$ of $A$, defined by: for all $x,y\in A$, $x\sim _Fy$ iff $x\leftrightarrow y\in F$. Let $F$ be a filter of $A$. The congruence class of any $x\in A$ with respect to $\sim _F$ is denoted by $x/F$, and the quotient set of $A$ with respect to $\sim _F$ is denoted by $A/F$. Residuated lattices form an equational class, thus $A/F$ becomes a residuated lattice, with the operations defined canonically. We shall denote by $p_F:A\rightarrow A/F$ the canonical surjective morphism. Notice that $1/F=F$. For any $x,y\in A$, $x\leq y$ implies $x/F\leq y/F$, and $x/F\leq y/F$ iff $x\rightarrow y\in F$. For any $X\subseteq A$, we denote by $X/F=p_F(X)=\{x/F\ |\ x\in X\}$. We have: ${\rm Filt}(A/F)=\{G/F\ |\ G\in {\rm Filt}(A),F\subseteq G\}$. It is well known and straightforward that the function $h_A:{\rm Filt}(A)\rightarrow {\rm Con}(A)$, for all $F\in {\rm Filt}(A)$, $h_A(F)=\sim _F$, is a bounded lattice isomorphism; hence the bounded lattice ${\rm Con}(A)$ is distributive.

By ${\cal B}(A)$ we denote the set of the complemented elements of the underlying bounded lattice of $A$, which, although not necessarily distributive, is uniquely complemented, and has ${\cal B}(A)$ as a bounded sublattice. Moreover, ${\cal B}(A)$ is a Boolean algebra. ${\cal B}(A)$ is called the {\em Boolean center of $A$}.

If $B$ is a residuated lattice and $f:A\rightarrow B$ is a residuated lattice morphism, then $f({\cal B}(A))\subseteq {\cal B}(B)$, thus, just as in the case of bounded distributive lattices, we can define ${\cal B}(f):{\cal B}(A)\rightarrow {\cal B}(B)$ by: ${\cal B}(f)(x)=f(x)$ for all $x\in {\cal B}(A)$. Then ${\cal B}(f)$ is a Boolean morphism. Hence ${\cal B}$ becomes a covariant functor from the category of residuated lattices to the category of Boolean algebras. We believe that there is no danger of confusion between this functor and the functor ${\cal B}$ from the category of bounded distributive lattices to the category of Boolean algebras.

\begin{proposition}{\rm \cite{kow}} Any residuated lattice is an arithmetical algebra and satisfies (H).\end{proposition}

\begin{definition}{\rm \cite{ggcm}}
For any filter $F$ of $A$, we say that $F$ has the {\em Boolean Lifting Property} (abbreviated {\em BLP}) iff the Boolean morphism ${\cal B}(p_F):{\cal B}(A)\rightarrow {\cal B}(A/F)$ is surjective; also, we say that $\sim _F$ has the {\em Boolean Lifting Property (BLP)} iff $F$ has BLP.

We say that $A$ has the {\em Boolean Lifting Property (BLP)} iff each filter of $A$ has the BLP (equivalently, iff each congruence of $A$ has the BLP).\end{definition}

\begin{remark}{\rm \cite{ggcm}}
For any filter $F$ of $A$, ${\cal B}(A)/F\subseteq {\cal B}(A/F)$ and the image of ${\cal B}(p_F)$ is ${\cal B}(A)/F$, hence: $F$ has BLP iff ${\cal B}(A)/F={\cal B}(A/F)$ iff ${\cal B}(A)/F\supseteq {\cal B}(A/F)$.\end{remark}

\begin{lemma}{\rm \cite{ggcm}}
${\cal B}({\rm Filt}(A))=\{[e)\ |\ e\in {\cal B}(A)\}$.\label{boolfilt}\end{lemma}

Let us define $i_A:A\rightarrow {\rm Filt}(A)$, for all $x\in A$, $i_A(x)=[x)$. Clearly, $i_A$ is an injective bounded lattice anti--morphism between the underlying bounded lattice of $A$ and ${\rm Filt}(A)$. Now let us define ${\cal B}(i_A):{\cal B}(A)\rightarrow {\cal B}({\rm Filt}(A))$, for all $e\in {\cal B}(A)$, ${\cal B}(i_A)(e)=[e)$.

\begin{lemma}
${\cal B}(i_A)$ is well defined and it is a Boolean anti--isomorphism.\label{atentie}\end{lemma}

\begin{proof} By Lemma \ref{boolfilt}, ${\cal B}(i_A)$ is well defined and surjective. Since $i_A$ is injective, it follows that ${\cal B}(i_A)$ is injective. Since $i_A$ is a bounded lattice anti--morphism, ${\cal B}(A)$ is a bounded sublattice of $A$ and ${\cal B}({\rm Filt}(A))$ is a bounded sublattice of ${\rm Filt}(A)$, it follows that ${\cal B}(i_A)$ is a bounded lattice anti--morphism between two Boolean algebras, thus ${\cal B}(i_A)$ is a Boolean anti--morphism. Hence ${\cal B}(i_A)$ is a Boolean anti--isomorphism.\end{proof}

Now let $(F_i)_{i\in I}$ be a non--empty family of filters of $A$ such that $\displaystyle A\subseteq \bigvee _{i\in I}F_i$, that is $\displaystyle A=\bigvee _{i\in I}F_i$, that is $\displaystyle 0\in \bigvee _{i\in I}F_i=[\bigcup _{i\in I}F_i)$, which means that there exist $n\in \N ^*$ and $\displaystyle x_1,\ldots ,x_n\in \bigcup _{i\in I}F_i$ such that $x_1\odot \ldots \odot x_n\leq 0$, that is $x_1\odot \ldots \odot x_n=0$. Then $x_1\in F_{\textstyle i_1},\ldots ,x_n\in F_{\textstyle i_n}$ for some $i_1,\ldots ,i_n\in I$. So $0=x_1\odot \ldots \odot x_n\in [F_{\textstyle i_1}\cup \ldots \cup F_{\textstyle i_n})=F_{\textstyle i_1}\vee \ldots \vee F_{\textstyle i_n}$, thus $A=F_{\textstyle i_1}\vee \ldots \vee F_{\textstyle i_n}$, hence $A\subseteq F_{\textstyle i_1}\vee \ldots \vee F_{\textstyle i_n}$. Therefore $A$ is a compact element of the bounded distributive lattice ${\rm Filt}(A)$. Since ${\rm Filt}(A)$ is isomorphic to ${\rm Con}(A)$, it follows that $\nabla _A$ is a compact element of the bounded distributive lattice ${\rm Con}(A)$, which means that $A$ fulfills the hypothesis (H).

Until mentioned otherwise, $F$ will be a filter of $A$, arbitrary but fixed. We shall denote by $\delta _F:{\rm Filt}(A)\rightarrow {\rm Filt}(A/F)$ the function defined by: for all $G\in {\rm Filt}(A)$, $\delta _F(G)=(G\vee F)/F$.

\begin{lemma}\begin{enumerate}
\item\label{lamunca1} For all $a\in A$, $([a)\vee F)/F=[a/F)$.
\item\label{lamunca2} For all $J,K\in {\rm Filt}(A)$ such that $F\subseteq J$ and $F\subseteq K$, $(J\vee K)/F=J/F\vee K/F$.
\item\label{lamunca3} $\delta _F$ is well defined and it is a bounded lattice morphism.
\item\label{lamunca4} The following diagram is commutative:

\vspace*{-30pt}

\begin{center}\begin{picture}(140,80)(0,0)
\put(30,50){$A$}
\put(24,5){$A/F$}
\put(100,50){${\rm Filt}(A)$}
\put(93,5){${\rm Filt}(A/F)$}
\put(21,30){$p_F$}

\put(33,48){\vector(0,-1){34}}
\put(113,30){$\delta _F$}
\put(110,48){\vector(0,-1){34}}
\put(61,57){$i_A$}
\put(38,53){\vector(1,0){60}}
\put(57,12){$i_{A/F}$}
\put(45,8){\vector(1,0){45}}\end{picture}\end{center}

\vspace*{-20pt}

\item\label{lamunca5} $F$ has BLP iff ${\cal B}(\delta _F)$ is surjective.\end{enumerate}\label{lamunca}\end{lemma}

\begin{proof} (\ref{lamunca1}) Let $a\in A$. Then $[a/F)=\{b/F\ |\ b\in A,(\exists \, n\in \N )((a/F)^n\leq b/F)\}=\{b/F\ |\ b\in A,(\exists \, n\in \N )\, (a^n/F\leq b/F)\}=\{b/F\ |\ b\in A,(\exists \, n\in \N )\, (a^n\rightarrow b\in F)\}$ and $([a)\vee F)/F=[[a)\cup F)/F=[\{a\}\cup F)/F=\{b/F\ |\ b\in [\{a\}\cup F)\}=\{b/F\ |\ (\exists \, n,k\in \N )\, (\exists \, x_1,\ldots ,x_k\in F)\, (a^n\odot x_1\odot \ldots \odot x_n\leq b)\}=\{b/F\ |\ (\exists \, n\in \N )\, (\exists \, x\in F)\, (a^n\odot x\leq b)\}$, since any product of a finite family of elements of $F$ belongs to $F$, and the converse is trivial; we shall be using this property repeatedly in what follows. Let $b\in A$ such that $b/F\in [a/F)$. Then $a^n\rightarrow b\in F$ for some $n\in \N $, thus there exists an $x\in F$ such that $a^n\rightarrow b=x$, so $x\rightarrow a^n\rightarrow b$, hence $a^n\odot x\leq b$ by the law of residuation, therefore $b/F\in ([a)\vee F)/F$. In what follows, we shall be using the law of residuation without mentioning it. Now let $b\in A$ such that $b/F\in ([a)\vee F)/F$. Then there exist $n\in \N $ and $x\in F$ such that $a^n\odot x\leq b$, thus $x\leq a^n\rightarrow b$, hence $a^n\rightarrow b\in F$, therefore $b/F\in [a/F)$. Therefore $([a)\vee F)/F=[a/F)$.

\noindent (\ref{lamunca2}) Let $J$ and $K$ be as in the enunciation. Then $(J\vee K)/F=[J\cup K)/F=\{a/F\ |\ a\in [J\cup K)\}=\{a/F\ |\ a\in A,(\exists \, n,k\in \N )\, (\exists \, x_1,\ldots ,x_n\in J)\, (\exists \, y_1,\ldots ,y_k\in K)\, (x_1\odot \ldots \odot x_n\odot y_1\odot \ldots \odot y_k\leq a)\}=\{a/F\ |\ a\in A,(\exists \, x\in J)\, (\exists \, y\in K)\, (x\odot y\leq a)$. And $J/F\vee K/F=[J/F\cup K/F)=\{a/F\ |\ a\in A,(\exists \, n,k\in \N )\, (\exists \, x_1,\ldots ,x_n\in J)\, (\exists \, y_1,\ldots ,y_k\in K)\, (x_1/F\odot \ldots \odot x_n/F\odot y_1/F\odot \ldots \odot y_k/F\leq a/F)\}=\{a/F\ |\ a\in A,(\exists \, n,k\in \N )\, (\exists \, x_1,\ldots ,x_n\in J)\, (\exists \, y_1,\ldots ,y_k\in K)\, ((x_1\odot \ldots \odot x_n)/F\odot (y_1\odot \ldots \odot y_k)/F\leq a/F)\}=\{a/F\ |\ a\in A,(\exists \, x\in J)\, (\exists \, y\in K)\, (x/F\odot y/F\leq a/F)\}=\{a/F\ |\ a\in A,(\exists \, x\in J)\, (\exists \, y\in K)\, ((x\odot y)/F\leq a/F)\}$. If $x,y,a\in A$ such that $x\odot y\leq a$, then $(x\odot y)/F\leq a/F$, thus $(J\vee K)/F\subseteq J/F\vee K/F$. Now let $a\in A$ such that $a/F\in J/F\vee K/F$, thus there exist $x\in J$ and $y\in K$ such that $(x\odot y)/F\leq a/F$, that is $(x\odot y)\rightarrow a\in F$, so $(x\odot y)\rightarrow a=z$ for some $z\in F$, thus $z\leq (x\odot y)\rightarrow a$, that is $x\odot y\odot z\leq a$. We have: $x\in J$, $y\in K$ and $z\in F\subseteq K$, thus $y\odot z\in K$. So $a/F\in (J\vee K)/F$. Therefore $(J\vee K)/F=J/F\vee K/F$.

\noindent (\ref{lamunca3}) For all $G\in {\rm Filt}(A)$, $G\vee F\supseteq F$, thus $(G\vee F)/F\in {\rm Filt}(A/F)$, so $\delta _F$ is well defined. $\delta _F(\{1\})=(\{1\}\vee F)/F=F/F=\{1/F\}$; $\delta _F(A)=(A\vee F)/F=A/F$. Now let $G,H\in {\rm Filt}(A)$. By (\ref{lamunca2}), we have: $\delta _F(G\vee H)=(G\vee H\vee F)/F=(G\vee F\vee H\vee F)/F=(G\vee F)/F\vee (H\vee F)/F=\delta _F(G)\vee \delta _F(H)$. By the distributivity of the lattice of filters of a residuated lattice, we have: $\delta _F(G\cap H)=((G\cap H)\vee F)/F=((G\vee F)\cap (H\vee F))/F=(G\vee F)/F\cap (H\vee F)/F=\delta _F(G)\cap \delta _F(H)$. Therefore $\delta _F$ is a bounded lattice morphism.

\noindent (\ref{lamunca4}) Let $a\in A$. By (\ref{lamunca1}), $\delta _F(i_A(a))=\delta _F([a))=([a)\vee F)/F=[a/F)=i_{A/F}(a/F)=i_{A/F}(p_F(a))$. Therefore $\delta _F\circ i_A=i_{A/F}\circ p_F$.

\noindent (\ref{lamunca5}) By taking the restrictions to the Boolean centers in the commutative diagram in (\ref{lamunca4}), we get the following commutative diagram, where we have denoted by ${\cal B}(i_A)$ the restriction of $i_A$ to ${\cal B}(A)$, and the same goes for ${\cal B}(i_{A/F})$:

\vspace*{-23pt}

\begin{center}\begin{picture}(140,80)(0,0)
\put(22,50){${\cal B}(A)$}

\put(16,5){${\cal B}(A/F)$}
\put(100,50){${\cal B}({\rm Filt}(A))$}
\put(93,5){${\cal B}({\rm Filt}(A/F))$}
\put(7,30){${\cal B}(p_F)$}
\put(33,48){\vector(0,-1){34}}
\put(123,30){${\cal B}(\delta _F)$}
\put(121,48){\vector(0,-1){34}}
\put(58,57){${\cal B}(i_A)$}
\put(44,53){\vector(1,0){54}}
\put(53,12){${\cal B}(i_{A/F})$}
\put(51,8){\vector(1,0){39}}\end{picture}\end{center}

\vspace*{-3pt}

Thus ${\cal B}(\delta _F)\circ {\cal B}(i_A)={\cal B}(i_{A/F})\circ {\cal B}(p_F)$. By Lemma \ref{atentie}, ${\cal B}(i_A)$ and ${\cal B}(i_{A/F})$ are Boolean anti--isomorphisms. Hence: $F$ has BLP iff ${\cal B}(p_F)$ is surjective iff ${\cal B}(\delta _F)$ is surjective.\end{proof}

Now let us consider the congruence $h_A(F)=\sim _F$ associated to $F$ and the bounded lattice morphism $u_{\textstyle \sim _F}:{\rm Con}(A)\rightarrow {\rm Con}(A)/\sim _F$, for all $\theta \in {\rm Con}(A)$, $u_{\textstyle \sim _F}(\theta )=(\theta \vee \sim _F)/\sim _F$  (see Section \ref{thecblp}).

\begin{lemma}

The following diagram is commutative:\vspace*{-24pt}

\begin{center}\begin{picture}(140,80)(0,0)
\put(22,50){${\rm Filt}(A)$}
\put(16,5){${\rm Filt}(A/F)$}
\put(103,50){${\rm Con}(A)$}
\put(96,5){${\rm Con}(A/F)$}
\put(26,30){$\delta _F$}
\put(37,48){\vector(0,-1){34}}
\put(123,30){$u_{\textstyle \sim _F}$}
\put(121,48){\vector(0,-1){34}}
\put(70,57){$h_A$}
\put(53,53){\vector(1,0){48}}
\put(65,12){$h_{A/F}$}
\put(60,8){\vector(1,0){34}}\end{picture}\end{center}\vspace*{-10pt}\label{capace}\end{lemma}

\begin{proof} Let $G\in {\rm Filt}(A)$. Then $h_{A/F}(\delta _F(G))=h_{A/F}((G\vee F)/F)=\sim _{(G\vee F)/F}=\{(x/F,y/F)\ |\ x,y\in A,x/F\leftrightarrow y/F\in (G\vee F)/F\}=\{(x/F,y/F)\ |\ x,y\in A,(x\leftrightarrow y)/F\in (G\vee F)/F\}$ and $u_{\textstyle \sim _F}(h_A(G))=u_{\textstyle \sim _F}(\sim _G)=(\sim _G\vee \sim _F)/\sim _F=(h_A(G)\vee h_A(F))/\sim _F=h_A(G\vee F)/\sim _F=\sim _{G\vee F}/\sim _F=\{(x/\sim _F,y/\sim _F)\ |\ x,y\in A, (x,y)\in /\sim _{G\vee F}\}=\{(x/F,y/F)\ |\ x,y\in A, x\leftrightarrow y\in G\vee F\}$. For any $x,y\in A$, if $x\leftrightarrow y\in G\vee F$, then $(x\leftrightarrow y)/F\in (G\vee F)/F$; conversely, if $(x\leftrightarrow y)/F\in (G\vee F)/F$, then $(x\leftrightarrow y)/F=z/F$ for some $z\in G\vee F$, thus $(x\leftrightarrow y)\leftrightarrow z\in F$, so, since $(x\leftrightarrow y)\leftrightarrow z\leq z\rightarrow (x\leftrightarrow y)$, it follows that $z\rightarrow (x\leftrightarrow y)\in F$, that is $z\rightarrow (x\leftrightarrow y)=t$ for some $t\in F$, hence $t\leq z\rightarrow (x\leftrightarrow y)$, thus $t\odot z\leq x\leftrightarrow y$, with $t\in F\subseteq G\vee F$ and $z\in G\vee F$, hence $t\odot z\in G\vee F$, thus $x\leftrightarrow y\in G\vee F$. Hence $h_{A/F}(\delta _F(G))=u_{\textstyle \sim _F}(h_A(G))$. Therefore $h_{A/F}\circ \delta _F=u_{\textstyle \sim _F}\circ h_A$.\end{proof}

\begin{proposition}\begin{enumerate}
\item\label{happy1} For every filter $F$ of $A$: $F$ has BLP iff $\sim _F$ has CBLP.
\item\label{happy2} $A$ has BLP iff $A$ has CBLP.\end{enumerate}\label{happy}\end{proposition}

\begin{proof} (\ref{happy1}) By applying the functor ${\cal B}$ from the category of bounded distributive lattices to the category of Boolean algebras to the commutative diagram in Lemma \ref{capace}, we get the following commutative diagram in the category of Boolean algebras:\vspace*{-23pt}

\begin{center}\begin{picture}(140,80)(0,0)
\put(22,50){${\cal B}({\rm Filt}(A))$}
\put(13,5){${\cal B}({\rm Filt}(A/F))$}
\put(111,50){${\cal B}({\rm Con}(A))$}
\put(106,5){${\cal B}({\rm Con}(A/F))$}
\put(18,30){${\cal B}(\delta _F)$}
\put(43,48){\vector(0,-1){34}}
\put(136,30){${\cal B}(u_{\textstyle \sim _F})$}
\put(134,48){\vector(0,-1){34}}
\put(75,57){${\cal B}(h_A)$}
\put(67,53){\vector(1,0){43}}
\put(70,12){${\cal B}(h_{A/F})$}
\put(71,8){\vector(1,0){33}}\end{picture}\end{center}

\vspace*{-4pt}

This means that: ${\cal B}(h_{A/F})\circ {\cal B}(\delta _F)={\cal B}(u_{\textstyle \sim _F})\circ {\cal B}(h_A)$. $h_A$ and $h_{A/F}$ are bounded lattice isomorphisms, hence ${\cal B}(h_A)$ and ${\cal B}(h_{A/F})$ are Boolean isomorphisms. By Lemma \ref{lamunca}, (\ref{lamunca5}), we get that: $F$ has BLP iff ${\cal B}(\delta _F)$ is surjective iff ${\cal B}(u_{\textstyle \sim _F})$ is surjective iff $\sim _F$ has CBLP.

\noindent (\ref{happy2}) By (\ref{happy1}) and the fact that $h_A:{\rm Filt}(A)\rightarrow {\rm Con}(A)$, for all $F\in {\rm Filt}(A)$, $h_A(F)=\sim _F$, is a bijection.\end{proof}

\begin{definition}{\rm \cite{dcggcm}} $A$ is a {\em Gelfand residuated lattice} iff any prime filter of $A$ is included in a unique maximal filter of $A$.\end{definition}

\begin{proposition}{\rm \cite{dcggcm}} $A$ is Gelfand iff the lattice ${\rm Filt}(A)$ is normal.\label{lrgelfand}\end{proposition}

\begin{corollary}

$A$ is Gelfand iff the lattice ${\rm Con}(A)$ is normal.\label{gelfand}\end{corollary}

\begin{proof} By Proposition \ref{lrgelfand} and the fact that the bounded distributive lattices ${\rm Filt}(A)$ and ${\rm Con}(A)$ are isomorphic.\end{proof}

The following corollary is part of \cite[Theorem $6.20$]{dcggcm}, but here we provide a different proof for it, by using the equivalence between CBLP and BLP in residuated lattices.

\begin{corollary}
Any residuated lattice with BLP is Gelfand.\label{cordinp2}\end{corollary}

\begin{proof} By Proposition \ref{happy}, (\ref{happy2}), Proposition \ref{blpbnorm}, Corollary \ref{gelfand} and the trivial fact that any B--normal lattice is normal.\end{proof}

An element of $a\in A$ is said to be {\em idempotent} iff $a^2=a$. The set of the idempotents of $A$ is denoted by ${\cal I}(A)$. An element of $a\in A$ is said to be {\em regular} iff $\neg \, \neg \, a=a$. The set of the regular elements of $A$ is denoted by ${\rm Reg}(A)$.

\begin{definition}{\rm \cite{dcggcm}} Let $F$ be an arbitrary filter of $A$. We say that $F$ has the {\em Idempotent Lifting Property} (abbreviated {\em ILP}) iff ${\cal I}(A/F)={\cal I}(A)/F$.

We say that $A$ has the {\em Idempotent Lifting Property (ILP)} iff all of its filters have the ILP.\end{definition}

\begin{proposition}{\cite{dcggcm}}
Neither of the properties BLP and ILP in residuated lattices implies the other.\end{proposition}

\begin{proposition}{\rm \cite{dcggcm}} For any filter $F$ of $A$, ${\rm Reg}(A/F)={\rm Reg}(A)/F$.\label{regtriv}\end{proposition}

MV--algebras form a subclass of the class of BL--algebras, which, in turn, form a subclass of the class of residuated lattices. If $A$ is a BL--algebra, then so is $A/F$ for any $F\in {\rm Filt}(A)$; the same goes for MV--algebras.

\begin{proposition}{\rm \cite{dcggcm}} Any BL--algebra is a Gelfand residuated lattice.\label{blgelfand}\end{proposition}

\begin{proposition}{\rm \cite{bal}, \cite{gal}, \cite{haj}, \cite{ior}, \cite{kow}, \cite{pic}, \cite{tur}} If $A$ is an MV--algebra, then ${\cal B}(A)={\cal I}(A)$.\label{bvsi}\end{proposition}

\begin{corollary}
If $A$ is an MV--algebra, then:\begin{itemize}
\item for any filter $F$ of $A$: $\sim _F$ has CBLP iff $F$ has BLP iff $F$ has ILP;
\item $A$ has CBLP iff $A$ has BLP iff $A$ has ILP.\end{itemize}\end{corollary}

\begin{proof} By Propositions \ref{happy} and \ref{bvsi}.\end{proof}

The equivalences not involving CBLP in the previous corollary were proven in \cite{dcggcm} by using Proposition \ref{bvsi}. Now let us investigate them for BL--algebras. Of course, by Proposition \ref{happy}, a BL--algebra has BLP iff it has CBLP, and, furthermore, any filter of it has BLP iff its associated congruence has CBLP.

\begin{proposition}{\rm \cite{bal}, \cite{gal}, \cite{haj}, \cite{ior}, \cite{kow}, \cite{pic}, \cite{tur}} If $A$ is a BL--algebra, then ${\cal B}(A)={\cal I}(A)\cap {\rm Reg}(A)$.\label{bvsir}\end{proposition}

\begin{corollary}
If $A$ is a BL--algebra, then:\begin{enumerate}
\item\label{lpsbl1} for any filter $F$ of $A$: if $F$ has ILP, then $F$ has BLP;
\item\label{lpsbl2} if $A$ has ILP, then $A$ has BLP.\end{enumerate}\label{lpsbl}\end{corollary}

\begin{proof} (\ref{lpsbl1}) If $F$ has ILP, then ${\cal I}(A/F)={\cal I}(A)/F$, so, by Propositions \ref{bvsir} and \ref{regtriv}, ${\cal B}(A/F)={\cal I}(A/F)\cap {\rm Reg}(A/F)={\cal I}(A)/F\cap {\rm Reg}(A)/F={\cal B}(A)/F$, thus $F$ has BLP.

\noindent (\ref{lpsbl2}) By (\ref{lpsbl1}).\end{proof}

Throughout the rest of this section, unless mentioned otherwise, $(L,\vee ,\wedge ,0,1)$ shall be an arbitrary bounded distributive lattice. We shall denote the set of the filters of $L$ by ${\rm Filt}(L)$, and the set of the ideals of $L$ by ${\rm Id}(L)$. To each filter $F$ of $L$, one can associate a congruence $\equiv _F$ of $L$, defined by: for any $x,y\in L$, $x\equiv _Fy$ iff $x\wedge a=y\wedge a$ for some $a\in F$; the mapping $F\mapsto \equiv _F$ is an embedding of the bounded distributive lattice ${\rm Filt}(L)$ into ${\rm Con}(L)$. For every $F\in {\rm Filt}(L)$, any $x\in L$ and any $X\subseteq L$, we shall denote by $x/F=x/\equiv _F$ and $X/F=X/\equiv _F$. Dually, to each ideal $I$ of $L$, one can associate a congruence $\approx _I$ of $L$, defined by: for any $x,y\in L$, $x\approx _Iy$ iff $x\vee a=y\vee a$ for some $a\in I$; the mapping $I\mapsto \approx _I$ is a bounded lattice embedding of ${\rm Id}(L)$ into ${\rm Con}(L)$.

\begin{definition}{\rm \cite{blpiasi},\cite{blpdacs},\cite{dcggcm}} We say that a congruence $\sim $ of $L$ has the {\em Boolean Lifting Property (BLP)} iff ${\cal B}(L/\sim )={\cal B}(L)/\sim $. We say that $L$ has the {\em Boolean Lifting Property (BLP)} iff all of its congruences have the BLP.

We say that a filter $F$ of $L$ has the {\em Boolean Lifting Property (BLP)} iff $\equiv _F$ has the BLP, that is iff ${\cal B}(L/F)={\cal B}(L)/F$. We say that $L$ has the {\em Boolean Lifting Property for filters ({\rm Filt}--BLP)} iff all of its filters have the BLP.

Similarly, we say that an ideal $I$ of $L$ has the {\em Boolean Lifting Property (BLP)} iff $\approx _I$ has the BLP, and we say that $L$ has the {\em Boolean Lifting Property for ideals ({\rm Id}--BLP)} iff all of its ideals have the BLP.\end{definition}

Clearly, in any bounded distributive lattice $L$, the BLP implies the Filt--BLP and Id--BLP, and the BLP is self--dual, while the Filt--BLP and Id--BLP are duals of each other.

See in \cite{eu3}, \cite{eu1}, \cite{eu}, \cite{eu2}, \cite{eu4}, \cite{eu5}, \cite{dcggcm} the definition of {\em the reticulation functor ${\cal L}$} from the category of residuated lattices to the category of bounded distributive lattices, which takes every residuated lattice $A$ to the unique (up to a bounded lattice isomorphism) bounded distributive lattice ${\cal L}(A)$ whose prime spectrum is homeomorphic to that of $A$, where the prime spectra are the sets of the prime filters of ${\cal L}(A)$, respectively $A$, endowed with the Stone topologies. ${\cal L}(A)$ is called the {\em reticulation of $A$}. The bounded distributive lattice ${\cal L}(A)$ is isomorphic to the dual of ${\rm PFilt}(A)$ (\cite{eu1}, \cite{eu}). 

\begin{proposition}{\rm \cite[Proposition $5.19$]{dcggcm}} $A$ has BLP iff ${\cal L}(A)$ has Filt--BLP.\label{propQ}\end{proposition}

\begin{remark} In bounded distributive lattices, CBLP always holds, as proven in Corollary \ref{d01cblp}. But the next remark contains an example of a bounded distributive lattice without Filt--BLP, thus without BLP. See, in what follows, whole classes of bounded distributive lattices without BLP.\end{remark}

\begin{remark} Trivially, the functor ${\cal L}$ preserves the CBLP, by Corollary \ref{d01cblp}. But ${\cal L}$ does not reflect the CBLP. Indeed, let us consider the following example of residuated lattice from \cite{ior}: $A=\{0,a,b,c,d,1\}$, with the following Hasse diagram, with $\odot =\wedge $ and $\rightarrow $ defined by the following table:\vspace*{-5pt}

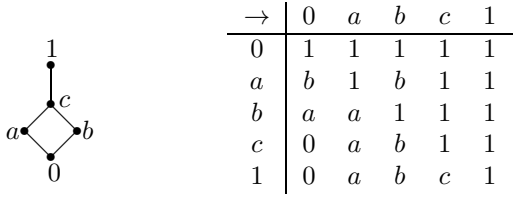
\begin{center}\begin{tabular}{cc}
\begin{picture}(50,70)(0,0)
\put(25,10){\line(1,1){10}}
\put(25,10){\line(-1,1){10}}
\put(25,30){\line(1,-1){10}}
\put(25,30){\line(-1,-1){10}}
\put(25,30){\line(0,1){15}}
\put(25,30){\circle*{3}}
\put(25,10){\circle*{3}}
\put(25,45){\circle*{3}}
\put(15,20){\circle*{3}}
\put(35,20){\circle*{3}}
\put(24,1){$0$}
\put(23,48){$1$}
\put(8,17){$a$}
\put(37,17){$b$}
\put(28,29){$c$}
\end{picture} & \begin{picture}(120,70)(0,0)
\put(30,30){\begin{tabular}{c|ccccc}
$\rightarrow $ & $0$ & $a$ & $b$ & $c$ & $1$ \\ \hline
$0$ & $1$ & $1$ & $1$ & $1$ & $1$ \\
$a$ & $b$ & $1$ & $b$ & $1$ & $1$ \\
$b$ & $a$ & $a$ & $1$ & $1$ & $1$ \\
$c$ & $0$ & $a$ & $b$ & $1$ & $1$ \\
$1$ & $0$ & $a$ & $b$ & $c$ & $1$\end{tabular}}\end{picture}\end{tabular}\end{center}

According to \cite[Example $3.5$]{ggcm}, the residuated lattice $A$ does not have BLP, because its filter $[c)$ does not have the BLP, hence, by Proposition \ref{happy}, (\ref{happy2}), $A$ does not have CBLP. But, according to Corollary \ref{d01cblp}, ${\cal L}(A)$ has CBLP, as does every bounded distributive lattice.

Concerning the structure of ${\cal L}(A)$, we may notice that ${\cal L}(A)$ has the same Hasse diagram as $A$, because $\odot =\wedge $ in $A$ ($A$ is a ${\rm G\ddot{o}del}$ algebra) and hence, according to a result in \cite{eu2}, ${\cal L}(A)$ is isomorphic to the underlying bounded lattice of $A$. In \cite[Example $1$]{blpiasi}, we have proven that this bounded distributive lattice does not have Filt--BLP, since its filter $[c)$ does not have BLP. Of course, this and Proposition \ref{propQ} provide another proof for the fact that $A$ does not have the BLP, but this issue is, actually, trivial here, because the underlying bounded lattice of a ${\rm G\ddot{o}del}$ algebra is distributive, since, in any residuated lattice, $\odot $ is distributive with respect to $\vee $, and, obviously, the filters of a ${\rm G\ddot{o}del}$ algebra coincide with the filters of its bounded lattice reduct, and so do the congruences associated to these filters, hence the BLP in a ${\rm G\ddot{o}del}$ algebra coincides with the Filt--BLP in its bounded lattice reduct.\label{rcuex}\end{remark}

Here is an extended version of one of the results recalled above:

\begin{proposition}{\rm \cite{dcggcm}} The following are equivalent:\begin{enumerate}
\item\label{rezulttare0} $A$ is a Gelfand residuated lattice;
\item\label{rezulttare1} ${\rm Filt}(A)$ is a normal lattice;
\item\label{rezulttare2} ${\rm PFilt}(A)$ is a normal lattice;
\item\label{rezulttare3} ${\cal L}(A)$ is a conormal lattice;
\item\label{rezulttare4} any prime filter of $A$ is included in a unique maximal filter of $A$;
\item\label{rezulttare5} any prime filter of ${\cal L}(A)$ is included in a unique maximal filter of ${\cal L}(A)$.\end{enumerate}\label{rezulttare}\end{proposition}

\begin{remark} In \cite{blpiasi}, by noticing that, if $L$ is a bounded distributive lattice which is not local and in which $\{1\}$ is a prime filter, then $L$ does not satisfy condition (\ref{rezulttare5}) from Proposition \ref{rezulttare}, and thus $L$ is not conormal, we have pointed out that, for instance, an ordinal sum between a bounded distributive lattice which is not local (for example, a direct product of at least two non--trivial chains) and a non--trivial chain is not a conormal bounded distributive lattice. Such a lattice is ${\cal L}(A)$ from Remark \ref{rcuex}, which is the ordinal sum between ${\cal L}_2^2$ and ${\cal L}_2$, where ${\cal L}_2$ is the two--element chain.\label{tarecool}\end{remark}

\begin{remark} Concerning the class of the B--normal lattices, notice that it includes all congruence lattices of bounded distributive lattices, according to Corollary \ref{d01cblp} and Proposition \ref{blpbnorm}, and it also includes all congruence lattices of algebras satisfying (H) from any discriminator equational class, according to Corollary \ref{cordiscrim} and Proposition \ref{blpbnorm} (see examples of classes of such algebras in Remark \ref{remdiscrim}).\label{totcool}\end{remark}

\begin{corollary}\begin{enumerate}
\item\label{sohappy1} The image of the class of Gelfand residuated lattices through the reticulation functor is included in the class of conormal bounded distributive lattices, thus it is not the whole class of the bounded distributive lattices. 
\item\label{sohappy0} The image of the class of the residuated lattices with BLP through the reticulation functor is included in the class of conormal bounded distributive lattices, thus it is not the whole class of the bounded distributive lattices. 
\item\label{sohappy4} The image of the class of the residuated lattices with CBLP through the reticulation functor is included in the class of conormal bounded distributive lattices, thus it is not the whole class of the bounded distributive lattices. 
\item\label{sohappy2} The image of the class of BL--algebras through the reticulation functor is included in the class of conormal bounded distributive lattices, thus it is not the whole class of the bounded distributive lattices.
\item\label{sohappy3} The image of the class of MV--algebras through the reticulation functor is included in the class of conormal bounded distributive lattices, thus it is not the whole class of the bounded distributive lattices.\end{enumerate}\label{sohappy}\end{corollary}

\begin{proof} (\ref{sohappy1}) By Proposition \ref{rezulttare} and Remark \ref{tarecool}, which actually provides quite a productive method for obtaining bounded distributive lattices which are outside of the image through the reticulation functor of the class of Gelfand residuated lattices, and thus of any of the classes mentioned in (\ref{sohappy0}), (\ref{sohappy4}), (\ref{sohappy2}), (\ref{sohappy3}) (see just below). 

\noindent (\ref{sohappy0}) By (\ref{sohappy1}) and Corollary \ref{cordinp2}.

\noindent (\ref{sohappy4}) By (\ref{sohappy0}) and Proposition \ref{happy}, which shows that the class of the residuated lattices with BLP coincides to the class of the residuated lattices with CBLP.

\noindent (\ref{sohappy2}) By (\ref{sohappy1}) and Proposition \ref{blgelfand}.

\noindent (\ref{sohappy3}) By (\ref{sohappy2}) and the fact that the class of MV--algebras is included in the class of BL--algebras.\end{proof}

Now let us see if the statements in Proposition \ref{rezulttare} which refer to normality or conormality remain equivalent if we replace these properties by B--normality and B--conormality, respectively.

\begin{corollary}
${\cal B}({\rm PFilt}(A))={\cal B}({\rm Filt}(A))$.\label{corboolfilt}\end{corollary}

\begin{proof} Since ${\rm PFilt}(A)$ is a bounded sublattice of the bounded distributive lattice ${\rm Filt}(A)$, it follows that\linebreak ${\cal B}({\rm PFilt}(A))\subseteq {\cal B}({\rm Filt}(A))$. But, according to Lemma \ref{boolfilt}, ${\cal B}({\rm Filt}(A))=\{[e)\ |\ e\in {\cal B}(A)\}\subseteq {\rm PFilt}(A)$, hence ${\cal B}({\rm Filt}(A))=\{[e)\ |\ e\in {\cal B}(A)\}\subseteq {\cal B}({\rm PFilt}(A))$ since ${\cal B}({\rm Filt}(A))$ is a Boolean algebra. Therefore ${\cal B}({\rm PFilt}(A))={\cal B}({\rm Filt}(A))$.\end{proof}

\begin{proposition} The following are equivalent:\begin{enumerate}
\item\label{blpretic0} $A$ has CBLP;
\item\label{blpretic1} $A$ has BLP;
\item\label{blpretic2} ${\rm Con}(A)$ is a B--normal lattice;
\item\label{blpretic3} ${\rm Filt}(A)$ is a B--normal lattice;
\item\label{blpretic4} ${\rm PFilt}(A)$ is a B--normal lattice;
\item\label{blpretic5} ${\cal L}(A)$ is a B--conormal lattice;
\item\label{blpretic6} ${\cal L}(A)$ has Filt--BLP;
\item\label{blpretic7} ${\rm PFilt}(A)$ has Id--BLP.\end{enumerate}\label{blpretic}\end{proposition}

\begin{proof} (\ref{blpretic0})$\Leftrightarrow $(\ref{blpretic1}): By Proposition \ref{happy}, (\ref{happy2}).

\noindent (\ref{blpretic0})$\Leftrightarrow $(\ref{blpretic2}): By Proposition \ref{blpbnorm}.

\noindent (\ref{blpretic3})$\Rightarrow $(\ref{blpretic4}): By Corollary \ref{corboolfilt} and the fact that ${\rm PFilt}(A)$ is a bounded sublattice of ${\rm Filt}(A)$.

\noindent (\ref{blpretic4})$\Rightarrow $(\ref{blpretic3}): Let $F,G\in {\rm Filt}(A)$ such that $F\vee G=A$. Then $x\odot y=0$ for some $x\in F$ and $y\in G$. So $[x)\vee [y)=[x\odot y)=[0)=A$, hence, according to Corollary \ref{corboolfilt}, there exist $H,K\in {\cal B}({\rm PFilt}(A))={\cal B}({\rm Filt}(A))$ such that $H\cap K=\{1\}$ and $[x)\vee H=[y)\vee K=A$. But $x\in F$ and $y\in G$, thus $[x)\subseteq F$ and $[y)\subseteq G$, hence $A=[x)\vee H\subseteq F\vee H$ and $A=[y)\vee K\subseteq G\vee K$, thus $F\vee H=G\vee K=A$. Therefore ${\rm Filt}(A)$ is B--normal.

\noindent (\ref{blpretic1})$\Leftrightarrow $(\ref{blpretic6}): By Proposition \ref{propQ}.

\noindent (\ref{blpretic4})$\Leftrightarrow $(\ref{blpretic5}) and (\ref{blpretic6})$\Leftrightarrow $(\ref{blpretic7}): By the fact that ${\cal L}(A)$ is isomorphic to the dual of ${\rm PFilt}(A)$.\end{proof}

\begin{corollary} Any bounded distributive lattice which is not B--conormal does not belong to the image of the class of the residuated lattices with BLP (equivalently, with CBLP) through the reticulation functor. Consequently, any bounded distributive lattice which is not conormal does not belong to the image of the class of the residuated lattices with BLP (equivalently, with CBLP) through the reticulation functor.\end{corollary}

\begin{remark} The previous corollary shows that, for instance, the bounded distributive lattices constructed as in Remark \ref{tarecool} do not belong to the image of the class of residuated lattices with BLP (equivalently, with CBLP) through the reticulation functor.\end{remark}

\begin{corollary}
The following are equivalent:

\begin{itemize}
\item the lattice ${\rm PFilt}(A)$ is normal and it is not B--normal; 
\item the lattice ${\rm PFilt}(A)$ is normal and it does not have Id--BLP; 
\item the lattice ${\cal L}(A)$ is conormal and it is not B--conormal;
\item the lattice ${\cal L}(A)$ is conormal and it does not have Filt--BLP;
\item $A$ is Gelfand and it does not have BLP;
\item $A$ is Gelfand and it does not have CBLP.\end{itemize}\end{corollary}

\begin{proof} By Propositions \ref{rezulttare} and \ref{blpretic}.\end{proof}

\begin{example} Let $A$ be the residuated lattice in Remark \ref{rcuex}, which does not have BLP, and whose underlying bounded lattice is isomorphic to ${\cal L}(A)$. The prime filters of ${\cal L}(A)$ are $[a)$ and $[b)$, which coincide to its maximal filters, hence $A$ is Gelfand by Proposition \ref{rezulttare}. So $A$ is a Gelfand residuated lattice without BLP (equivalently, without CBLP). 

Furthermore, since $A$ has $\odot =\wedge $, $A$ is a ${\rm G\ddot{o}del}$ algebra, so, by \cite[Corollary $4.5$]{dcggcm}, it follows that $A$ has ILP, hence, by Corollary \ref{lpsbl}, (\ref{lpsbl2}), we get that $A$ is not a BL--algebra. So this is an example of a Gelfand residuated lattice which is not a BL--algebra, that is a counter--example for the converse of Proposition \ref{blgelfand}.\end{example}

\section{CBLP in Semilocal Algebras}
\label{semilocal}

In this section, we study the CBLP in semilocal arithmetical algebras.

Throughout this section, we shall assume that all the algebras from ${\cal C}$ are arithmetical, and that the algebra ${\cal A}$ is non--trivial. Consequently, ${\rm Max}({\cal A})$ is non--empty. 

We say that ${\cal A}$ is {\em semilocal} iff ${\rm Max}({\cal A})$ is finite.

The results in this section generalize the results on semilocal residuated lattices from \cite[Section $6$]{ggcm}.

\begin{proposition}
The following are equivalent:

\begin{enumerate}
\item\label{propozitie5.2(0)} ${\cal A}$ is semilocal and satisfies $(\star )$;
\item\label{propozitie5.2(1)} ${\cal A}$ is semilocal and has CBLP;
\item\label{propozitie5.2(2)} ${\cal A}$ is semilocal and ${\rm Rad}({\cal A})$ has CBLP;
\item\label{propozitie5.2(3)} there exist $n\in \N ^*$ and $\alpha _1,\ldots ,\alpha _n\in {\cal B}({\rm Con}({\cal A}))$ such that $\displaystyle \bigcap _{i=1}^n\alpha _i=\Delta _A$, $\alpha _i\vee \alpha _j=\nabla _A$ for all $i,j\in \overline{1,n}$ with $i\neq j$ and ${\cal A}/\alpha _i$ is local for all $i\in \overline{1,n}$;
\item\label{propozitie5.2(4)} ${\cal A}$ is isomorphic to a finite direct product of local algebras.\end{enumerate}\label{propozitie5.2}\end{proposition}

\begin{proof} (\ref{propozitie5.2(0)})$\Rightarrow $(\ref{propozitie5.2(1)}): By Proposition \ref{propozitie4.14}.

\noindent (\ref{propozitie5.2(1)})$\Rightarrow $(\ref{propozitie5.2(2)}): Trivial.

\noindent (\ref{propozitie5.2(2)})$\Rightarrow $(\ref{propozitie5.2(3)}): Let $n\in \N ^*$ be the cardinality of ${\rm Max}({\cal A})$ and ${\rm Max}({\cal A})=\{\phi _1,\ldots ,\phi _n\}$. Then, according to \cite[Lemma 2]{afgg}, ${\cal A}/{\rm Rad}({\cal A})$ is isomorphic to $\displaystyle \prod _{i=1}^n{\cal A}/\phi _i$, hence, by Lemma \ref{dinbj}, it follows that there exist $\theta _1,\ldots ,\theta _n\in {\rm Con}({\cal A})$ such that ${\rm Rad}({\cal A})\subseteq \theta _i$ for all $i\in \overline{1,n}$ and the following hold:

\begin{flushleft}\begin{tabular}{cl}
$(a)$ & ${\rm Rad}({\cal A})\subseteq \theta _i$ for all $i\in \overline{1,n}$;\\ 

$(b)$ & $\theta _i/{\rm Rad}({\cal A})\in {\cal B}({\rm Con}({\cal A}/{\rm Rad}({\cal A})))$ for all $i\in \overline{1,n}$;\\ 
$(c)$ & $\displaystyle \bigcap _{i=1}^n\theta _i/{\rm Rad}({\cal A})=\Delta _{\textstyle {\cal A}/{\rm Rad}({\cal A})}$;\\ 
$(d)$ & $\theta _i/{\rm Rad}({\cal A})\vee \theta _j/{\rm Rad}({\cal A})=\nabla _{\textstyle {\cal A}/{\rm Rad}({\cal A})}$ for all $i,j\in \overline{1,n}$ such that $i\neq j$;\\ 
$(e)$ & for all $i\in \overline{1,n}$, ${\cal A}/\phi _i$ is isomorphic to $({\cal A}/{\rm Rad}({\cal A}))/_{\textstyle (\theta _i/{\rm Rad}({\cal A}))}$, which in turn is isomorphic to ${\cal A}/\theta _i$ by\\
& the Second Isomorphism Theorem, hence ${\cal A}/\phi _i$ is isomorphic to ${\cal A}/\theta _i$.\end{tabular}\end{flushleft}

Since ${\rm Rad}({\cal A})\subseteq \theta _i$ for all $i\in \overline{1,n}$, it follows that $\displaystyle {\rm Rad}({\cal A})\subseteq \bigcap _{i=1}^n\theta _i$. From $(c)$ we obtain: $\displaystyle (\bigcap _{i=1}^n\theta _i)/{\rm Rad}({\cal A})=\Delta _{\textstyle {\cal A}/{\rm Rad}({\cal A})}$, thus, for all $a,b\in A$ such that $\displaystyle (a,b)\in \bigcap _{i=1}^n\theta _i$, it follows that $(a/{\rm Rad}({\cal A}),b/{\rm Rad}({\cal A}))\in \Delta _{\textstyle {\cal A}/{\rm Rad}({\cal A})}$, that is $a/{\rm Rad}({\cal A})=b/{\rm Rad}({\cal A})$, which means that $(a,b)\in {\rm Rad}({\cal A})$; hence $\displaystyle \bigcap _{i=1}^n\theta _i\subseteq {\rm Rad}({\cal A})$. Therefore:

\begin{flushleft}\begin{tabular}{cl}
$(f)$ & $\displaystyle \bigcap _{i=1}^n\theta _i={\rm Rad}({\cal A})$.\end{tabular}\end{flushleft}

From $(d)$ we get that: for all $i,j\in \overline{1,n}$ such that $i\neq j$, $(\theta _i\vee \theta _j)/{\rm Rad}({\cal A})=\nabla _{\textstyle {\cal A}/{\rm Rad}({\cal A})}=\nabla _{\cal A}/{\rm Rad}({\cal A})$, so $s_{\textstyle {\rm Rad}({\cal A})}(\theta _i\vee \theta _j)=s_{\textstyle {\rm Rad}({\cal A})}(\nabla _{\cal A})$, hence, by the injectivity of $s_{\textstyle {\rm Rad}({\cal A})}$:

\begin{flushleft}\begin{tabular}{cl}
$(g)$ & $\theta _i\vee \theta _j=\nabla _{\cal A}$ for all $i,j\in \overline{1,n}$ such that $i\neq j$.\end{tabular}\end{flushleft}

From $(b)$ and the fact that ${\rm Rad}({\cal A})$ has CBLP, it follows that there exist $\alpha _1,\ldots ,\alpha _n\in {\cal B}({\rm Con}({\cal A}))$ such that, for all $i\in \overline{1,n}$, $\theta _i/{\rm Rad}({\cal A})=(\alpha _i\vee {\rm Rad}({\cal A}))/{\rm Rad}({\cal A})$, that is $s_{\textstyle {\rm Rad}({\cal A})}(\theta _i)=s_{\textstyle {\rm Rad}({\cal A})}(\alpha _i\vee {\rm Rad}({\cal A}))$, thus $\theta _i=\alpha _i\vee {\rm Rad}({\cal A})$ by the injectivity of $s_{\textstyle {\rm Rad}({\cal A})}$. We obtain that $\displaystyle (\bigcap _{i=1}^n\alpha _i)\vee {\rm Rad}({\cal A})=\bigcap _{i=1}^n(\alpha _i\vee {\rm Rad}({\cal A}))=\bigcap _{i=1}^n\theta _i={\rm Rad}({\cal A})$ according to $(f)$, thus $\displaystyle \bigcap _{i=1}^n\alpha _i\subseteq {\rm Rad}({\cal A})$. But $\displaystyle \bigcap _{i=1}^n\alpha _i\in {\cal B}({\rm Con}({\cal A}))$, hence $\displaystyle \bigcap _{i=1}^n\alpha _i=\Delta_{\cal A}$ by Lemma \ref{camcalalr}. From $(g)$ we obtain: for all $i,j\in \overline{1,n}$ such that $i\neq j$, $\alpha _i\vee \alpha _j\vee {\rm Rad}({\cal A})=\theta _i\vee \theta _j={\rm Rad}({\cal A})$, hence $\alpha _i\vee \alpha _j=\nabla _{\cal A}$ by Lemma \ref{lema4.11}.

For every $i\in \overline{1,n}$, since $\phi _i\in {\rm Max}({\cal A})$, it follows that $\phi _i\neq \nabla _{\cal A}$, thus ${\cal A}/\phi _i$ is non--trivial, hence, by $(e)$, ${\cal A}/\theta _i$ is non--trivial, thus $\theta _i\neq \nabla _{\cal A}$; but $\theta _i=\alpha _i\vee {\rm Rad}({\cal A})$, so $\alpha _i\leq \theta _i<\nabla _{\cal A}$, thus $\alpha _i\neq \nabla _{\cal A}$, so ${\cal A}/\alpha _i$ is non--trivial, thus $\Delta _{\cal A}$ is a proper congruence of ${\cal A}$, hence ${\rm Max}({\cal A}/\alpha _i)$ is non--empty by Lemma \ref{folclor}, (\ref{folclor1}). According to Proposition \ref{specprod}, (\ref{specprod1}), $\displaystyle \sum _{i=1}^n|{\rm Max}({\cal A}/\alpha _i)|=|{\rm Max}(\prod _{i=1}^n{\cal A}/\alpha _i)|=|{\rm Max}({\cal A}/{\rm Rad}({\cal A}))|=n$. It follows that, for all $i\in \overline{1,n}$, $|{\rm Max}({\cal A}/\alpha _i)|=1$, that is ${\cal A}/\alpha _i$ is local.

\noindent (\ref{propozitie5.2(3)})$\Leftrightarrow $(\ref{propozitie5.2(4)}): By Lemma \ref{dinbj}.

\noindent (\ref{propozitie5.2(4)})$\Rightarrow $(\ref{propozitie5.2(1)}): By Corollaries \ref{corolar4.8} and \ref{corolar4.9}.

\noindent (\ref{propozitie5.2(4)})$\Rightarrow $(\ref{propozitie5.2(0)}): By Corollary \ref{finprodloc}.\end{proof}

\begin{corollary}
If ${\cal A}$ is semilocal, then: ${\cal A}$ satisfies $(\star )$ iff ${\cal A}$ has CBLP iff ${\rm Rad}({\cal A})$ has CBLP.\label{corolar5.2,5}\end{corollary}

\begin{corollary}
If ${\cal A}$ is finite, then: ${\cal A}$ satisfies $(\star )$ iff ${\cal A}$ has CBLP iff ${\rm Rad}({\cal A})$ has CBLP.\label{corolar5.3}\end{corollary}

\begin{openproblem}
Find an arithmetical algebra fulfilling (H) with CBLP that does not satisfy $(\star )$.\end{openproblem}

\begin{remark}
Concerning the open problem above, note that, according to Corollary \ref{corolar5.2,5}, an arithmetical algebra fulfilling (H) with CBLP that does not satisfy $(\star )$ is not semilocal. Finding an algebra which is not semilocal is easy: for instance, according to Proposition \ref{specprodarb}, (\ref{specprodarb1}), a direct product of an infinite family of non--trivial algebras is not semilocal. In some particular cases, such as that of ${\rm G\ddot{o}del}$ algebras, CBLP is preserved by arbitrary direct products, according to Proposition \ref{happy} and \cite[Proposition $5.2$]{ggcm}; unfortunately, in this particular case, so is $(\star )$. The open problem above may prove difficult.\end{remark}

\begin{corollary}
If ${\rm Rad}({\cal A})$ has CBLP, then: ${\cal A}$ is semilocal iff it is isomorphic to a finite direct product of local algebras.\label{corolar5.4}\end{corollary}

\begin{corollary}
If ${\cal A}$ is normal, then: ${\cal A}$ is semilocal iff it is isomorphic to a finite direct product of local algebras.\label{corolar5.5}\end{corollary}

\begin{proof} By Corollary \ref{corolar5.4} and Proposition \ref{propozitie4.12}.\end{proof}

\begin{definition}{\rm \cite{afgg}} The algebra ${\cal A}$ is said to be {\em maximal} iff, given any index set $I$, any family $(a_i)_{i\in I}\subseteq A$ and any family $(\theta _i)_{i\in i}\subseteq {\rm Con}({\cal A})$ with the property that, for any finite subset $J$ of $I$, there exists an $x_J\in A$ such that $(x_J,a_i)\in \theta _i$ for all $i\in J$, it follows that there exists an $x\in A$ such that $(x_i,a_i)\in \theta _i$ for all $i\in I$.\end{definition}

\begin{lemma}{\rm \cite{afgg}} Any maximal algebra is semilocal.\label{lema5.8}\end{lemma}

\begin{lemma}{\rm \cite{afgg}} Let $n\in \N ^*$, ${\cal A}_1,\ldots ,{\cal A}_n$ be algebras and $\displaystyle {\cal A}=\prod _{i=1}^n{\cal A}_i$. Then: ${\cal A}$ is maximal iff ${\cal A}_i$ is maximal for every $i\in \overline{1,n}$.\label{lema5.9}\end{lemma}

\begin{proposition} The following are equivalent:

\begin{enumerate}
\item\label{propozitie5.10(0)} ${\cal A}$ is maximal and satisfies $(\star )$;
\item\label{propozitie5.10(1)} ${\cal A}$ is maximal and has CBLP;
\item\label{propozitie5.10(2)} ${\cal A}$ is maximal and ${\rm Rad}({\cal A})$ has CBLP;
\item\label{propozitie5.10(3)} ${\cal A}$ is isomorphic to a finite direct product of maximal local algebras.\end{enumerate}\label{propozitie5.10}\end{proposition}

\begin{proof} (\ref{propozitie5.10(0)})$\Leftrightarrow $(\ref{propozitie5.10(1)}): By Proposition \ref{propozitie4.14}.

\noindent (\ref{propozitie5.10(1)})$\Leftrightarrow $(\ref{propozitie5.10(2)}): By Lemma \ref{lema5.8} and Proposition \ref{propozitie5.2}.

\noindent (\ref{propozitie5.10(1)})$\Rightarrow $(\ref{propozitie5.10(3)}): By Lemma \ref{lema5.8}, Proposition \ref{propozitie5.2} and Lemma \ref{lema5.9}.

\noindent (\ref{propozitie5.10(3)})$\Rightarrow $(\ref{propozitie5.10(1)}): Assume that ${\cal A}$ is isomorphic to a finite direct product of maximal local algebras. Then ${\cal A}$ is maximal by Lemma \ref{lema5.9}, and ${\cal A}$ has CBLP by Corollaries \ref{corolar4.8} and \ref{corolar4.9}.

\noindent (\ref{propozitie5.10(3)})$\Rightarrow $(\ref{propozitie5.10(0)}): By Lemma \ref{lema5.9} and Corollary \ref{finprodloc}.\end{proof}

\begin{corollary} If ${\rm Rad}({\cal A})$ has CBLP, then: ${\cal A}$ is maximal iff it is isomorphic to a finite direct product of maximal local algebras.\label{corolar5.11}\end{corollary}

\begin{corollary} If ${\cal A}$ is normal, then: ${\cal A}$ is maximal iff it is isomorphic to a finite direct product of maximal local algebras.\label{corolar5.12}\end{corollary}

\begin{proof} By Proposition \ref{propozitie4.12} and Corollary \ref{corolar5.11}.\end{proof}

\end{document}